\documentclass[final,hidelinks,onefignum,onetabnum]{siamart220329}

\usepackage{enumitem}
\usepackage{subfig}
\usepackage{amsmath}
\usepackage{bm}
\usepackage{tikz}
\usepackage{appendix}
\usepackage{amsfonts}
\usepackage{multirow}
\usepackage{algorithm,algorithmicx,algpseudocode}
\usepackage{extarrows} 

\algdef{SE}[DOWHILE]{Do}{doWhile}{\algorithmicdo}[1]{\algorithmicwhile\ #1}%

\usepackage[left=2.39cm,right=5.81cm,top=3.1cm,bottom=3.10cm,asymmetric]{geometry}

\setlist[description]{leftmargin=1.5em,   
	labelindent=0pt,    
	labelsep=0.4em}     

\setlist[enumerate]{leftmargin=1.5em}

\setlist[itemize]{leftmargin=1.5em}

\usepackage[marginal,hang]{footmisc}

\usepackage{booktabs}

\newsiamremark{remark}{Remark}
\newsiamthm{Def}{Definition}
\newsiamthm{property}{Property}
\newsiamthm{Prop}{Proposition}
\newsiamthm{claim}{Claim}
\newsiamthm{exmp}{Example}

\ifpdf
\hypersetup{
	pdftitle={},
	pdfauthor={}
}
\fi

\begin{document}

\headers{}{}

\title{{
	Geometric-Perturbation-Robust Cut-Cell Scheme for Two-Material Flows: Exact Pressure-Equilibrium Preservation and Rigorous Analysis}
	\thanks{
			This work was partially supported by the National Key R\&D Program of China (No.~2022YFA1004500).}
	}
	
\author{
	Chaoyi Cai\thanks{Co-first author. School of Mathematical Sciences, Xiamen University, Xiamen, Fujian 361005, China (\email{caichaoyi@stu.xmu.edu.cn}).}
	\and Di Wu\thanks{Co-first author. School of Mathematical Sciences, Xiamen University, Xiamen, Fujian 361005, China (\email{wudiwork@stu.xmu.edu.cn}).}
	\and Jianxian Qiu\thanks{Corresponding author. School of Mathematical Sciences and Fujian Provincial Key Laboratory of Mathematical Modeling and High-Performance Scientific Computing, Xiamen University, Xiamen, Fujian 361005, China (\email{jxqiu@xmu.edu.cn}).}
	}
	
	\maketitle
	
	\begingroup
	\renewcommand\thefootnote{\fnsymbol{footnote}}%
	\setcounter{footnote}{0}%
	\setlength{\footnotemargin}{1.1em}
	\endgroup
\begin{abstract}
	Preserving pressure equilibrium across material interfaces is critical for the stability of compressible multi‐material flow simulations, yet most interface‐fitted sharp‐interface schemes are notoriously sensitive to interface geometry: even slight perturbations of the captured (or tracked) interface can trigger large spurious pressure oscillations. We present a moving-interface cut-cell method that is \emph{geometric-perturbation-robust (GPR)} for the compressible Euler equations with two materials. By construction, the scheme provably preserves exact interfacial pressure equilibrium in the presence of small interface-position errors.
	The key is a strict consistency between the conserved variables and the \emph{geometric moments} (i.e., the integrals of monomials) of every cut cell. We formulate auxiliary transport equations whose discrete solutions furnish \emph{evolved} geometric moments; using a discretization consistent with the flow solver, these geometric moments remain perfectly synchronized with the conserved variables---even on a deforming mesh. Surpassing the classical geometric conservation law, which constrains only cell volumes, our approach keeps \emph{all} higher-order geometric moments consistent, thereby eliminating accuracy loss caused by geometric mismatches and offering significant potential for \emph{general} moving‐mesh schemes.
	To prevent the reconstruction step from destroying pressure equilibrium, we introduce the notion of \emph{equilibrium-compatible} (EC) reconstructions. 
	A carefully designed modification equips any conventional weighted essentially non-oscillatory (WENO) reconstruction with the EC property; we detail a third-order EC multi-resolution WENO (EC-MRWENO) variant and employ it both for cell-interface flux calculations and for the redistribution procedure following each Runge--Kutta step.
	The tight coupling of EC-MRWENO with the evolved moments yields the first cut-cell solver that is simultaneously provably GPR and \emph{genuinely high-order}: it attains second-order accuracy (in the analytic-continuation sense) precisely at material interfaces---even in the presence of density discontinuities---while preserving third-order accuracy in smooth regions. Extensive two-dimensional tests confirm the framework’s robustness, accuracy, and stability under geometric perturbations and topological changes.
\end{abstract}


\section{Introduction}

\subsection{Governing equations}
High-speed compressible flows involving multiple materials occur in a variety of demanding applications in computational fluid dynamics, such as inertial-confinement fusion (ICF) and underwater explosions. When viscous and diffusive effects are negligible, such \emph{multi-material} (sometimes called \emph{multi-medium}) problems are routinely modeled by the inviscid compressible Euler equations:
\vspace{-5pt}
\begin{equation}
	\label{equ:euler}
	\frac{\partial \mathbf U}{\partial t} \;+\; \nabla\!\cdot\!\mathbf F(\mathbf U)
	\;=\; 0
	\vspace{-5pt}
\end{equation}
with the conserved variable and flux tensor defined by
$
\mathbf U = \bigl(\rho,\,\rho\mathbf v^\top,\,E\bigr)^{\top}
$,
$\label{equ:eulerflux}
\mathbf F(\mathbf U)=\bigl(\rho\mathbf v,\rho\mathbf v\otimes\mathbf v + p\mathbf I,(E+p)\mathbf v\bigr)^\top.$ 
Here $\rho$ denotes the density, $\mathbf v=(v_1,\dots,v_d)^\top$ the velocity vector, $E$ the total energy, and $p$ the thermodynamic pressure. Energy and pressure are linked through
\vspace{-5pt}
\begin{equation}\label{equ:Ep}
	E = \rho e+\frac12 \rho\lvert\mathbf v\rvert^{2},
	\vspace{-5pt}
\end{equation}
where $e$ is the specific internal energy.  An equation of state (EOS)
$p=p(\rho,e)$ is required to close the system. In this work, we assume a {\em stiffened} EOS of the form
\vspace{-5pt}
\begin{equation}
	\label{equ:EOS}
	p \;=\; (\gamma-1)\,\rho e \;-\;\gamma B,
	\vspace{-5pt}
\end{equation}
where the parameters $(\gamma,B)$ are material-specific constants. In particular, a ``multi-material'' flow is one in which the EOS parameters $(\gamma,B)$ are piecewise constant and tied to each material, moving with the flow. The compressible Euler equations \eqref{equ:euler} equipped with such piecewise‑constant EOS parameters are used to model interactions among distinct compressible fluids (gases, liquids, etc.) under high‑pressure conditions~\cite{saurel1999simple}.

Classical schemes for hyperbolic conservation laws (e.g.\ essentially non-oscillatory (ENO) and weighted ENO (WENO) schemes \cite{shu2020essentially}, discontinuous Galerkin (DG) schemes \cite{cockburn2012discontinuous}, etc.) are well-established and successfully capture shocks in single-material flows. However, if applied naively to a multi-material flow with discontinuous EOS parameters, these schemes produce severe nonphysical pressure oscillations at material interfaces. 
This well-known interface instability arises not from the high-order reconstruction, but from the discontinuity in the EOS (and energy field) across the interface \cite{abgrall2001computations}. To robustly simulate multi-material flows, it is therefore essential to account for the different EOS across interfaces and preserve pressure equilibrium across interfaces.

\subsection{Brief review}

Numerical methods for multi-material flows generally fall into two broad categories.  {\em Diffuse-interface} methods embed the material interface within a thin, numerically diffuse mixing layer. These methods are easy to implement in multiple dimensions and naturally handle topological changes.  Diffuse-interface methods have achieved some success, but they tend to produce artificially smeared interfaces. The mixing layer inexorably spreads and the precise interface location becomes blurred, eventually causing modeling error. Moreover, because the interface is intentionally smeared across several grid cells, the scheme's inherent numerical dissipation limits the convergence rate to first order, even when the underlying algorithm is formally high order. These issues have been explored and partially addressed; see the comprehensive reviews \cite{saurel2018diffuse,adebayo2025review}.

By contrast, {\em sharp-interface} methods explicitly capture or track material interfaces. Such schemes typically involve two ingredients: (i) an interface-capturing (or tracking) algorithm (e.g., level set \cite{GIBOU201882} and volume of fluid \cite{mohan2024volume} methods) to locate the discrete interface $\Gamma_h$, and (ii) a single-material solver advances the flow in regions away from $\Gamma_h$, while a multi-material solver handles the cells adjacent to $\Gamma_h$. A popular sharp-interface approach is the ghost fluid method (GFM) \cite{fedkiw1999isobaric}: a level set method identifies the interface, and ghost states are defined across it so as to impose the jump conditions (e.g., equal pressure) between fluids. The GFM essentially solves single-material Riemann problems on each side and uses ``ghost'' cells to prevent pressure oscillations. Subsequent variants of the GFM, such as the modified GFM \cite{LIU2003651} and the real GFM \cite{WANG2006RGFM}, further improve the shock-capturing robustness. However, the reliance on virtual ghost cells inherently limits the formal order of accuracy near the interface. Cut-cell methods are another popular class of sharp-interface methods, where grid cells intersected by the interface are geometrically cut into irregular polyhedral subcells (cut cells) \cite{SEO20117347,SCHNEIDERS2013786,DUPUY2025106627,liang2025fourth}. This ensures the interface lies exactly on cell faces of a locally unstructured mesh. Cut-cell methods have the benefit of smaller conservation errors and, in principle, enable higher-order reconstruction near the interface \cite{lin2017comparison}.  Numerous cut-cell methods have been developed for multi-material flows, often involving sophisticated cell-cutting and merging algorithms \cite{hu2006conservative,luo2015conservative,lin2017simulation,deng2018simulating,zheng2021high}. Other sharp-interface strategies include, but are not limited to, moment-of-fluid methods \cite{HERGIBO2025113908}, front-tracking methods \cite{TERASHIMA20094012}, interface treatment methods \cite{chertock2008interface}, simple single-fluid algorithms \cite{abgrall2001computations} and path-conservative methods \cite{xiong2012weno}. For further details on sharp‐interface methods, see the review \cite{GIBOU2019442}.

\begin{figure}[!thbp]
	\centering
	\captionsetup[subfigure]{labelformat=empty}
	\setlength{\abovecaptionskip}{-2pt}
	\subfloat[]{
		\includegraphics[height=0.14\textheight]{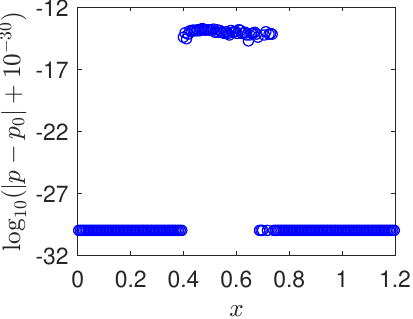}\label{fig:pureI_p_noPer_old}
	}\hspace{0.00\textwidth}
	\subfloat[]{
		\includegraphics[height=0.137\textheight]{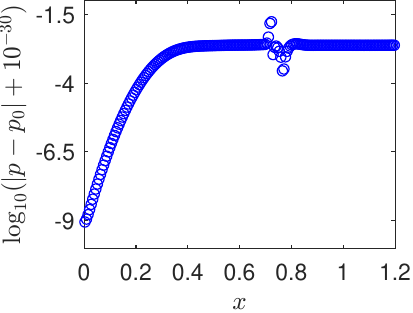}\label{fig:pureI_p_Per_old}
	}\hspace{0.00\textwidth}
	\subfloat[]{
		\includegraphics[height=0.14\textheight]{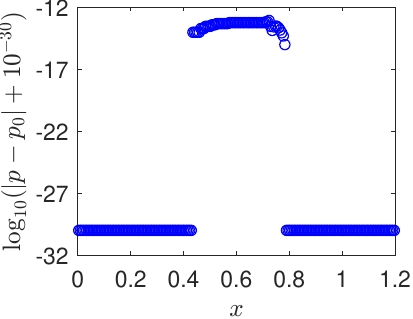}\label{pureI_p_Per_new}
	}\\[-18pt]
	\hspace{0.000\textwidth}
	\subfloat[]{
		\includegraphics[height=0.139\textheight]{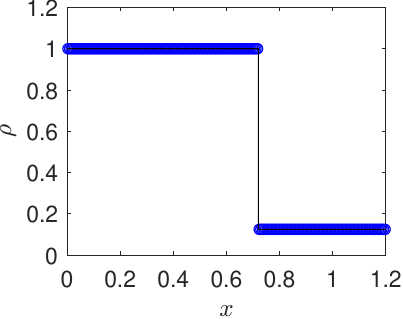}\label{fig:pureI_rho_noPer_old}
	}\hspace{0.01\textwidth}
	\subfloat[]{
		\includegraphics[height=0.139\textheight]{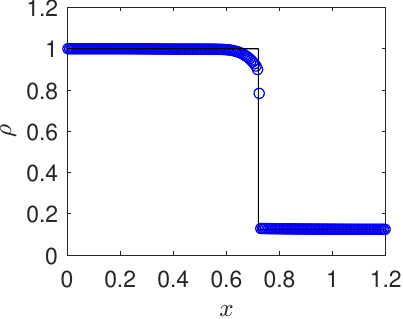}\label{fig:pureI_rho_Per_old}
	}\hspace{0.01\textwidth}
	\subfloat[]{
		\includegraphics[height=0.139\textheight]{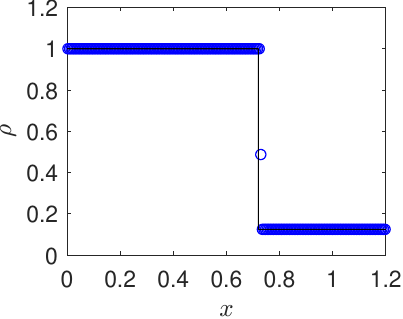}\label{pureI_rho_Per_new}
	}
	\captionsetup{font=small}
	\caption{Sensitivity of three schemes to geometric perturbations for problem~\eqref{test:pureI}: 
		(Left) solution from the conservative one‑dimensional cut‑cell method~\cite{liu2001simulation}, employing first‑order reconstruction with an LLF flux; 
		(Middle) same method with the level‑set function randomly perturbed at each time step by amplitude~\(10^{-4}\,\Delta x\); 
		(Right) solution from the GPR‑LLF scheme (see Section~\ref{subsec:whole GPR}) under the same perturbation level. Computations use 200 grid points and final time~\(T=0.32\).}

	\label{fig:pureI_randPer}
\end{figure}

\subsection{Geometric perturbation robustness}\label{sec:intro_gpr}

Physically, a material interface behaves as a contact discontinuity: once the local pressure and velocity are uniform $(p^{\ast},\mathbf v^{\ast})$, they remain unchanged under the evolution of the Euler equations. A numerical method must therefore preserve this pressure equilibrium at the discrete level. Failing to do so (i.e., spurious pressure oscillations arise), the scheme becomes unstable and the solution can deviate drastically from the entropy solution. In the early 1990s, Karni proposed evolving the primitive variables $(\rho,\mathbf v,p)$, thereby suppressing the oscillations but sacrificing exact energy conservation~\cite{KARNI199431}. Subsequently, Abgrall et al. \cite{abgrall1996prevent} and Shyue et al. \cite{shyue1998efficient} introduced alternative algorithms that preserve pressure equilibrium. In parallel, seven-equation models and their variants enforce a single or relaxed pressure condition at the model level \cite{BAER1986861,saurel1999multiphase}. Over the past twenty years, researchers have continued to refine pressure-equilibrium-preserving techniques and to survey the field in depth \cite{saurel2018diffuse,adebayo2025review}. It is
therefore fair to state that pressure-equilibrium preservation lies at the \emph{heart} of modern multi-material flow simulation.

However, in interface‐fitted sharp‐interface schemes, the interface-capturing (or tracking) algorithm inevitably introduces small geometric errors, which make it challenging to preserve pressure equilibrium. For example, in level‑set methods one solves $\phi_t + \mathbf{v}\cdot\nabla\phi=0$ and locates the interface via $\phi=0$.  In this process, at least four sources of interface‑position error arise: (i) errors arising from velocity‐field extrapolation \cite{CHANG2013946}; (ii) errors in the level set solver itself \cite{GIBOU201882}; (iii) geometric reconstruction errors when extracting the interface from $\phi$ \cite{lin2017simulation}; and (iv) errors from reinitialization (to enforce $|\nabla\phi|=1$) \cite{shakoor2025review}. Although these errors may be minute---merely tiny perturbations of the exact interface---they can nonetheless undermine the discrete pressure equilibrium. The following test illustrates how sensitive the pressure equilibrium is to the interface location. In the one-dimensional pure interface problem
\vspace{-5pt}
\begin{equation}\label{test:pureI}
	(\rho,v_1,p,\gamma,B)=
	\begin{cases}
		(1,\;1,\;1,\;1.4,\;0), & x\le 0.4,\\
		(0.125,\;1,\;1,\;4,\;1), & x>0.4.
	\end{cases}
	\vspace{-5pt}
\end{equation}
A conservative first-order Lax--Friedrichs (LLF) cut-cell method \cite{liu2001simulation} yields a perfectly flat pressure profile if the interface position is exact. However, even a tiny random perturbation of the interface location causes large spurious pressure oscillations (see Figure \ref{fig:pureI_randPer}). Thus, a key challenge is to mitigate the coupling between interface‑location errors and the flow solver, so that small geometric errors do not catastrophically affect the numerical solution. To our knowledge, this important issue of ``geometric perturbation robustness'' has not been systematically addressed in the literature.

\subsection{Contributions and innovations of this paper}
The contributions and innovations of this work include:
\begin{description}
	\item[$\bullet$] {\bf Geometric perturbation robustness.} We introduce the concept of {\em geometric-perturbation-robust (GPR)} for sharp-interface schemes. A scheme is said to be GPR if small perturbations in the captured interface position do not compromise exact pressure-equilibrium preservation. We analyze why conservative sharp-interface methods struggle to be GPR, shedding light on the fundamental conflict between conservation and pressure-equilibrium preservation.
	
	\item[$\bullet$] {\bf PDE-based evolved geometric moments for \emph{general moving-mesh schemes}.} Our tests show that a cut-cell scheme using reconstructed geometric moments may fail to converge. We instead formulate auxiliary partial differential equations (PDEs) whose discrete solutions yield the \emph{evolved} geometric moments. We prove that, with a \emph{linear} spatial discretization, \emph{any moving-mesh scheme} that uses these evolved geometric moments reproduces the pure advection equation $u_t = 0$ \emph{exactly} for polynomials of \textbf{arbitrary degree}, even on a deforming mesh.
	
	\item[$\bullet$] \textbf{GPR cut-cell scheme with \emph{genuinely second-order accuracy}.}
	We develop a finite-volume cut-cell solver for two-material flows that is
	\emph{provably} GPR.  The key ingredient is a tight	coupling between the \emph{evolved} geometric moments and a single-material (one-sided) third-order WENO reconstruction; this pairing not only enforces	the geometric conservation law (GCL) but also keeps all higher-order moments consistent with the discrete solution. As a result,
	\emph{even when a contact discontinuity lies on the material interface, the scheme achieves second-order accuracy} (in the analytic-continuation sense) at the interface and retains third-order accuracy
	in the single-material regions, marking a significant improvement in
	interface resolution. With a curved-interface representation, the framework
	can, in principle, be extended to even higher order (this extension is the subject of our ongoing work).
	
	
	\item[$\bullet$] \textbf{EC reconstruction.}
	We introduce the notion of an \emph{equilibrium-compatible} (EC) reconstruction:	a reconstruction method is called EC if it returns the same constant pressure $p^{*}$ and velocity $\mathbf v^{*}$ whenever every cell in the stencil shares these values $(p^{*},\mathbf v^{*})$. We prove that when \emph{an} EC reconstruction is employed, the proposed GPR cut-cell scheme preserves the discrete pressure equilibrium exactly. Furthermore, we give a minimal modification that endows any WENO reconstructions with the EC property.
	As a concrete example, we design a third-order EC multi-resolution WENO (EC-MRWENO) reconstruction. The proposed EC modification is essential---particularly
	during the redistribution step---for preserving pressure equilibrium.
	
	
	
	\item[$\bullet$] \textbf{Topology‐change capability.}
	The proposed framework can be integrated into other cut-cell solvers to preserve pressure equilibrium without undermining their native ability to handle complex topological changes. Representative tests---such as shock-bubble interactions and underwater explosion problems---confirm that the method remains stable even when the interface undergoes severe deformation and breakup.
\end{description}

The remainder of the paper is organized as follows. Section~2 formalizes the concept of a pressure-equilibrium-preserving scheme, introduces the GPR property, and analyzes the intrinsic incompatibility between pressure-equilibrium preservation and conservation. Section~3 derives the PDE-based evolved geometric moments and describes the resulting GPR cut-cell scheme. Section~4 investigates the EC property of polynomial reconstructions and presents a simple EC modification for general WENO reconstructions. Section~5 presents extensive numerical experiments that demonstrate the accuracy and robustness of the proposed methods.
Finally, Section~6 offers concluding remarks.
\section{Mathematical preliminaries on pressure-equilibrium preservation and geometric perturbation robustness}\label{sec:mathpre}
This section investigates the intrinsic pressure-equilibrium preservation properties of the Euler equations, introduces the notion of \emph{geometric-perturbation-robust} (GPR), and thereby provides the theoretical foundation for the GPR cut-cell scheme developed in the next section. We also analyze the intrinsic incompatibility between exact pressure‑equilibrium preservation and strict conservation, and discuss the extension of pressure‑equilibrium preservation to general conservation laws.


\subsection{Pressure-equilibrium preservation}\label{sec:PE}
We begin by recalling a well-known property of the compressible Euler
equations:

\begin{property}[Linear-degeneracy property]\label{prop:linear-degen}
	Consider the $d$-dimensional compressible Euler equations in
	$\mathbb R^d$ (Eq.~\eqref{equ:euler}) modeling a multi-material flow (i.e., the EOS parameters ($\gamma,B$) are piecewise-constant and advected with the material).  Let
	$\mathbf x=(x_1,\dots,x_d)^\top$ denote the spatial coordinates and assume
	that the initial data satisfy $\mathbf v_0(\mathbf x)=\mathbf v^\ast$, $p_0(\mathbf x)=p^\ast$, where $\mathbf v^\ast$ and $p^\ast$ are constants.  Then, for any
	$t>0$, the entropy solution is
	\begin{equation}\label{equ:passive solution}
		\begin{cases}
			\rho(t,\mathbf x)=\rho_0(\mathbf x-\mathbf v^\ast t),\\[2pt]
			p(t,\mathbf x)=p^\ast,\\[2pt]
			\mathbf v(t,\mathbf x)=\mathbf v^\ast.
		\end{cases}
	\end{equation}
	In other words, the initial density profile is merely translated by the constant velocity $\mathbf v^\ast$ without distortion, while pressure and velocity remain uniform.
\end{property}

It is natural to define the pressure‐equilibrium state set:

\begin{Def}\label{def:set G}
	The pressure-equilibrium state set is defined by $$\mathcal{G}(\mathbf{v}^*,p^*;\gamma,B)\!=\!\left\{\mathbf{U}\middle|\mathbf{v}\!=\!\mathbf{v}^*,p\!=\!p^*,\rho\!>\!0\right\}\!=\!\left\{\mathbf{U}\!=\!\left(\rho,(\rho \mathbf{v}^*)^\top,\frac{p^*\!+\!\gamma B}{\gamma-1}\!+\!\frac12\rho|\mathbf{v}^*|\right)^\top\!\middle|\rho\!>\!0\right\},$$
where $\gamma>1$ and $B\geq0$ are the material-specific constants in the EOS \eqref{equ:EOS}.
\end{Def} 

Since every component of $\left(\rho,(\rho \mathbf{v}^*)^\top,\frac{p^*+\gamma B}{\gamma-1}+\frac12\rho|\mathbf{v}^*|^2\right)^\top$ is a linear function of $\rho$, it follows that the set $\mathcal{G}(\mathbf{v}^*,p^*;\gamma,B)$ is convex.
\begin{lemma}\label{lem:convex}
	$\mathcal{G}(\mathbf{v}^*,p^*;\gamma,B)$ is a convex set.
\end{lemma}

We now present the definition of a pressure‑equilibrium‑preserving scheme.

\begin{definition}[Pressure-equilibrium-preserving scheme]\label{def:pe_scheme}
	Let \(\mathbf{U}_h(\mathbf{x},t;\gamma_h,B_h)\) denote the discrete solution produced by a numerical method for the multi‑material Euler equations \eqref{equ:euler}, where \((\gamma_h,B_h)\) are the numerical EOS parameters at \(\mathbf{x}\) (approximating \((\gamma,B)\)). The scheme is said to be pressure‐equilibrium‐preserving if, whenever
	\[
	\mathbf{U}_h(\mathbf{x},0)\in \mathcal{G}(\mathbf{v}^*, p^*;\gamma_h,B_h)
	\quad \forall~\mathbf{x},
	\]
	then, for any time $t>0$, it holds that
	\[
	\mathbf{U}_h(\mathbf{x},t)\in\mathcal{G}(\mathbf{v}^*, p^*;\gamma_h,B_h)
	\quad \forall~\mathbf{x}.
	\]
	In the finite‐volume context, $\mathbf{U}_h$ refers to the cell‐average value at location~$\mathbf{x}$.
\end{definition}

%

\subsection{Extension for general conservation laws}
Although this paper focuses on pressure-equilibrium preservation for the Euler equations \eqref{equ:euler}, the same idea can be naturally extended to a broader class of conservation laws
$\partial_t \mathbf U + \nabla\!\cdot\!\mathbf F(\mathbf U)
	=0,$ 
where \(\mathbf U\in\mathbb R^m\) is the unknown vector,
\(\mathbf F(\mathbf U)\in\mathbb R^{m\times d}\) the flux tensor.

Suppose the flux can be decomposed as
$
\mathbf F(\mathbf U)
=\mathbf v(\mathbf U)\otimes \mathbf U
+\mathbf R(\mathbf U)$ $~\big(\mathbf v(\mathbf U), \mathbf R(\mathbf U)\in\mathbb R^d\big)$, so that \(\mathbf v\) plays the role of a transport velocity.  If the initial data satisfy $\nabla\cdot\mathbf R\bigl(\mathbf U(\mathbf x,0)\bigr)=\mathbf 0$ and $\mathbf v(\mathbf U(\mathbf x,0))=\mathbf v^*$, then the conservation laws degenerate to the linear advection equations $\partial_t \mathbf U + \mathbf v^*\cdot\nabla\mathbf U = 0,$
whose exact solution is simply $\mathbf U(\mathbf x,t)=\mathbf U_0(\mathbf x-\mathbf v^*t).$ This solution continues to satisfy $\nabla \cdot \mathbf R\bigl(\mathbf U(\mathbf x,0)\bigr)=\mathbf 0$ and $\mathbf v(\mathbf U(\mathbf x,0))=\mathbf v^*$, so both the transport velocity and the ``equilibrium'' condition $\nabla\cdot\mathbf R=0$ are exactly maintained.
Therefore, numerical schemes for the conservation laws should be designed to discretely preserve this linear‑advection structure of non‑uniform states: it must maintain both a constant transport velocity \(\mathbf{v}^*\) and the remainder term $\nabla\cdot\mathbf R=0$, rather than merely preserving a uniform solution \(\mathbf{U}=\mathrm{const}.\)

\begin{remark}
	The above linear‑advection preservation property can be regarded as a generalization of the well‑known \emph{free‑stream preservation} property \cite{obayashi1991free}.
\end{remark}

\subsection{Geometric-perturbation-robust (GPR): Conservative sharp-interface schemes are \emph{not} GPR}
This subsection focuses on \emph{sharp-interface} schemes for the
multi-material Euler equations \eqref{equ:euler}. In our context, a sharp-interface scheme is a \emph{compound scheme} consisting of two ingredients:

\begin{enumerate}\setlength{\itemsep}{2pt}
	\item \textbf{Interface-capturing (or tracking) algorithm.}
	A dedicated interface capturing (or tracking) algorithm (e.g., level set and volume of fluid methods) is
	used to locate the discrete material interface~$\Gamma_h$.
	\item \textbf{Mixed single-/multi-material evolution.}
	Once $\Gamma_h$ is known, a single-material solver advances the flow in regions away from $\Gamma_h$, while a multi-material solver handles the cells adjacent to $\Gamma_h$.
\end{enumerate}

This broad definition subsumes many popular schemes, including the Lagrangian methods, front-tracking methods, ghost-fluid methods, and the cut-cell methods proposed in Section~\ref{sec:framework}. As discussed in Section~\ref{sec:intro_gpr}, \emph{all} interface‑capture (or tracking) algorithms inevitably incur small errors in the interface position. Even though these errors may be minute, \emph{minimal} shifts in the interface location can undermine the discrete pressure equilibrium and severely distort the numerical solution, as shown in Figure~\ref{fig:pureI_randPer}. This sensitivity motivates the following definition.

\begin{Def}[Geometric perturbation robustness]\label{def:GPR}
	A sharp‐interface scheme is GPR if, whenever it is pressure-equilibrium-preserving for the captured (tracked) material interface~$\Gamma_h$, there
	exists a constant $\varepsilon>0$ such that the scheme remains pressure-equilibrium-preserving for any perturbed interface
	$\tilde{\Gamma}_h$ satisfying
	$d_H(\Gamma_h,\tilde{\Gamma}_h)<\varepsilon$,
	where $d_H$ denotes the Hausdorff distance.
\end{Def}

\begin{remark}
	The concept of geometric perturbation robustness can be interpreted
	in a broader sense: small perturbations in the interface
	position do not compromise a given structure of the numerical solution (e.g., positivity of pressure or density \cite{cheng2014positivity}).
	In the present study, we focus exclusively on the preservation of
	pressure equilibrium.
\end{remark}

\begin{figure}[!thbp]
	\centering
	\captionsetup[subfigure]{labelformat=empty}
	\includegraphics[width=0.8\textwidth]{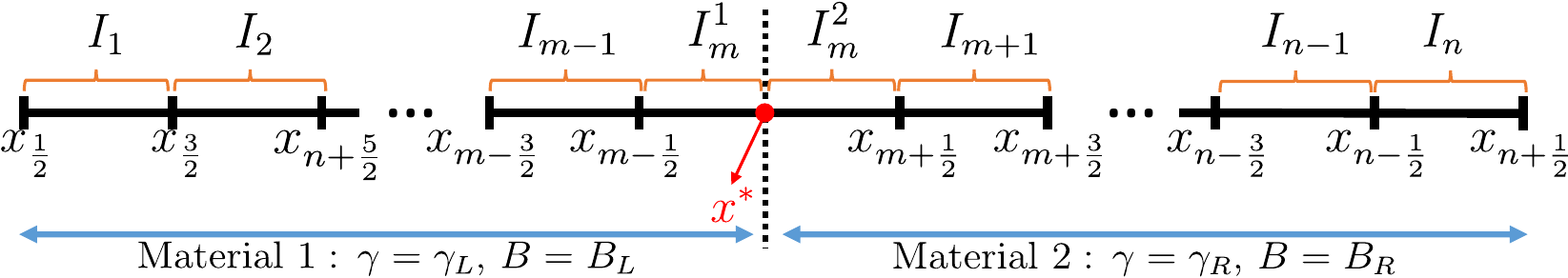}
	\captionsetup{font=small}
	\caption{Grid configuration for a one-dimensional two-material flow problem, $x^*$ is the numerical interface position computed by the interface capturing (or tracking) algorithm.}\label{fig:pureI_mesh}
\end{figure}


The following theorem, viewed through the lens of geometric‑perturbation robustness, reveals the intrinsic incompatibility between conservation and pressure‑equilibrium preservation in numerical schemes:

\begin{theorem}[GPR--conservation conflict]\label{thm:connotGPR}
	A conservative sharp-interface scheme for the multi-material problem
	\eqref{equ:euler} is \emph{not} GPR. Here ``conservative'' means that the scheme preserves the total momentum, the total energy, and the mass of	\emph{each individual material} exactly.
\end{theorem}

\begin{proof}
	For clarity we analyze the simple one-dimensional contact-discontinuity
	problem
	\begin{equation}\label{equ:pureI_proof}
		(\rho,v_1,p,\gamma,B)=
		\begin{cases}
			(\rho_L,v^\ast,p^\ast,\gamma_L,B_L), & x\le x^\ast,\\
			(\rho_R,v^\ast,p^\ast,\gamma_R,B_R), & x>  x^\ast,
		\end{cases}
	\end{equation}
	and the computational mesh is sketched in Figure \ref{fig:pureI_mesh}.
	The interface position produced by the interface capturing (or tracking) algorithm is
	$x^\ast(t)$, while
	$x_{i-\frac12}(t)$ $(1\le i\le n+1)$ denote cell edges;
	$ I_m^{1}(t)$ and $ I_m^{2}(t)$ are the two interface cells,
	all others being single-material cells. For simplicity, we assume fixed domain endpoints: $x_{\frac12}(t)\equiv x_{\frac12}(0)$, $x_{n+\frac12}(t)\equiv x_{n+\frac12}(0)$.
	
	\medskip
	During the first time step ($t=\Delta t$), material~1 occupies the cells
	$I_1(\Delta t)$, $\dots$, $I_{m-1}(\Delta t)$, $I_m^{1}(\Delta t)$.
	The total volume of these cells is
	\begin{equation}\label{equ:dt_Im1}
		\sum_{i=1}^{m-1}\bigl|I_i(\Delta t)\bigr|
		\;+\;
		\bigl|I_m^{1}(\Delta t)\bigr|
		\;=\;
		\sum_{i=1}^{m-1}\bigl|I_i(0)\bigr|
		\;+\;
		\bigl|I_m^{1}(0)\bigr|
		\;+\;
		\bigl(x^\ast(\Delta t)-x^\ast(0)\bigr),
	\end{equation}
	and the total (integrated) density (mass) of material 1 is
	\begin{equation}\label{equ:totalrho1}
		\rho_{L,\text{tot}}(\Delta t)
		\;=\;
		\sum_{i=1}^{m-1}\!\bigl|I_i(\Delta t)\bigr|\,\rho_i
		\;+\;
		\bigl|I_m^{1}(\Delta t)\bigr|\,\rho_m^{1},
	\end{equation}
	where $\rho_i$ and $\rho_m^1$ denote the cell‐averaged densities in $I_i$ and $I_m^1$ at $t=\Delta t$. Note that $\rho_i$ and $\rho_m^1$ need not equal the initial uniform density $\rho_L$.
	
	On a sufficiently fine grid (i.e., for sufficiently large $n$), the cells immediately adjacent to the left boundary $x_{\frac12}$ remain in a uniform state $(\rho_L,v^\ast,p^\ast)$. Consequently, after $\Delta t$, the total mass entering material 1 from the left boundary $x_{\frac12}$ is $\rho_Lv^*\Delta t$, which is computed directly from the fluxes in \eqref{equ:euler}. Moreover, since the scheme strictly conserves mass in each material, there is no mass exchange between two materials. Therefore, the total mass of material 1 at $t=\Delta t$ is given by 
	\begin{equation}\label{equ:totalrho2}
		\rho_{L,\text{tot}}(\Delta t)
		\;=\;
		\left(x^\ast(0)-x_{\frac12}(0)\right)\rho_L
		+ \rho_L\,v^\ast\,\Delta t .
	\end{equation}
	Similarly, let $f^E$ denote the \emph{total} $E$ transferred from material 1 to material 2 during the time step $\Delta t$, we can obtain the total energy of material 1:
	\begin{equation}\label{equ:totalE}
		E_{L,\text{tot}}(\Delta t)
		=\sum_{i=1}^{m-1}\!\bigl|I_i(\Delta t)\bigr|\,E_i
		+\bigl|I_m^{1}(\Delta t)\bigr|\,E_m^{1}
		=\left(x^\ast(0)-x_{\frac12}(0)\right)E_L
		+v^\ast\!(E_L+p^\ast)\Delta t-f^{E},
	\end{equation}
	where $E_L=\frac{p^*+\gamma_LB_L}{\gamma_L-1}+\frac12\rho_L(v^*)^2$, $E_i$ and $E_m^1$ denote the cell‐averaged energy in cells $I_i$ and $I_m^1$ at $t=\Delta t$.
	
	Substituting $E_L \;=\;
	\frac{p^* + \gamma_L B_L}{\gamma_L - 1}
	+ \tfrac12\,\rho_L (v^*)^2$ together with \eqref{equ:dt_Im1}, \eqref{equ:totalrho1}, and \eqref{equ:totalrho2} into \eqref{equ:totalE}, and carrying out straightforward algebraic rearrangements, we obtain
	\begin{equation}\label{equ:fE_final1}
		\frac{f^{E}}{\Delta t}
		= v^\ast p^\ast
		+\frac{p^\ast+\gamma_L B_L}{\gamma_L-1}
		\left(
		v^\ast
		-\frac{x^\ast(\Delta t)-x^\ast(0)}{\Delta t}
		\right).
	\end{equation}
	An analogous calculation for material~2 gives
	\begin{equation}\label{equ:fE_final2}
		\frac{f^{E}}{\Delta t}
		= v^\ast p^\ast
		+\frac{p^\ast+\gamma_R B_R}{\gamma_R-1}
		\left(
		v^\ast
		-\frac{x^\ast(\Delta t)-x^\ast(0)}{\Delta t}
		\right).
	\end{equation}
	Combining \eqref{equ:fE_final1} and \eqref{equ:fE_final2} yields
	\begin{equation}\label{equ: gprnotconserv_equ}
		\bigl(
		\frac{p^\ast+\gamma_L B_L}{\gamma_L-1}
		-\frac{p^\ast+\gamma_R B_R}{\gamma_R-1}
		\bigr)
		\Bigl(
		v^\ast-\frac{x^\ast(\Delta t)-x^\ast(0)}{\Delta t}
		\Bigr)=0.
	\end{equation}
	Because $p^\ast$, $B_L$, $B_R$, $\gamma_L$, and $\gamma_R$
	are arbitrary, the only way to satisfy the identity is to enforce
	$x^\ast(\Delta t)=x^\ast(0)+v^\ast\Delta t$ \emph{exactly}.
	Therefore, regardless of how the energy flux $f^{E}$ is chosen,
	\textbf{any} perturbation of the interface position $x^\ast(\Delta t)$---no
	matter how small---breaks pressure-equilibrium preservation; a conservative sharp-interface scheme is thus \emph{not} GPR.
\end{proof}

If we relax the conservation requirements in Theorem~\ref{thm:connotGPR} to demand only conservation of total energy and \emph{total} mass, one can still derive equation \eqref{equ: gprnotconserv_equ} by an argument similar to that used in Theorem~\ref{thm:connotGPR}. Consequently, we have the following
result:
\begin{proposition}\label{prop:connotGPR_relax}
	Sharp-interface schemes that conserves total energy and \emph{total} mass are \emph{not} GPR.
\end{proposition}
Theorem \ref{thm:connotGPR} and Proposition \ref{prop:connotGPR_relax} reveal a fundamental incompatibility between conservation and pressure equilibrium preservation. This conflict originates from discontinuities in the EOS parameters
$(\gamma,B)$. By contrast, single‑material flows automatically satisfy equation~\eqref{equ: gprnotconserv_equ}, and thus no GPR--conservation conflict arises in that case.

\section{A GPR cut-cell scheme for two-material flows}\label{sec:framework}
In this section, we develop a novel cut-cell scheme for two-dimensional, two-material flows that is not only GPR but also genuinely high-order. For the reader's convenience, the complete computational workflow of the GPR cut‑cell scheme is presented in Figure~\ref{fig:precise} of Section~\ref{subsec:whole GPR}.
Owing to its GPR property, the proposed method integrates seamlessly with a variety of interface-capturing and interface‑cell generation methods. For concreteness, we illustrate the framework using the classical WENO-JS level set method~\cite{jiangwenojshamilton,GIBOU201882} together with the cell-cutting and merging algorithm proposed by Lin et al. \cite{lin2017simulation}.

The proposed cut-cell framework is based on a finite volume method. For convenience, we henceforth denote the spatial coordinate by \(\mathbf{x}=(x,y)\). As illustrated in Figure \ref{fig:face_rec}, the computational domain is discretized by a uniform Cartesian grid, where each cell is denoted by $I_{i,j}=\left[x_{i-\frac12},x_{i+\frac12}\right]\times\left[y_{j-\frac12},y_{j+\frac12}\right]$. The mesh is equispaced in each direction, i.e., $x_{i+\frac12}-x_{i-\frac12}=\Delta x$ and $y_{j+\frac12}-y_{j-\frac12}=\Delta y$.

\begin{figure}[!thbp]
	\centering
	\captionsetup[subfigure]{labelformat=empty}
	\includegraphics[width=0.9\textwidth]{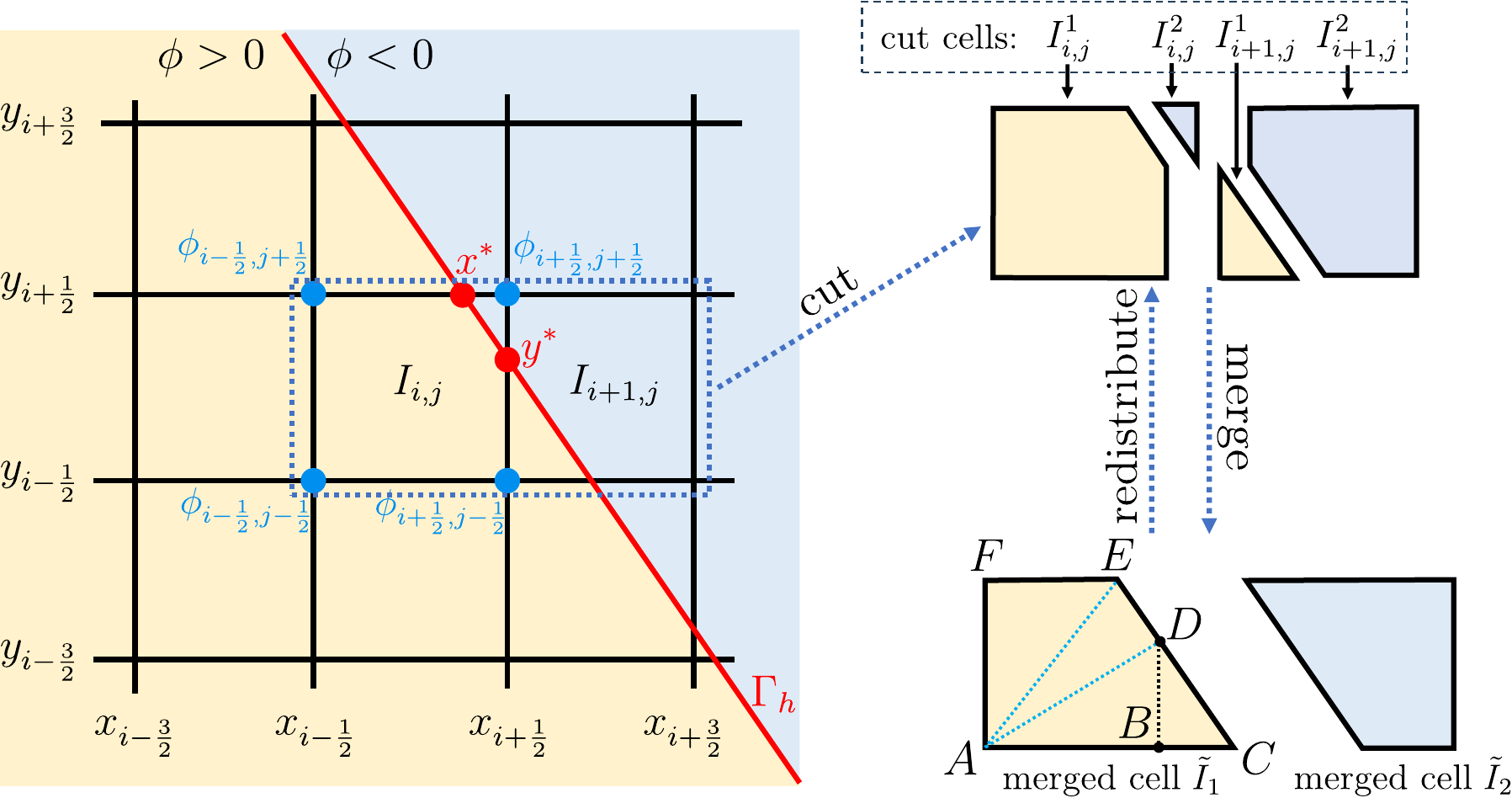}
	\captionsetup{font=small}
	\caption{Partitioning a Cartesian cell into cut cells using the level set function $\phi$, and employing a cell merging technique to solve the small cell problem.}\label{fig:face_rec}
\end{figure}


\subsection{Level‑set method and interface‑cell handling}
In this work, the moving interface \(\Gamma\) is tracked by the level‑set equation
\begin{equation}\label{equ:levelset}
	\phi_t + \mathbf{v}\!\cdot\!\nabla\phi = 0,
\end{equation}
where \(\phi\) is initialized as the signed‑distance function (\(\lvert\nabla\phi\rvert=1\)),
positive in material 1 and negative in material 2.  For convenience, \(\phi\) is stored at
the cell vertices \((x_{i\pm\frac12},y_{j\pm\frac12})\).  Equation \eqref{equ:levelset}
is discretized with the classical WENO–JS scheme~\cite{jiangwenojshamilton}, and vertex
velocities are obtained via WENO reconstruction (see Section \ref{sec:EC-MRWENO}).  
Alternative solvers may be substituted without loss of generality.

The zero‑level contour of \(\phi\) defines the numerical interface \(\Gamma_h\). For example, in Figure \ref{fig:face_rec}, the right and upper boundaries of cell $I_{i,j}$ are intersected by the interface. We simply use linear interpolation to determine the cut-point coordinates \cite{hu2006conservative}:
\begin{equation}\label{equ:lsm recon}
	x^* = \frac{x_{i-\frac12}\,\phi_{i+\frac12,j+\frac12} - x_{i+\frac12}\,\phi_{i-\frac12,j+\frac12}}
	{\phi_{i+\frac12,j+\frac12}-\phi_{i-\frac12,j+\frac12}},  
	\quad
	y^* = \frac{y_{j-\frac12}\,\phi_{i+\frac12,j+\frac12} - y_{j+\frac12}\,\phi_{i+\frac12,j-\frac12}}
	{\phi_{i+\frac12,j+\frac12}-\phi_{i+\frac12,j-\frac12}}.
\end{equation}
These coordinates, $x^*$ and $y^*$, are then used to reconstruct the linear interface within $I_{i,j}$, thereby partitioning the Cartesian cell $I_{i,j}$ into two cut cells, $I_{i,j}^1$ and $I_{i,j}^2$, each dominated by a single material.

To avoid excessively small cut‑cells, an undersized cell is merged with its neighbor along the interface normal. In Figure~\ref{fig:face_rec}, for instance, cut‑cells $I_{i,j}^1$ and $I_{i+1,j}^1$ are combined into the merged cell $\tilde I_1$, while $I_{i,j}^2$ and $I_{i+1,j}^2$ form $\tilde I_2$. Throughout the remainder of this paper, the symbol $\tilde I_i$ denotes the $i$‑th merged cell. While numerous cutting
and merging algorithms are available, the GPR cut‑cell framework is agnostic to the specific choice; here, the method of Lin et al. \cite{lin2017simulation} is used for illustration.

Since \(\phi\) gradually drifts from a signed‑distance function under advection,
periodic reinitialization is required to restore \(\lvert\nabla\phi\rvert=1\).
This work adopts the WENO–JS reinitialization procedure \cite{jiangwenojshamilton,zheng2021high}.

\subsection{Geometric moments on a diet:  matching ``inexact'' geometry to an inexact solution}
\subsubsection{Reconstructed geometric moments are inconsistent}
The integral equations equivalent to Eq.~\eqref{equ:euler} read
\begin{equation}\label{equ:euler_int form}
	\frac{\mathrm{d}}{\mathrm{d}t}\int_{\Omega(t)} \mathbf{U}\, \mathrm{d}x\mathrm{d}y + \int_{\partial\Omega(t)} \Big[\mathbf{F}(\mathbf{U}) - \mathbf{U}\cdot\mathbf{v}_b\Big]\cdot\mathbf{n}\, \mathrm{d}S = 0,
\end{equation}
where $\Omega$ denotes a control volume, $\partial \Omega$ its boundary, $\mathbf{n}=(n_x,n_y)^\top$ the unit outward normal on $\partial \Omega$, and \(\mathbf{v}_b\) the velocity of $\partial \Omega$. When employing the finite volume method to solve equations \eqref{equ:euler}, we are essentially solving the above integral equations, with the total (integrated) conserved quantity within each cell, $\mathbf{U}^{\mathrm{tot}}=\int \mathbf{U}\, \mathrm{d}x\mathrm{d}y,$
serving as the evolving variables. 

For fixed-mesh schemes, $\mathbf{U}^{\mathrm{tot}}$ can be rewritten directly in terms of cell averages.  
However, when the mesh itself moves we can only monitor $\mathbf{U}^{\mathrm{tot}}$. Consequently, a crucial step in finite volume computations is reconstructing point values (or averages) from the numerical total conserved quantities $\mathbf{U}_h^{\mathrm{tot}}$ in each cell. As an example, the WENO reconstruction (detailed in Section \ref{sec:WENO}) seeks a polynomial
$p(x,y)=\sum_{s,r} a_{s,r}x^{s}y^{r}$ (where $a_{s,r}\in\mathbb{R}$)
such that the total integral of $p(x,y)$ over a given control
volume~$\Omega(t)$ matches the cell‐integrated conserved variable:
\[
\int_{\Omega(t)} p(x,y)\,\mathrm{d}x\mathrm{d}y
\!=\!
\sum_{s,r} \left(a_{s,r}\!\int_{\Omega(t)} \!x^{s}y^{r}\!\mathrm{d}x\mathrm{d}y\!\right)
\!=\!
\int_{\Omega(t)} (\mathbf U)_k\,\mathrm{d}x\,\mathrm{d}y
\!=\!
(\mathbf{U}^{\mathrm{tot}}_h)_k,
\quad 1\le k\le 4 .
\]

The crucial ingredients are therefore the \emph{high-order geometric
	moments}
\begin{equation}\label{def:geo moment M}
	\mathcal{M}_{s,r}(t)
	\;=\;
	\int_{\Omega(t)} x^{s}y^{r}\,\mathrm{d}x\,\mathrm{d}y,
\end{equation}
which must be calculated for all required
indices~$(s,r)$.

In conventional schemes, the geometric moment \(\mathcal{M}_{s,r}\) is often computed directly from the cell geometry. For example, for the merged cell \(\tilde I_1\) shown in Figure~\ref{fig:face_rec}, one evaluates \(\mathcal{M}_{s,r}\) by integrating the monomial \(x^{s}y^{r}\) over the polygon \(ABCDEF\). Here, \(A\), \(B\), and \(F\) are fixed Cartesian vertices, while \(C\), \(D\), and \(E\) lie on the interface and are reconstructed from the level set function (with Eq.~\eqref{equ:lsm recon}). In practice, one may subdivide the polygon into triangles---for example,\ \(\triangle ABD\), \(\triangle ADE\), \(\triangle AEF\), and \(\triangle BCD\)---and apply a sufficiently accurate Gaussian quadrature on each triangle to calculate the integrals of \(x^{s}y^{r}\). Beyond direct volume integration, one may also invoke the divergence theorem to transform these into boundary integrals, as in the work of Evrard et al. \cite{Evrard20023divergence,wasquel2025first}. We refer to the resulting high-order moments as the \emph{reconstructed geometric moments}, denoted by \(\mathcal{M}_{s,r}^{\mathrm{rec}}\).


Although the purely geometric reconstruction of $\mathcal{M}_{s,r}^{\mathrm{rec}}$ can be made arbitrarily accurate, it does not ensure that these reconstructed moments evolve in step with $\mathbf{U}^{\mathrm{tot}}$: using $\mathcal{M}_{s,r}^{\mathrm{rec}}$ does not naturally preserve pressure equilibrium. Such inconsistency can also degrade
accuracy or even destroy convergence in the cut‑cell scheme (see
Test \ref{sec:test:smooth_single} in Section~\ref{sec:numerical_tests}). The problem is not limited to cut‑cell schemes; it also appears in other moving‑mesh methods (e.g.\ Arbitrary Lagrangian--Eulerian method) and becomes especially pronounced under large mesh deformations (expansion, contraction, or twisting). Although the geometric conservation law (GCL) \cite{ThomasLombard1979} can alleviate this inconsistency, the classical GCL enforces consistency only for the zeroth‑order moment (i.e., the volume $\mathcal{M}_{0,0}$) and overlooks all higher‑order geometric moments $\mathcal{M}_{s,r}$ with $s+r>0$.

\subsubsection{Consistent \emph{evolved geometric moments} via auxiliary transport equations}

Our remedy is to allow the geometry to be just as ``inexact'' as the discrete solution. We treat every geometric moment
$\mathcal{M}_{s,r}(t)=\int_{\Omega(t)}x^{s}y^{r}\,\mathrm{d} x\mathrm{d}y$
as a time-dependent unknown and evolve it:
\begin{equation}\label{equ:high order gcl}
	\frac{\mathrm{d}\mathcal{M}_{s,r}}{\mathrm{d}t}
	=\frac{\mathrm{d}}{\mathrm{d}t}\int_{\Omega(t)}x^{s}y^{r}\,\mathrm{d}x\mathrm{d}y
	=\int_{\partial\Omega(t)}x^{s}y^{r}
	(\mathbf v_b\,\cdot\,\mathbf n)\,\mathrm{d}S\xlongequal{v_\perp:=\mathbf v_b\,\cdot\,\mathbf n}\int_{\partial\Omega(t)}x^{s}y^{r}\,v_\perp\,\mathrm{d}S,
\end{equation}
where Reynolds' transport theorem is applied.
Geometric moments produced by discretizing the auxiliary transport equations \eqref{equ:high order gcl} are called \emph{\textit{evolved geometric moments}},
written $\mathcal{M}_{s,r}^{\mathrm{evo}}$. Notably, when \(s=r=0\), equation~\eqref{equ:high order gcl} reduces to the
classical GCL.

To demonstrate that the evolved geometric moments can be consistent with the numerical solution, we consider the trivial equation $u_t=0$ in its integral form 
\begin{equation}\label{equ:ut=0_int_equ}
	\frac{\mathrm{d}}{\mathrm{d}t}\int_{\Omega(t)}u\,\mathrm{d}x\,\mathrm{d}y
	\;=\;
	\int_{\partial\Omega(t)}u\,(\mathbf{v}_b\cdot\mathbf{n})\,\mathrm{d}S,
\end{equation}
which is structurally identical to Eq.~\eqref{equ:high order gcl}. When discretized in a finite‑volume framework, the evolving variable is the total conserved quantity $u^{\mathrm{tot}}
\;=\;
\int u\,\mathrm{d}x\,\mathrm{d}y$ in each cell.
The following theorem demonstrates that, provided Eqs.~\eqref{equ:high order gcl}--\eqref{equ:ut=0_int_equ} employ the same \emph{linear} spatial discretization operator, any moving‑mesh scheme transports polynomials of \emph{arbitrary degree} exactly, even when using a first‑order forward Euler method:

\begin{theorem}[Polynomial exactness with $\mathcal{M}_{s,r}^{\mathrm{evo}}$]\label{thm:poly-exact}
	Let $\{I_i(t)\}_{i=1}^{N}$ be the moving control volumes of a finite‑volume scheme, whose cell boundaries move with a prescribed velocity. Assume the \textbf{same linear} spatial operator $\tilde{\mathcal L}(\,\cdot\,)$ is applied to advance both the transport
	equation \eqref{equ:ut=0_int_equ} and the geometric‑moment equations
	\eqref{equ:high order gcl} with a forward-Euler method
	\[
	(\cdot)^{n+1}
	\;=\;
	(\cdot)^{n} + \Delta t\,\tilde{\mathcal L}\!\bigl((\cdot)^{n}\bigr).
	\]
	If, at time $t^{n}$, the solution is the polynomial
	$u(x,y)=\sum_{s,r} a_{s,r}\,x^{s}y^{r}$, then each cell $I_i(t^{n})$
	satisfies
	\[
	u^{\mathrm{tot},n}_h
	= 
	\sum_{s,r} a_{s,r}\,
	\mathcal M_{s,r}^{\mathrm{evo},n}
	\left(=\int_{I_i(t^n)} \sum_{s,r} a_{s,r}\,x^s y^r \,\mathrm dx\mathrm dy
	\right),
	\]
	and the same identity holds at $t^{n+1}=t^{n}+\Delta t$:
	\begin{equation}\label{equ: exact-poly}
		u^{\mathrm{tot},n+1}_h
		=
		\sum_{s,r} a_{s,r}\,
		\mathcal M_{s,r}^{\mathrm{evo},n+1}
		\left(=\int_{I_i(t^{n+1})} \sum_{s,r} a_{s,r}\,x^s y^r \,\mathrm dx\mathrm dy
		\right).
	\end{equation}
	Hence the scheme transports polynomials of \textbf{arbitrary degree}
	exactly.
\end{theorem}

\begin{proof}
	Using the forward Euler step on $u^{\mathrm{tot}}_h$, linearity of $\tilde{\mathcal L}$ gives
	\begin{align*}
		u^{\mathrm{tot},n+1}_h
		&= u^{\mathrm{tot},n}_h
		+ \Delta t\,\tilde{\mathcal L}\bigl(u^{\mathrm{tot},n}_h\bigr)\\
		&= \sum_{s,r}a_{s,r}\,\mathcal{M}^{\mathrm{evo}}_{s,r}(t^n)
		+ \Delta t\,\tilde{\mathcal L}\!\Bigl(\sum_{s,r}a_{s,r}\,\mathcal{M}^{\mathrm{evo},n}_{s,r}\Bigr)\\
		&= \sum_{s,r}a_{s,r}\,\mathcal{M}^{\mathrm{evo},n}_{s,r}
		+ \Delta t\,\sum_{s,r}a_{s,r}\,\tilde{\mathcal L}\bigl(\mathcal{M}^{\mathrm{evo},n}_{s,r}\bigr)\\
		&= \sum_{s,r}a_{s,r}\,\mathcal{M}^{\mathrm{evo},n+1}_{s,r},
	\end{align*}
	which completes the proof.
\end{proof}

Theorem \ref{thm:poly-exact} holds not only for the forward Euler method but also for general strong‐stability‐preserving Runge--Kutta (SSPRK) time discretization schemes \cite{shu1989efficient}.

\begin{proposition}\label{prop:ssprk3_poly_exact}
	If, in Theorem \ref{thm:poly-exact}, the forward Euler method is replaced by any SSPRK time discretization scheme, then the conclusion of exact polynomial transport still holds.
\end{proposition}

\begin{proof}
	Since an SSPRK time discretization can be expressed as a convex combination of forward Euler steps \cite{shu1989efficient}, the proof is analogous to that of
	Theorem \ref{thm:poly-exact}.
\end{proof}

Proposition \ref{prop:ssprk3_poly_exact} guarantees that the polynomial information is preserved exactly; with a suitable reconstruction algorithm, the original polynomial can be fully recovered. Note that the validity of Proposition \ref{prop:ssprk3_poly_exact} is independent of mesh quality or the specifics of the spatial discretization. Therefore, by employing the evolved geometric moments, any moving‑mesh scheme ensures consistency between the numerical solution and the conserved variables---even under large mesh deformations.

\subsubsection{Discretization of \textit{evolved geometric moments} in the present cut-cell scheme}\label{subsec:egm-discrete}
We now return to the Euler equations~\eqref{equ:euler_int form}.  
In our moving-interface cut-cell framework, the material interface travels with the local fluid
velocity; consequently the normal speed appearing in
\eqref{equ:high order gcl} satisfies
$v_{\perp}=\mathbf v \cdot \mathbf n$. 	Motivated by Theorem~\ref{thm:poly-exact}, we require that the
\emph{spatial discretization of the geometric moment equation
	\eqref{equ:high order gcl} be consistent with that of Euler equations}~\eqref{equ:euler_int form}.  
Such compatibility guarantees that the evolved geometric moments and the
finite‐volume solution advance coherently, an ingredient that will
prove indispensable for pressure-equilibrium preservation
(see Theorem~\ref{thm:eulerforward pe} in Section \ref{sec:euler_pe}).

\begin{figure}[!thbp]
	\centering
	\captionsetup[subfigure]{labelformat=empty}
	\includegraphics[width=0.6\textwidth]{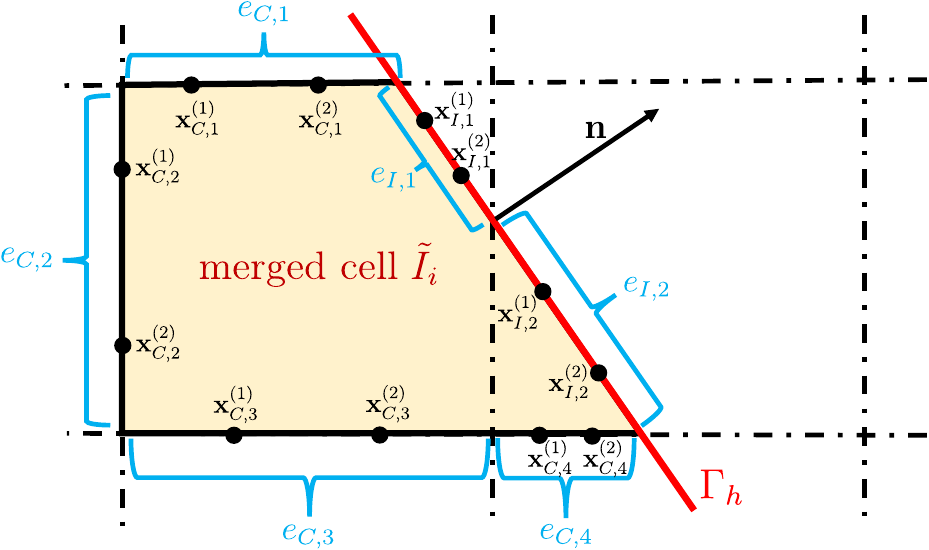}
	\captionsetup{font=small}
	\caption{The boundary integrals for the merged cells are computed using piecewise Gaussian quadrature; the black dots denote the Gauss points along each segment, and the red line represents the interface between the materials.}\label{fig:face_flux}
\end{figure}

We now present our spatial discretization for
equations~\eqref{equ:high order gcl}. As shown in Figure \ref{fig:face_flux},
\[
\mathcal{E}_C(\tilde{I}_i) = \{\,e_{C,1},\dots,e_{C,N_C}\}
\quad\text{and}\quad
\mathcal{E}_I(\tilde{I}_i) = \{\,e_{I,1},\dots,e_{I,N_I}\}
\]
denote the sets of Cartesian‐edge segments and interface‐edge segments of the merged cell $\tilde{I}_i$, respectively. We simply place 2 Gauss (quadrature) points \(\{\,\mathbf{x}_{C,\alpha}^{(1)},\mathbf{x}_{C,\alpha}^{(2)}\}\) on each Cartesian‐edge segment $e_{C,\alpha}\in\mathcal{E}_C$, 
and 2 Gauss points \(\{\,\mathbf{x}_{I,\beta}^{(1)},\mathbf{x}_{I,\beta}^{(2)}\}\) on each interface segment $e_{I,\beta}\in\mathcal{E}_I$. For further details on Gaussian quadrature, see \cite{quarteroni2006numerical}.

In merged cell \(\tilde I_i\), the Cartesian‐edge segments are either remain fixed or move strictly tangentially ($v_\perp=\mathbf v_b\cdot\mathbf n=0$), whereas the interface‐edge segments move with the interface. Consequently, the normal velocity \(v_\perp\) vanishes on every Cartesian‐edge segments, only the interface‐edge segments contribute to the boundary integral in Eq.~\eqref{equ:high order gcl}. Accordingly, we employ the following piecewise Gaussian quadrature to discretize equations \eqref{equ:high order gcl} in space:
\begin{equation}\label{equ:semidisc_M}
	\frac{\mathrm{d}\mathcal{M}_{s,r}^{\mathrm{evo}}}{\mathrm{d}t}
	\;\approx\;
	\mathcal{L}\bigl(\mathcal{M}_{s,r}^{\mathrm{evo}}\bigr)\;:=\;\sum_{\beta=1}^{N_I}
	\sum_{\ell=1}^{2}
	\bigl(x_{I,\beta}^{(\ell)}\bigr)^s \bigl(y_{I,\beta}^{(\ell)}\bigr)^r
	\bigl|e(\mathbf{x}_{I,\beta}^{(\ell)}\bigr)\bigr|\,
	\omega\bigl(\mathbf{x}_{I,\beta}^{(\ell)}\bigr)\,
	v_\perp\bigl(\mathbf{x}_{I,\beta}^{(\ell)}\bigr),
\end{equation}
where:
\begin{itemize}
	\item \(\mathbf{x}_{I,\beta}^{(\ell)} = \bigl(x_{I,\beta}^{(\ell)},\,y_{C,\alpha}^{(\ell)}\bigr)\) (or \(\mathbf{x}_{C,\alpha}^{(\ell)} = \bigl(x_{C,\alpha}^{(\ell)},\,y_{C,\alpha}^{(\ell)}\bigr)\)) denotes the $\ell$-th Gauss point on the interface‐edge segment \(e_{I,\beta}\) (or Cartesian‐edge segment \(e_{C,\alpha}\));
	\item \(\omega(\mathbf{x}_{I,\beta}^{(\ell)})\) (or \(\omega(\mathbf{x}_{C,\alpha}^{(\ell)})\)) is the Gaussian quadrature weight at Gauss point \(\mathbf{x}_{I,\beta}^{(\ell)}\) (or \(\mathbf{x}_{C,\alpha}^{(\ell)}\));  
	\item \(\lvert e(\mathbf{x}_{I,\beta}^{(\ell)}\bigr)\rvert\) (or \(\lvert e(\mathbf{x}_{C,\alpha}^{(\ell)}\bigr)\rvert\)) denotes the length of the segment on which the Gauss point \(\mathbf{x}_{I,\beta}^{(\ell)}\) (or \(\mathbf{x}_{C,\alpha}^{(\ell)}\)) lies; and
	\item \(v_\perp(\mathbf{x}_{I,\beta}^{(\ell)})\) is the normal velocity at the interface‐edge Gauss point \(\mathbf{x}_{I,\beta}^{(\ell)}\). The computation of \(v_\perp(\mathbf{x}_{I,\beta}^{(\ell)})\) is detailed in Section~\ref{subsec:semidisc-num}.
\end{itemize}


The semi‐discrete scheme \eqref{equ:semidisc_M} is compatible with the semi‐discrete scheme of $\mathbf{U}^\text{tot}$ (Eq.~\eqref{equ:semidiscre-totalU} in Section \ref{subsec:semidisc-num}), thus ensuring that $\mathcal{M}_{s,r}^{\text{evo}}$ remain fully consistent with the numerical conserved variables $\mathbf{U}_h^{\mathrm{tot}}$. Theorem \ref{thm:eulerforward pe} demonstrates that reconstruction based on these evolved moments preserves discrete pressure equilibrium. Furthermore, Test \ref{sec:test:smooth_single} (Section \ref{sec:numerical_tests}) confirms that incorporating $\mathcal{M}_{s,r}^{\text{evo}}$ increases the cut‑cell scheme to higher‑order accuracy, underscoring the consistency between $\mathcal{M}_{s,r}^{\text{evo}}$ and $\mathbf{U}_h^{\mathrm{tot}}$.

\begin{remark}
	In practical computations, the geometric moments $\mathcal{M}_{s,r}^{\text{evo}}$ of Cartesian cells outside the vicinity of the interface remain constant over time and do not require calculating Eq.~$\eqref{equ:semidisc_M}$.
\end{remark}

\subsection{EC reconstruction method}\label{subsec: EC1}
With the evolved geometric moments $\mathcal{M}_{s,r}^{\text{evo}}$ available, one can reconstruct cell averages (and point values) from the cells' total conserved quantities $\mathbf{U}^{\mathrm{tot}}$. To preserve pressure-equilibrium, it is natural to require that the reconstruction method satisfy the following equilibrium-compatible (EC) property:
\begin{definition}[Equilibrium-compatible property]\label{def:EC property}
	A reconstruction method is EC if, whenever all the cell-averaged conserved variables $\overline{\mathbf{U}}=\frac{\mathbf{U}^{\text{tot}}}{\mathcal{M}_{0,0}}$ in the
	reconstruction stencil share a uniform velocity $\mathbf{v}^*$ and pressure $p^*$,
	the reconstructed point values reproduce \emph{exactly} \(\mathbf v^*\) and \(p^*\).
\end{definition}

This property is a crucial ingredient for preserving pressure equilibrium. In Section \ref{sec:WENO}, we propose a simple EC modification for general WENO reconstructions. This modification preserves the original high‑order accuracy and non‑oscillatory behavior of WENO while endowing it with the EC property.

\subsection{Semi-discrete scheme for \(\mathbf{U}^{\mathrm{tot}}=\int \mathbf{U}\, \mathrm{d}x\mathrm{d}y\) in each cell}\label{subsec:semidisc-num}
With the evolved geometric moments \(\mathcal{M}_{s,r}^{\mathrm{evo}}\) and the EC reconstruction in hand, we are prepared to spatially discretize the Euler equations \(\eqref{equ:euler_int form}\). We present only the semi‐discrete formulation for merged cells; the formulation for fixed Cartesian cells is formally identical. Adopting the same notations as in Eq.~\(\eqref{equ:semidisc_M}\), the spatial discretization of Eq.~\(\eqref{equ:euler_int form}\) then proceeds as follows:

\begin{description}
	\item[\textbf{Step 1.}] Within the merged cell \(\tilde I_i\), employ an EC reconstruction to calculate the conserved variables $\mathbf{U}$ at all Gauss points
		\[
		\bigl\{\mathbf x_{C,\alpha}^{(k)} \mid 1\le k\le2,\;1\le\alpha\le N_C\bigr\}
		\quad\text{and}\quad
		\bigl\{\mathbf x_{I,\beta}^{(\ell)} \mid 1\le \ell\le2,\;1\le\beta\le N_I\bigr\},
		\]
		using \(\mathbf U^{\mathrm{tot}}_h\) and \(\mathcal M_{s,r}^{\mathrm{evo}}\) from neighboring cells. For instance, in Figure~\ref{fig:face_flux}, we reconstruct values at 
		\(\mathbf x_{C,\alpha}^{(k)}\) for \(\alpha=1,\dots,4\) and \(k=1,2\),
		and at \(\mathbf x_{I,\beta}^{(\ell)}\) for \(\beta=1,2\) and \(\ell=1,2\).
		
		For brevity, let 
		\(\mathbf U^{\mathrm{int}}_{C,\alpha,k}\) and \(\mathbf U^{\mathrm{int}}_{I,\beta,\ell}\)
		denote the values reconstructed at the Cartesian‐edge Gauss point 
		\(\mathbf x_{C,\alpha}^{(k)}\) and the interface‐edge Gauss point 
		\(\mathbf x_{I,\beta}^{(\ell)}\), respectively.  The superscript “\(\mathrm{int}\)” indicates reconstruction from the interior of \(\tilde I_i\).

	\item[\textbf{Step 2.}] Similarly, reconstruct these Gauss point values on the exterior of $\tilde{I}_i$, which we denote by \(\mathbf U^{\mathrm{ext}}_{C,\alpha,k}\) and \(\mathbf U^{\mathrm{ext}}_{I,\beta,\ell}\). The superscript ``ext'' signifies that these values are obtained from the exterior of $\tilde{I}_i$.
	
	\item[\textbf{Step 3.}] The semi-discrete finite volume scheme for Eq.~\eqref{equ:euler_int form} is obtained by applying a Gaussian quadrature rule to each cell boundary segment:
	\begin{equation}\label{equ:semidiscre-totalU}
		\begin{aligned}
			\frac{\mathrm{d} \mathbf{U}^{\text{tot}}}{\mathrm{d} t}\approx\mathcal{L}(\mathbf{U}^{\text{tot}}_h):=&-\sum_{\alpha=1}^{N_C}\sum_{k=1}^{2}|e(\mathbf{x}_{C,\alpha}^{(k)})|\omega(\mathbf{x}_{C,\alpha}^{(k)})\widehat{\mathbf{F}}^{\text{single}}\left(\mathbf U^{\mathrm{int}}_{C,\alpha,k},\mathbf U^{\mathrm{ext}}_{C,\alpha,k}\right)\\
			&-\sum_{\beta=1}^{N_I}\sum_{\ell=1}^{2}|e(\mathbf{x}_{I,\beta}^{(\ell)})|\omega(\mathbf{x}_{I,\beta}^{(\ell)})\widehat{\mathbf{F}}^{\text{multi}}\left(\mathbf U^{\mathrm{int}}_{I,\beta,\ell},\mathbf U^{\mathrm{ext}}_{I,\beta,\ell}\right).
		\end{aligned}
	\end{equation}
	
	Here, \(\widehat{\mathbf F}^{\mathrm{single}}\) denotes the single‐material numerical flux, whereas the multi‐material flux is defined by
	\begin{equation}\label{equ:multimedia flux}
		\widehat{\mathbf F}^{\mathrm{multi}}\bigl(\mathbf U^{\mathrm{int}}_{I,\beta,\ell},\,\mathbf U^{\mathrm{ext}}_{I,\beta,\ell}\bigr)
		= \bigl(0,\;p\mathbf n^\top,\;pv_{\perp}\bigr)^{\!\top},
	\end{equation}
	\(\mathbf n\) is the outward unit normal of $\partial \tilde{I}_i$ at \(\mathbf x_{I,\beta}^{(\ell)}\), and \((p, v_{\perp})\) are the pressure and normal velocity obtained from a multi‐material Riemann solver applied to the initial data $\left(\mathbf U^{\mathrm{int}}_{I,\beta,\ell},\mathbf U^{\mathrm{ext}}_{I,\beta,\ell}\right)$. The flux \eqref{equ:multimedia flux} is derived from the Reynolds transport theorem and has been widely used in the literature \cite{liu2001simulation,qiu2007runge,deng2018simulating,zheng2021high}.
	
	In this work, \(\widehat{\mathbf F}^{\mathrm{single}}\) is taken to be the local Lax--Friedrichs flux
	\begin{equation}\label{equ:LLF flux}
		\widehat{\mathbf F}^{\mathrm{LLF}}\bigl(\mathbf U^{\mathrm{int}}_{C,\alpha,k},\,\mathbf U^{\mathrm{ext}}_{C,\alpha,k}\bigr)
		= \tfrac12\Bigl[
		\mathbf F\bigl(\mathbf U^{\mathrm{int}}_{C,\alpha,k}\bigr)\cdot\mathbf n
		+ \mathbf F\bigl(\mathbf U^{\mathrm{ext}}_{C,\alpha,k}\bigr)\cdot\mathbf n
		- a_{\max}\bigl(\mathbf U^{\mathrm{ext}}_{C,\alpha,k}
		- \mathbf U^{\mathrm{int}}_{C,\alpha,k}\bigr)
		\Bigr],
	\end{equation}
	where \(\mathbf n\) is the outward unit normal at \(\mathbf x_{C,\alpha}^{(k)}\) and \(a_{\max}\) denotes the local maximum spectral radius of the Jacobian \(\partial(\mathbf F\cdot\mathbf n)/\partial\mathbf U\) \cite{cockburn1989tvb}; the multi‐material solver is taken to be the exact Riemann solver \cite{liu2001simulation}.
\end{description}
\begin{remark}\label{rem:multiriemann}
A natural requirement for the multi‐material Riemann solver used in Eq.~\eqref{equ:multimedia flux} is: if both $\mathbf U^{\mathrm{int}}_{I,\beta,\ell}$ and $\mathbf U^{\mathrm{ext}}_{I,\beta,\ell}$ share the same pressure $p^*$ and normal velocity $v^{*}_{\perp}$, then the Riemann solver return a solution with $p=p^*$ and $v_{\perp}=v^{*}_{\perp}$.
\end{remark}

\subsection{Key result: pressure-equilibrium preservation of semi-discrete schemes \eqref{equ:semidisc_M}--\eqref{equ:semidiscre-totalU} under forward Euler method}\label{sec:euler_pe}
For ease of reading, we first state the key theorem for the pressure-equilibrium-preserving property of finite-volume discretizations \eqref{equ:semidisc_M}--\eqref{equ:semidiscre-totalU};
the proof is deferred to Section~\ref{sec:proof of pe}.

\begin{theorem}[Pressure‐equilibrium preservation]\label{thm:eulerforward pe}
Let the spatial semi‐discrete schemes for both (i) the evolved geometric moments $\mathcal{M}_{s,r}^{\mathrm{evo}}$ (Section~\ref{subsec:egm-discrete}) and (ii) the total conserved variables $\mathbf{U}^{\mathrm{tot}}$ (Section~\ref{subsec:semidisc-num}) be advanced by the forward Euler method
\begin{equation}\label{equ:euler forward}
	(\cdot)^{n+1} \;=\; (\cdot)^n \;+\;\Delta t\,\mathcal{L}\bigl(\bigl(\cdot\bigr)^n\bigr).
\end{equation}
If at time $t^n$, every cell average
\[
\overline{\mathbf{U}}^n_h
\;=\;
\frac{\mathbf{U}^{\mathrm{tot},\,n}_h}
{\mathcal{M}_{0,0}^{\mathrm{evo},\,n}}
\;\in\;
\mathcal{G}(\mathbf{v}^*,p^*;\gamma_h, B_h),
\]
then after the forward‐Euler step \eqref{equ:euler forward}, each updated cell average preserves the pressure equilibrium
\[
\overline{\mathbf{U}}_{h}^{\,n+1}
\;=\;
\frac{\mathbf{U}^{\mathrm{tot},\,n+1}_h}
{\mathcal{M}_{0,0}^{\mathrm{evo},\,n+1}}\in\mathcal{G}(\mathbf{v}^*,p^*;\gamma_h, B_h).
\]
\end{theorem}

\subsection{Complete GPR scheme}\label{subsec:whole GPR}
The forward Euler method \eqref{equ:euler forward} is only
first-order accurate in time. To achieve higher temporal accuracy we employ the SSPRK method.  
Because an SSPRK method is formally a convex combination of forward Euler steps \cite{shu1989efficient}, Theorem~\ref{thm:eulerforward pe} together with
Lemma~\ref{lem:convex} yields the following result.
\begin{proposition}\label{prop:pe_0}
	If, in Theorem \ref{thm:eulerforward pe}, the forward Euler method is replaced by any SSPRK method, then the conclusion of pressure‐equilibrium preservation remains valid.
\end{proposition}

We adopt a third-order SSPRK (SSPRK3) method to advance both the geometric moments ${\mathcal{M}_{s,r}^{\text{evo}}}$ and the total conserved variables $\mathbf{U}^\text{tot}_h$ in time:
\begin{equation}\label{equ:ssprk3}
	\begin{cases}
		(\cdot)^{(1)}=(\cdot)^n+\Delta t\mathcal{L}((\cdot)^n),\\
		(\cdot)^{(2)}=\frac34(\cdot)^n+\frac14\left((\cdot)^{(1)}+\Delta t\mathcal{L}((\cdot)^{(1)})\right),\\
		(\cdot)^{n+1}=\frac13(\cdot)^n+\frac23\left((\cdot)^{(2)}+\Delta t\mathcal{L}((\cdot)^{(2)})\right).\\
	\end{cases}
\end{equation}
Here, the superscript $n$ indicates the numerical solution at time level $t^{n}$, $n+1$ denotes the solution at time level $t^{n+1}=t^{n}+\Delta t$. 

After a single SSPRK3 update to $t^{n+1}$, the discrepancy between the evolved moments $\mathcal{M}_{s,r}^{\mathrm{evo}}$ and the reconstructed moments \(\mathcal{M}_{s,r}^{\mathrm{rec}}\) increases, creating a mismatch between the flow solution and the interface geometry $\Gamma_h$. To restore consistency---and to prepare for the next cell-merging operation---the total conserved quantity $\mathbf U^{\mathrm{tot}}_h$ stored in each merged cell needs to be redistributed to its constituent sub‑cells, and the evolved moments $\mathcal{M}_{s,r}^{\mathrm{evo}}$ should be reset to $\mathcal{M}_{s,r}^{\mathrm{rec}}$. In this paper, we employ an EC polynomial reconstruction (detailed in Section \ref{sec:WENO}) to perform this redistribution, thereby achieving high‑order accuracy without disrupting pressure‑equilibrium preservation:

\noindent\textbf{Redistribution example.}  
Consider the configuration in Figure \ref{fig:face_rec}.  
The merged cell $\tilde I_1 = I_{i,j}^{1} \cup I_{i+1,j}^{1}$ has a known total conserved quantity $\mathbf U^{\text{tot}}_h(\tilde I_1)$ that must be redistributed to its two cut-cell children,
$I_{i,j}^{1}$ and $I_{i+1,j}^{1}$.
Using the neighboring cells' evolved geometric moments
$\mathcal M^{\text{evo}}_{s,r}$ and their total conserved quantities
$\mathbf U^{\text{tot}}_h$, the EC polynomial reconstruction produces a $N_p$-th-degree vector polynomial
\begin{equation}\label{equ:red_p}
	\mathbf U(x,y)
	\;\approx\;
	\mathbf p(x,y)
	\;=\;
	\sum_{s+r\le N_{p}}\mathbf a_{s,r}\,x^{s}y^{r},
	\qquad
	\mathbf a_{s,r}\in\mathbb R^{4},
\end{equation}
which subject to the conservation constraint
\begin{equation}\label{equ:red_p_con}
	\int_{\tilde I_{1}}\mathbf p(x,y)\,\mathrm{d}x\,\mathrm{d}y
	=\sum_{s+r\le N_p} \mathbf a_{s,r}
	\int_{\tilde I_{1}} x^{s}y^{r}\,\mathrm{d}x\,\mathrm{d}y=\sum_{s+r\le N_p} \mathbf a_{s,r}\,
	\mathcal M_{s,r}^{\mathrm{evo}}(\tilde I_{1})=\mathbf U^{\text{tot}}_h(\tilde I_{1}).
\end{equation}
The total conserved quantity assigned to cut cell $I_{i,j}^{1}$ can then be approximated by
\begin{equation}\label{equ:redistribute_1}
	\mathbf U^{\text{tot}}_h(I_{i,j}^{1})
	=
	\int_{I_{i,j}^{1}} \mathbf p(x,y)\,\mathrm{d}x\,\mathrm{d}y
	=
	\sum_{s+r\le N_p} \mathbf a_{s,r}
	\int_{I_{i,j}^{1}}\! x^{s}y^{r}\,\mathrm{d}x\,\mathrm{d}y
	=
	\sum_{s+r\le N_p} \mathbf a_{s,r}\,
	\mathcal M_{s,r}^{\mathrm{rec}}(I_{i,j}^{1}),
\end{equation}
where $\mathcal M_{s,r}^{\mathrm{rec}}(I_{i,j}^{1})$ denotes the reconstructed geometric moment of $I_{i,j}^{1}$.
An analogous approximation holds for the neighboring cut cell
$I_{i+1,j}^{1}$:
\begin{equation}\label{equ:redistribute_2}
	\mathbf U^{\text{tot}}_h(I_{i+1,j}^{1})
	\;=\;
	\sum_{s+r\le N_p} \mathbf a_{s,r}\,
	\mathcal M_{s,r}^{\mathrm{rec}}(I_{i+1,j}^{1}).
\end{equation}

\begin{lemma}\label{lem:pe_2}
	The EC polynomial redistribution procedure illustrated in Eqs.~\eqref{equ:redistribute_1}--\eqref{equ:redistribute_2}
	preserves pressure equilibrium of the cell average $\overline{\mathbf{U}}_h=\frac{\mathbf{U}^{\mathrm{tot}}_h}
	{\mathcal{M}_{0,0}^{\mathrm{rec}}}$ exactly.
\end{lemma}

For better readability, the proof of Lemma~\ref{lem:pe_2} is deferred to
Section~\ref{sec:red_pe}.

By assembling all of the preceding steps, we obtain the fully discrete GPR cut-cell scheme. For the reader's convenience, the detailed computational procedure of the complete GPR cut-cell scheme is presented in Figure \ref{fig:precise}. As proved in the following Theorem, this fully discrete scheme is pressure-equilibrium-preserving and GPR.

\begin{figure}[htbp]
	\centering
	\fcolorbox{black}{blue!10}{
		\parbox{.95\linewidth}{
			\small{\bf \begin{center}
					Precis of GPR cut-cell scheme for two-material flows with stiffened EOS \eqref{equ:EOS}
			\end{center}}
			\smallskip
			\begin{description}
				\item[\textbf{Step 1.}]  At time $t=0$, use the level set function $\phi$ to partition all Cartesian cells intersected by the interface into cut cells, and compute \(\mathcal{M}_{s,r}^{\mathrm{rec}}\) for each cut cell.
				
				\item[\textbf{Step 2.}] Apply a cell‐merging algorithm to mitigate the small‐cell problem. For each merged cell \(\tilde{I}\), obtain its total conserved variables \(\mathbf{U}^{\mathrm{tot}}_h\) and reconstructed geometric moments \(\mathcal{M}_{s,r}^{\mathrm{rec}}\) by summing those of its subcells, then initialize the evolved geometric moments \(\mathcal{M}_{s,r}^{\mathrm{evo}}\) as \(\mathcal{M}_{s,r}^{\mathrm{rec}}\).
				
				\item[\textbf{Step 3.}]  Advance \(\mathbf{U}^{\mathrm{tot}}_h\) and \(\mathcal{M}_{s,r}^{\mathrm{evo}}\) from $t^n$ to $t^{n+1}$ with the SSPRK3 method, using the semi-discrete operators in
				Eqs.~\eqref{equ:semidisc_M} and \eqref{equ:semidiscre-totalU}. Simultaneously, update the interface by solving the level set equation \eqref{equ:levelset}.
				
				\item[\textbf{Step 4.}] At time $t=t^{n+1},$ repartition the Cartesian cells intersected by the updated interface into cut cells, and
				recompute \(\mathcal{M}_{s,r}^{\mathrm{rec}}\) for each new cut cell.
				
				\item[\textbf{Step 5.}] Use an EC reconstruction to redistribute the
				conserved variables \(\mathbf{U}^{\mathrm{tot}}_h\) from each merged cell back to its subcells.
				Return to Step 2.
			\end{description}
		}
	}
	\captionsetup{font=small}
	\caption{Precise of the GPR cut-cell scheme for two-material flows.}
	\label{fig:precise}
\end{figure}

\begin{theorem}[Geometric-perturbation-robust]\label{thm:GPR is GPR}
	The GPR cut‐cell scheme is pressure-equilibrium-preserving and GPR.
\end{theorem}

\begin{proof}
	Proposition \ref{prop:pe_0} and Lemma~\ref{lem:pe_2} demonstrate that the manipulations performed in Steps 3, 4 and 5 each preserve pressure equilibrium. Moreover, Lemma~\ref{lem:convex} guarantees that the cell-merging procedure in Step 2 likewise preserves pressure equilibrium. Therefore, the GPR‑CC scheme is pressure‑equilibrium‑preserving.
	
	Finally, the proofs of Proposition~\ref{prop:pe_0} and Lemma~\ref{lem:pe_2} rely solely on the geometric assumption that the material interface never crosses a Cartesian vertex of any merged cell. Hence, any sufficiently small perturbation of the interface---provided it does not force the interface to cross a cell vertex---leaves every step of the foregoing argument intact, thus the scheme is GPR.
\end{proof}

Within the GPR framework outlined in Figure \ref{fig:precise},  
employing a \emph{first‐order} EC reconstruction
(i.e., using the cell average as the reconstructed value)
and replacing the SSPRK3 with a forward Euler method yields a first-order GPR-LLF scheme.
As presented in Figure \ref{fig:pureI_randPer}\,(right), even when the interface position is slightly perturbed, the solution structure obtained by GPR-LLF remains unchanged, confirming the GPR property of the proposed framework.

\subsection{Conservation error}\label{subsec:conserv-error}

Theorem~\ref{thm:connotGPR} reveals an intrinsic trade-off: one cannot simultaneously enforce strict conservation and GPR.
Within the proposed GPR cut-cell framework, the \emph{only}
non-conservative operation is the redistribution
(Eqs.~\eqref{equ:redistribute_1}--\eqref{equ:redistribute_2},
Step 5 in Figure \ref{fig:precise}).

\paragraph{Cell-wise error}
Adding Eqs.~\eqref{equ:redistribute_1} and \eqref{equ:redistribute_2} shows
that, after redistribution, the total conserved quantity of the merged cell
\(\tilde I_1=I_{i,j}^{1}\cup I_{i+1,j}^{1}\)
becomes
\[
\widehat{\mathbf U}^{\text{tot}}_h(\tilde I_1)
=\sum_{s+r\le N_p}\mathbf a_{s,r}\,
\bigl(
\mathcal M_{s,r}^{\mathrm{rec}}(I_{i,j}^{1})
+\mathcal M_{s,r}^{\mathrm{rec}}(I_{i+1,j}^{1})
\bigr).
\]
Because the reconstructed geometric moments satisfy 
\(
\mathcal M_{s,r}^{\mathrm{rec}}(I_{i,j}^{1})
+\mathcal M_{s,r}^{\mathrm{rec}}(I_{i+1,j}^{1})
=\mathcal M_{s,r}^{\mathrm{rec}}(I_{i,j}^{1}\cup I_{i+1,j}^{1})=\mathcal M_{s,r}^{\mathrm{rec}}(\tilde I_1)
\),
by the conservation constraint \eqref{equ:red_p_con},
the cell-wise conservation error is
\[
\bigl|
\widehat{\mathbf U}^{\text{tot}}_h(\tilde I_1)
-\mathbf U^{\text{tot}}_h(\tilde I_1)
\bigr|
=\Bigl|
\sum_{s+r\le N_p}\mathbf a_{s,r}\,
\bigl(
\mathcal M_{s,r}^{\mathrm{rec}}(\tilde I_1)
-\mathcal M_{s,r}^{\mathrm{evo}}(\tilde I_1)
\bigr)
\Bigr|,
\]
where both $\mathcal M_{s,r}^{\mathrm{evo}}$ and $\mathcal M_{s,r}^{\mathrm{rec}}$ are evaluated at the end of the SSPRK3 method ($t=t^{n+1}$). Thus, we have the following proposition immediately.

\begin{proposition}[Species-wise conservation error]
	\label{prop:species-mass-error}
	For any merged cell $\tilde I_i$, if at $t^{n+1}$ the equality
	\(\mathcal M_{s,r}^{\mathrm{rec}}(\tilde I_i)
	=\mathcal M_{s,r}^{\mathrm{evo}}(\tilde I_i)\)
	holds for all $s,r\in \mathbb{N}$,
	then the redistribution step introduces \emph{no} conservation error,
	including species (material-wise) mass errors.
\end{proposition}

\paragraph{Global error in a uniform state}
Figure \ref{fig:face_rec} and Eq.~\eqref{equ:semidisc_M} imply
\[
\sum_i\mathcal M_{s,r}^{\mathrm{evo}}(\tilde I_i)
=\sum_i\mathcal M_{s,r}^{\mathrm{rec}}(\tilde I_i)\qquad(\forall s,r\in \mathbb{N}).
\]
If every merged cell adjacent to the interface shares the same average value $\mathbf U^\ast$ 
(i.e., \(\mathbf U^{\text{tot}}_h/\mathcal M_{0,0}^{\text{evo}}=\mathbf U^\ast\)),
the polynomial in Eq.~\eqref{equ:red_p} degenerates to
\(\mathbf p(x,y)=\mathbf U^\ast\). Therefore,
\[
\sum_i\widehat{\mathbf U}^{\text{tot}}_h(\tilde I_i)
=\mathbf U^\ast\sum_i\mathcal M_{0,0}^{\text{rec}}(\tilde I_i)
=\mathbf U^\ast\sum_i\mathcal M_{0,0}^{\text{evo}}(\tilde I_i)
=\sum_i\mathbf U^{\text{tot}}_h(\tilde I_i).
\]

Hence, we obtain the proposition below.
\begin{proposition}[Global conservation in a uniform state]
	\label{prop:total-conserv-error}
	If all merged cells near the interface have the same average value $\mathbf U^\ast$, the redistribution step preserves the total mass,
	momentum, and energy; under these conditions the GPR cut-cell scheme is globally conservative.
\end{proposition}

Propositions \ref{prop:species-mass-error}–\ref{prop:total-conserv-error} offer the following practical insights:
\begin{enumerate}\setlength{\itemsep}{2pt}
	\item The cell-wise error is proportional to
	\(\lvert
	\mathcal M_{s,r}^{\text{rec}}
	-\mathcal M_{s,r}^{\text{evo}}
	\rvert\).
	Employing a more accurate interface-reconstruction technique and a time discretization scheme of order higher than SSPRK3 will reduce this discrepancy.
	In principle, enforcing $\mathcal{M}_{s,r}^{\text{evo}}=\mathcal{M}_{s,r}^{\text{rec}}$
	would deliver a fully conservative, pressure‑equilibrium‑preserving
	scheme. Such full conservation is expected to sacrifice the GPR property and would place stringent
	accuracy demands on the interface‑capturing algorithm.
	
	\item The global error diminishes as the local solution near the interface becomes smoother. When the solution is a (single-material) constant state, the conservation errors of total mass, momentum, and energy remain at machine precision.
\end{enumerate}

Although the present GPR scheme is not strictly conservative, the conservation error introduced in the redistribution procedure is high-order vanishing and numerically negligible (See Test \ref{sec:test:smooth_single} in Section \ref{sec:numerical_tests}).

\subsection{Proof of Theorem~\ref{thm:eulerforward pe}}\label{sec:proof of pe} We begin with the following trivial lemma.
\begin{lemma}[\cite{quarteroni2006numerical}]\label{lem:gauss_exact}
	Gauss quadrature can integrate constants exactly.
\end{lemma}

For clarity, let 
$\mathbf{U}^{\text{tot}}_h=(\rho^{\text{tot}},m_x^{\text{tot}},m_y^{\text{tot}},E^{\text{tot}})^\top=\int_{I} \mathbf{U}\, \mathrm{d}x\mathrm{d}y$
denote the total (i.e., cell-integrated) density $\rho$, $x$-momentum $\rho v_x$, $y$-momentum $\rho v_y$, and energy $E$ in a cell \(I\). Furthermore, let 
\begin{equation}\label{equ:ave def}
	\overline{\mathbf{U}}_h^n=(\overline{\rho},\overline{m}_x,\overline{m}_y,\overline{E})^\top=\frac{\mathbf{U}^{\text{tot},n}}{\mathcal{M}^{\text{evo},n}_{0,0}}
\end{equation} denote the cell-averaged conserved variables at $t=t^n$. Consider a merged cell $\tilde{I}$, and let $\tilde{\Gamma}=\bigcup_{\beta=1}^{N_I}e_{I,\beta}$ denote the interface edges of $\tilde{I}$. Then, we have the following lemma:

\begin{lemma}\label{lem:L(U)}
	If the cell-averaged conserved variables in the neighboring cells of the merged cell $\tilde{I}$ exhibit a constant velocity $\mathbf{v}^*=(v_x^*,v_y^*)$ and a constant pressure $p^*$, then the following relations hold:
	\begin{equation}\label{equ:L(U)}
		\begin{cases}
			\mathcal{L}(m_x^{\text{tot}})=v_x^*\mathcal{L}(\rho^{\text{tot}}),\\
			\mathcal{L}(m_y^{\text{tot}})=v_y^*\mathcal{L}(\rho^{\text{tot}}),\\
			\mathcal{L}(E^{\text{tot}})=\frac{|\mathbf{v}|^2}{2}\mathcal{L}(\rho^{\text{tot}})+\frac{p^*+\gamma_h B_h}{\gamma_h-1}\int_{\tilde{\Gamma}}(\mathbf{v^*}\cdot\mathbf{n})\mathrm{d}S.
		\end{cases}
	\end{equation}
	Here, $\mathcal{L}(\cdot)$ denote the spatial discretization operator defined in Eq.~\eqref{equ:semidiscre-totalU}, $\gamma_h$ and $B_h$ are the numerical EOS parameters for the merged cell $\tilde{I}$, and ``neighboring cells'' refers to all cells used in computing the Gauss point fluxes for $\tilde{I}$.
\end{lemma}
\begin{proof}
	From equations \eqref{equ:semidiscre-totalU}, \eqref{equ:multimedia flux} and \eqref{equ:LLF flux}, we immediately obtain
	$$\mathcal{L}(\rho^\text{tot})\!=\!-\sum_{\alpha=1}^{N_C}\sum_{k=1}^{2}|e(\mathbf{x}_{C,\alpha}^{(k)})|\frac{\omega(\mathbf{x}_{C,\alpha}^{(k)})}{2}\left(\left(\rho^{\mathrm{int}}_{C,\alpha,k}+\rho^{\mathrm{ext}}_{C,\alpha,k}\right)\left(\mathbf{v^*}\!\cdot\!\mathbf{n}\right)\!-\!a_{\max}\left(\rho^{\mathrm{ext}}_{C,\alpha,k}\!-\!\rho^{\mathrm{int}}_{C,\alpha,k}\right)\right).$$
	
	\begin{description}
		\item[(a)] To show that $\mathcal{L}(m_x^{\text{tot}})=v_x\mathcal{L}(\rho^{\text{tot}})$ and $\mathcal{L}(m_y^{\text{tot}})=v_y\mathcal{L}(\rho^{\text{tot}})$, we note from equations \eqref{equ:semidiscre-totalU}, \eqref{equ:multimedia flux} and \eqref{equ:LLF flux} that 
		\begin{equation*}
			\begin{aligned}
				\mathcal{L}(m_x^\text{tot})=&-\sum_{\alpha=1}^{N_C}\sum_{k=1}^{2}|e(\mathbf{x}_{C,\alpha}^{(k)})|\frac{\omega(\mathbf{x}_{C,\alpha}^{(k)})}{2}\left(\left(\rho^{\mathrm{int}}_{C,\alpha,k}+\rho^{\mathrm{ext}}_{C,\alpha,k}\right)(\mathbf{v^*}\cdot\mathbf{n})v_x^*\right.\\
				&\left.-a_{\max} v_x^*\left(\rho^{\mathrm{ext}}_{C,\alpha,k}-\rho^{\mathrm{int}}_{C,\alpha,k}\right)+2p^*n_x\right)-\sum_{\beta=1}^{N_I}\sum_{\ell=1}^{2}|e(\mathbf{x}_{I,\beta}^{(\ell)})|\omega(\mathbf{x}_{I,\beta}^{(\ell)})p^*n_x\\
				=&-p^*\left(\sum_{\alpha=1}^{N_C}\sum_{k=1}^{2}|e(\mathbf{x}_{C,\alpha}^{(k)})|\omega(\mathbf{x}_{C,\alpha}^{(k)})n_x+\sum_{\beta=1}^{N_I}\sum_{\ell=1}^{2}|e(\mathbf{x}_{I,\beta}^{(\ell)})|\omega(\mathbf{x}_{I,\beta}^{(\ell)})n_x\right)\\
				&+v_x^*\mathcal{L}(\rho^{\text{tot}}),
			\end{aligned}
		\end{equation*}
		where the first equality follows from the EC property (see Definition \ref{def:EC property}) of the reconstruction algorithm and from the property of the Riemann solver illustrated in Remark \ref{rem:multiriemann}. Moreover, by Lemma \ref{lem:gauss_exact}, we have
		\begin{equation*}
			\begin{aligned}
				&\left(\sum_{\alpha=1}^{N_C}\sum_{k=1}^{2}|e(\mathbf{x}_{C,\alpha}^{(k)})|\omega(\mathbf{x}_{C,\alpha}^{(k)})n_x+\sum_{\beta=1}^{N_I}\sum_{\ell=1}^{2}|e(\mathbf{x}_{I,\beta}^{(\ell)})|\omega(\mathbf{x}_{I,\beta}^{(\ell)})n_x\right)\\
				=&\int_{\partial{\tilde{I}}}n_x\mathrm{d}S=\int_{\tilde{I}}\nabla\cdot(1,0)\mathrm{d}x\mathrm{d}y=0,
			\end{aligned}
		\end{equation*}
		hence $\mathcal{L}(m_x^\text{tot})=v_x^*\mathcal{L}(\rho^{\text{tot}})$. A similar argument establishes that $\mathcal{L}(m_y^\text{tot})=v_y^*\mathcal{L}(\rho^{\text{tot}}).$
		\item[(b)] To prove $\mathcal{L}(E^{\text{tot}})=\frac{|\mathbf{v}|^2}{2}\mathcal{L}(\rho^{\text{tot}})+\frac{p^*+\gamma_h B_h}{\gamma_h-1}\int_{\tilde{\Gamma}}(\mathbf{v^*}\cdot\mathbf{n})\mathrm{d}S,$ we again use equations \eqref{equ:semidiscre-totalU}, \eqref{equ:multimedia flux} and \eqref{equ:LLF flux} to obtain 
		\begin{equation}
			\begin{aligned}
				\mathcal{L}(E^\text{tot})
				=&-\sum_{\alpha=1}^{N_C}\sum_{k=1}^{2}|e(\mathbf{x}_{C,\alpha}^{(k)})|\frac{\omega(\mathbf{x}_{C,\alpha}^{(k)})}{2}\left(\left(E^{\mathrm{int}}_{C,\alpha,k}+E^{\mathrm{ext}}_{C,\alpha,k}\right)(\mathbf{v^*}\!\cdot\!\mathbf{n})\right.\\
				&\left.-a_{\max}\left(E^{\mathrm{ext}}_{C,\alpha,k}-E^{\mathrm{int}}_{C,\alpha,k}\right)+2p^*(\mathbf{v^*}\!\cdot\!\mathbf{n})\right)-\sum_{\beta=1}^{N_I}\sum_{\ell=1}^{2}|e(\mathbf{x}_{I,\beta}^{(\ell)})|\omega(\mathbf{x}_{I,\beta}^{(\ell)})p^*(\mathbf{v^*}\!\cdot\!\mathbf{n})\\
				=&-p^*\left(\sum_{\alpha=1}^{N_C}\sum_{k=1}^{2}|e(\mathbf{x}_{C,\alpha}^{(k)})|\omega(\mathbf{x}_{C,\alpha}^{(k)})(\mathbf{v^*}\cdot\mathbf{n})+\sum_{\beta=1}^{N_I}\sum_{\ell=1}^{2}|e(\mathbf{x}_{I,\beta}^{(\ell)})|\omega(\mathbf{x}_{I,\beta}^{(\ell)})(\mathbf{v^*}\cdot\mathbf{n})\right)\\
				&\!-\!\sum_{\alpha=1}^{N_C}\!\sum_{k=1}^{2}|e(\mathbf{x}_{C\!,\alpha}^{(k)})|\frac{\omega(\mathbf{x}_{C\!,\alpha}^{(k)})}{2}\!\left(\left(E^{\mathrm{int}}_{C\!,\alpha,k}\!+\!E^{\mathrm{ext}}_{C\!,\alpha,k}\right)\!(\mathbf{v^*}\!\cdot\!\mathbf{n})\!-\!a_{\max}\!\left(E^{\mathrm{ext}}_{C\!,\alpha,k}\!-\!E^{\mathrm{int}}_{C\!,\alpha\!,k}\right)\right)\\
				=&-p^*\int_{\partial{\tilde{I}}}(\mathbf{v^*}\!\cdot\!\mathbf{n})\mathrm{d}S+\frac{|\mathbf{v}^*|^2}{2}\mathcal{L}(\rho^\text{tot})-\frac{p^*+\gamma_h B_h}{\gamma_h-1}\int_{\partial{\tilde{I}}-\tilde{\Gamma}}(\mathbf{v^*}\!\cdot\!\mathbf{n})\mathrm{d}S.
			\end{aligned}
		\end{equation}
		Since $$\int_{\partial{\tilde{I}}}(\mathbf{v^*}\!\cdot\!\mathbf{n})\mathrm{d}S=\int_{\tilde{I}_i}\nabla\cdot(v_x,v_y)\mathrm{d}x\mathrm{d}y=0,$$
		we have $\mathcal{L}(E^\text{tot})=\frac{|\mathbf{v}^*|^2}{2}\mathcal{L}(\rho^\text{tot})+\frac{p^*+\gamma_h B_h}{\gamma_h-1}\int_{\tilde{\Gamma}}(\mathbf{v^*}\!\cdot\!\mathbf{n})\mathrm{d}S$. This completes the proof of the lemma.
	\end{description}
\end{proof}

Next, we introduce the increment operator $\delta(\cdot)$ to denote the change of a quantity (in a given cell) from time $t^n$ to $t^{n+1}$. For instance, $\delta(\rho^{\text{tot}})=\rho^{\text{tot}}(t^{n+1})-\rho^{\text{tot}}(t^{n})$. With this definition, we can prove the following lemma:
\begin{lemma}\label{lem:delta}
	$\delta(xy)=x\delta(y)+y\delta(x)+\delta(x)\delta(y).$
\end{lemma}
\begin{proof}
	$\delta(xy)=(x+\delta x)(y+\delta y)-xy=x\delta(y)+y\delta(x)+\delta(x)\delta(y)$.
\end{proof}

We now have sufficient machinery to prove Theorem \ref{thm:eulerforward pe}.

\begin{proof}[Proof of Theorem~\ref{thm:eulerforward pe}]
	Fixed Cartesian cells can be viewed simply as a special ``merged cell", hence it is sufficient to analyze only the general merged-cell case in the proof of Theorem~\ref{thm:eulerforward pe}. Assume that at $t=t^n$, the cell-averaged conserved variables in the neighboring cells of a merged cell $\tilde{I}$ exhibit a constant velocity $\mathbf{v}^*=(v_x^*,v_y^*)$ and a constant pressure $p^*$. Then, after one time step, the cell volume becomes 
	\begin{equation}\label{equ:L(M) euler}
		\mathcal{M}^{\text{evo},n+1}_{0,0}=\mathcal{M}^{\text{evo},n}_{0,0}+\Delta t\mathcal{L}(\mathcal{M}^{\text{evo},n}_{0,0}),
	\end{equation} and the conserved variables become 
	\begin{equation}\label{equ:L(U) euler}
		\mathbf{U}^{\text{tot},n+1}_h=\mathbf{U}^{\text{tot},n}_h+\Delta t\mathcal{L}(\mathbf{U}^{\text{tot},n}_h).
	\end{equation}
	Here, the spatial discretization operator $\mathcal{L}$ is defined in \eqref{equ:semidisc_M} and \eqref{equ:semidiscre-totalU}. We now prove that, $\overline{\mathbf{U}}^{n+1}_h$ in $\tilde{I}_i$ remains at pressure $p^*$ and velocity $\mathbf{v}^*$.
	\begin{description}
		\item[(1)] First, we prove that the velocity remains $\mathbf{v}^*$. Since $m_x^{\text{tot}}=\mathcal{M}_{0,0}^{\text{evo},n}\cdot(\overline{\rho}v_x)=\rho^{\text{tot}}v_x$, by applying Lemma \ref{lem:delta} and \eqref{equ:L(U) euler}, we have 
		\begin{equation}
			\delta(v_x)=\frac{\delta{m_x^{\text{tot}}}-v_x\delta{\rho^{\text{tot}}}}{\rho^{\text{tot}}+\delta{\rho^{\text{tot}}}}
			=\Delta t \frac{\mathcal{L}({m_x^{\text{tot}}})-v_x\mathcal{L}({\rho^{\text{tot}}})}{\rho^{\text{tot}}+\mathcal{L}({\rho^{\text{tot}}})}.
		\end{equation} 
		Combined with equation \eqref{equ:L(U)} in Lemma \ref{lem:L(U)}, we obtain $\delta(v_x)=0$. Similarly, $\delta(v_y)=0$, so the velocity remains unchanged after one time step.
		
		\item[(2)] Next, we prove that the pressure remains $p^*$. From equations \eqref{equ:Ep} and \eqref{equ:EOS}, we deduce that 
		\begin{equation}
			\begin{aligned}
				\delta(p)&=\delta\left((\gamma_h-1)\left(\overline{E}-\frac12\overline{\rho}|\mathbf{v}^*|^2\right)-\gamma_h B_h\right)\\
				&\xlongequal{\delta{v_x}=\delta{v_y}=0}(\gamma_h-1)\left(\delta{\overline{E}}-\frac{|\mathbf{v}^*|^2}{2}\delta{\overline{\rho}}\right)\\
				&\xlongequal{\text{Lemma \ref{lem:delta}}}\frac{\gamma_h-1}{\mathcal{M}_{0,0}^{\text{evo},n}+\delta(\mathcal{M}_{0,0}^{\text{evo},n})}\left(\delta{E^{\text{tot}}}-\overline{E}\delta(\mathcal{M}_{0,0}^{\text{evo},n})\right.\\
				&~~~~~~~~~~~~~~~~\left.-\frac{|\mathbf{v}^*|^2}{2}\delta{\rho^{\text{tot}}}+\frac{|\mathbf{v}^*|^2}{2}\overline{\rho}\delta(\mathcal{M}_{0,0}^{\text{evo},n})\right)\\
				&\xlongequal{\text{Lemma \ref{lem:L(U)}}}\frac{\gamma_h-1}{\mathcal{M}_{0,0}^{\text{evo},n}+\delta(\mathcal{M}_{0,0}^{\text{evo},n})}\left(\Delta t\left(\frac{p^*+\gamma_h B_h}{\gamma_h-1}\int_{\tilde{\Gamma}}(\mathbf{v}^*\!\cdot\!{\mathbf{n}})\mathrm{d}S\right)\right.\\
				&~~~~~~~~~~~~~~~~\left.-\left(\overline{E}-\frac{|\mathbf{v}^*|^2}{2}\overline{\rho}\right)\delta(\mathcal{M}_{0,0}^{\text{evo},n})\right)\\
				&\xlongequal{\text{EOS \eqref{equ:EOS}}}\frac{p^*+\gamma_h B_h}{\mathcal{M}_{0,0}^{\text{evo},n}+\delta(\mathcal{M}_{0,0}^{\text{evo},n})}\left(\Delta t\int_{\tilde{\Gamma}}(\mathbf{v}^*\!\cdot\!{\mathbf{n}})\mathrm{d}S-\delta(\mathcal{M}_{0,0}^{\text{evo},n})\right)\\
				&=\frac{(p^*+\gamma_h B_h)\Delta t}{\mathcal{M}_{0,0}^{\text{evo},n}+\delta(\mathcal{M}_{0,0}^{\text{evo},n})}\left(\int_{\tilde{\Gamma}}(\mathbf{v}^*\!\cdot\!{\mathbf{n}})\mathrm{d}S-\mathcal{L}({\mathcal{M}_{0,0}^{\text{evo},n}})\right).
			\end{aligned}
		\end{equation} 
		By utilizing Eq.~\eqref{equ:semidisc_M} and Lemma \ref{lem:gauss_exact}, we have 
		\begin{equation*}
			\begin{aligned}
				\mathcal{L}({\mathcal{M}_{0,0}^{\text{evo},n}})=&\sum_{\alpha=1}^{N_C}\sum_{k=1}^{2}|e(\mathbf{x}_{C,\alpha}^{(k)})|\omega(\mathbf{x}_{C,\alpha}^{(k)})(\mathbf{v}^*\!\cdot\!{\mathbf{n}})+\sum_{\beta=1}^{N_I}\sum_{\ell=1}^{2}|e(\mathbf{x}_{I,\beta}^{(\ell)})|\omega(\mathbf{x}_{I,\beta}^{(\ell)})(\mathbf{v}^*\!\cdot\!{\mathbf{n}})\\
				=&\int_{\tilde{\Gamma}}(\mathbf{v}^*\!\cdot\!{\mathbf{n}})\mathrm{d}S.
			\end{aligned}
		\end{equation*}
		Hence, $\delta(p)=0$, the pressure remains $p^*$.
	\end{description}
	Since both the velocity and the pressure remain unchanged after one time step, $
	\overline{\mathbf{U}}_{h}^{\,n+1}\in\mathcal{G}(\mathbf{v}^*,p^*;\gamma_h, B_h)$, the proof is complete.
\end{proof}

\begin{remark}
	The proof idea of Theorem~\ref{thm:eulerforward pe} extends	directly to far more general configurations---be they unstructured (but stationary) meshes, or more general moving-mesh	methods. In particular, it follows immediately that fixed unstructured-mesh schemes solving Eq.~\eqref{equ:euler}, when utilizing an EC reconstruction, will naturally preserve pressure equilibrium.
\end{remark}

\subsection{Proof of Lemma~\ref{lem:pe_2}}
\label{sec:red_pe}
	Consider the configuration in Figure \ref{fig:face_rec}, assume the
	merged cell
	\(
	\tilde I_{1}=I_{i,j}^{1}\cup I_{i+1,j}^{1}
	\)
	is surrounded by neighbors that share the uniform velocity
	\(\mathbf v^\ast\) and pressure \(p^\ast\).
	The EC reconstruction yields the vector polynomial
	\[
	\mathbf U(x,y)\;\approx\;\mathbf p(x,y)
	\;=\;
	\sum_{s+r\le N_p}\mathbf a_{s,r}\,x^{s}y^{r},
	\qquad
	\mathbf a_{s,r}\in\mathbb R^{4}.
	\]
	Because the reconstruction is EC,
	\(
	\mathbf p(x,y)\in\mathcal G(p^\ast,\mathbf v^\ast;\gamma_h,B_h)
	\)
	for all \((x,y)\). Hence
	\(
	\mathbf a_{0,0}\in\mathcal G(p^\ast,\mathbf v^\ast;\gamma_h,B_h)
	\)
	and, for every \(s+r>0\),
	\begin{equation*}
		\mathbf a_{s,r}\in\mathcal H(p^\ast,\mathbf v^\ast):=\;
		\left\{\,\left(\rho,(\rho\mathbf v^\ast)^\top,\tfrac12\rho|\mathbf v^\ast|^2\right)^\top
		\;\middle|\;\rho>0\right\}.
	\end{equation*}
	Dividing Eq.~\eqref{equ:redistribute_1} by the reconstructed volume
	\(
	\mathcal M_{0,0}^{\mathrm{rec}}(I_{i,j}^{1})
	\)
	gives the post-redistribution cell average
	\begin{equation*}
		\overline{\mathbf U}_h(I_{i,j}^{1})
		=\frac{\mathbf U^{\text{tot}}_h(I_{i,j}^{1})}
		{\mathcal M_{0,0}^{\mathrm{rec}}(I_{i,j}^{1})}
		=\mathbf a_{0,0}
		+\sum_{s+r>0}
		\frac{\mathcal M_{s,r}^{\mathrm{rec}}(I_{i,j}^{1})}
		{\mathcal M_{0,0}^{\mathrm{rec}}(I_{i,j}^{1})}\,
		\mathbf a_{s,r}.
	\end{equation*}
	Since \(\mathbf a_{0,0}\in\mathcal G(p^\ast,\mathbf v^\ast;\gamma_h,B_h)\), it follows that
	\(
	\overline{\mathbf U}(I_{i,j}^{1})\in
	\mathcal G(p^\ast,\mathbf v^\ast;\gamma_h,B_h).
	\)
	The same argument applies to $I_{i+1,j}^{1}$.  Hence redistribution preserves pressure equilibrium for both cut cells $I_{i,j}^{1}$ and $I_{i+1,j}^{1}$. The proof for a general merged‐cell geometry proceeds identically.
\section{EC reconstruction: Third-order EC-MRWENO reconstruction method for cut-cell meshes}\label{sec:WENO}
As illustrated in Subsection~\ref{subsec: EC1}, an equilibrium-compatible (EC) reconstruction is crucial for both GPR property and pressure‐equilibrium preservation. In this section, we analyze the
polynomial reconstruction for the conserved variables and establish a
general theoretical framework for its EC property.  Building on this
foundation, we introduce a simple modification that endows general WENO reconstructions---including component-wise variants
\cite{MARTI20141,LI2023127583}---with the EC property. Finally, we apply
this modification to the third-order multi-resolution WENO (MRWENO)
reconstruction \cite{ZHU201919}, yielding a third‐order
EC‐MRWENO reconstruction tailored to cut‐cell meshes.

\subsection{Direct polynomial reconstruction naturally satisfies the EC Property}\label{sec:direct_poly_EC}
In this subsection, we prove that direct polynomial reconstruction inherently satisfies the EC property.

For notational simplicity, we hereafter assume that every interface cell is in its merged (but not yet redistributed) state and make no distinction between a merged cell \(\tilde I_i\) and a standard Cartesian cell \(I_{i,j}\): we uniformly index all cells as
\(
I_i \quad (1 \le i \le N),
\)
where \(N\) is the total number of Cartesian cells and merged cells.  Let
\(\mathcal{T}(t^n) = \{I_i\}_{i=1}^N\)
denote the partition of the computational domain at $t=t^n$. Note that, as the interface position evolves over time, both the number of cells and their geometric configurations may change from one time step to the next. We denote by $\overline{U}_{i,l}:=(\overline{\mathbf{U}}_i)_l$ the $l$th component of the cell-averaged conserved variables $\overline{\mathbf{U}}_i=\frac{\mathbf{U}^\text{tot}_h(I_i)}{\mathcal{M}_{0,0}(I_i)}$ in cell $I_i$. Given a stencil $S\subset\mathcal{T}(t^n)$ and a polynomial space $\mathbb{P}_k$ (the space of polynomials of total degree at most $k$), a direct polynomial reconstruction in cell $I_{i_0}$ is an operator that maps the data $\left\{\overline{U}_{i,l}:~i\in S\right\}$ to a polynomial $p_l\in\mathbb{P}_k$ satisfying the following conditions:

\begin{enumerate}
	\item {\textbf{Exact conservation on a designated subset.}} There exists a subset $S_0\subset S$ satisfying $i_0\in S_0$ such that
	\begin{equation}\label{equ: recon_conserv}
		\frac{1}{\mathcal{M}_{0,0}(I_i)}\int_{I_i}p_l(x,y)\mathrm{d}x\mathrm{d}y=\overline{U}_{i,l}\text{  for every  }i\in S_0,
	\end{equation}
	where $\mathcal{M}_{0,0}(I_i)=|I_i|$ is the volume cell $I_i$ (See Eq.~\eqref{def:geo moment M}).
	
	\item {\textbf{Optimal fit over the stencil.}} 
	When the number of cells in $S$ exceeds the number of degrees of freedom in $\mathbb{P}_k$ (i.e., the system is overdetermined), the polynomial $p_l(x,y)$ is determined by minimizing
	\begin{equation}\label{equ: recon_mini}
		\sum_{i\in S \backslash S_0}\alpha_i\left|\frac{1}{\mathcal{M}_{0,0}(I_i)}\int_{I_i}p_l(x,y)\mathrm{d}x\mathrm{d}y-\overline{U}_{i,l}\right|^q,
	\end{equation}
	where $\alpha_j>0$ are user-defined weights and $q\geq 1$ (commonly $q=2$).
\end{enumerate}

Direct polynomial reconstruction satisfies the following property:

\begin{lemma}\label{lem:recon_linear}
	Suppose $p_l(x,y)$ is the polynomial directly reconstructed from the cell-averaged data $\left\{\overline{U}_{i,l}:~j\in S\right\}$. Then, for any constants $c_1$ and $c_2$, $c_1p_l(x,y)+c_2$ is the reconstruction polynomial corresponding to the scaled and shifted data $\left\{c_1\overline{U}_{i,l}+c_2:~i\in S\right\}$.
\end{lemma}

\begin{proof}
	Since $p_l(x,y)$ satisfies the condition \eqref{equ: recon_conserv} and minimizes \eqref{equ: recon_mini}, and noting that $\int_{I_i}\mathrm{d}x\mathrm{d}y=\mathcal{M}_{0,0}^i,$
	applying the linear transformation $c_1(\cdot)+c_2$ uniformly to both $p_l(x,y)$ and the data $\left\{\overline{U}_{i,l}:~i\in S\right\}$ leaves the condition \eqref{equ: recon_conserv} and the minimization in \eqref{equ: recon_mini} unchanged. Thus, $c_1p_l(x,y)+c_2$ is the reconstruction corresponding to $\left\{c_1\overline{U}_{i,l}+c_2:~i\in S\right\}$.
\end{proof}

The following lemma establishes the EC property of direct polynomial reconstruction.

\begin{lemma}\label{lem:directpoly_EC}
	Suppose that the cell-averaged conserved variables $\overline{\mathbf{U}}$ over the reconstruction stencil $S$ all share the same pressure $p^*$, velocity $\mathbf{v}^*$, and material (i.e., \textbf{identical} parameters $\gamma$ and $B$ in the EOS \eqref{equ:EOS}). Then, the direct polynomial reconstruction of the conserved variable vector $\mathbf{U}$ yields a vector polynomial $\mathbf{p}\in \mathcal{G}^{\gamma,B}(\mathbf{v}^*,p^*)$. Consequently, the direct polynomial reconstruction of $\mathbf{U}$ possesses EC property. Here, the vector polynomial $\mathbf{p}$ is obtained by performing the polynomial reconstruction separately on each component of $\mathbf{U}$.
\end{lemma}

\begin{proof}
	By Definition \ref{def:set G}, all cell-averaged values $\overline{\mathbf{U}}$ in the stencil $S$ lie in $\mathcal{G}^{\gamma,B}(\mathbf{v}^*,p^*)=\left\{\mathbf{U}=\left(\rho,\rho \mathbf{v}^*,\frac{p^*+\gamma B}{\gamma-1}+\frac12\rho|\mathbf{v}^*|\right)^\top\middle|\rho>0\right\}$. Since the mapping 
	\vspace{-2.5mm}$$\rho\mapsto\left(\rho,\rho \mathbf{v}^*,\frac{p^*+\gamma B}{\gamma-1}+\frac12\rho|\mathbf{v}^*|^2\right)^\top$$
	is linear, it follows from the linearity of this mapping and Lemma \ref{lem:recon_linear} that the reconstructed polynomial $\mathbf{p}\in\mathcal{G}^{\gamma,B}(\mathbf{v}^*,p^*)$. This completes the proof.
\end{proof}

\subsection{A simple EC modification for WENO reconstructions}\label{sec:EC_WENO_modi}
WENO reconstructions constitute one of the most widely used families of
high-order reconstructions for hyperbolic conservation laws, valued for their shock-capturing capability, high resolution, and high-order accuracy \cite{shu2020essentially}.
Over the past three decades, numerous WENO variants have been proposed, each optimized for particular grids, smoothness indicators, or nonlinear‐weight formulations. Unfortunately, although characteristic-wise WENO reconstructions do possess the EC property \cite{johnsen2011treatment}, many WENO reconstructions do not. In what follows we briefly review the WENO reconstruction procedure and then present a minimal modification that equips any WENO reconstruction---including component-wise versions---with the EC property while preserving its original accuracy and non-oscillatory behavior.

A high-order WENO reconstruction for a scalar quantity \(u(x,y)\) at the point \((x_0,y_0)\) in the target cell \(I_{i_0}\) proceeds as follows.
\begin{enumerate}
	\item \textbf{Candidate polynomials.}
	Choose \(r\) stencils
	\(S_1,\dots,S_r\subset\mathcal{T}(t^n)\) (all containing \(I_{i_0}\)),
	and on each \(S_k~(k\geq1)\) construct a direct reconstruction
	polynomial \(q_k(x,y)\) (see Section \ref{sec:direct_poly_EC}).
	
	\item \textbf{Linear weights.}
	Select (possibly negative) constants \(\{\lambda_k\}_{k=1}^r\) with \(\sum_k\lambda_k=1\) so that
	\[
	q_{\mathrm{lin}}(x_0,y_0)
	=\sum_{k=1}^{r}\lambda_k\,q_k(x_0,y_0)
	\]
	exactly reproduces the desired high-order polynomial approximation at the target point \((x_0,y_0)\).
	
	\item \textbf{Nonlinear weights.}
	Compute a smoothness indicator \(\beta_k\) for each reconstruction \(q_k\), and then define nonlinear weights
	\(
	\omega_k(\beta_1,\cdots,\beta_r;\lambda_1,\cdots,\lambda_r)
	\)
	subject to the normalization
	\(\sum_{k}\omega_k=1\).
	These nonlinear weights are designed so that \(\omega_k\) rapidly approaches zero if the stencil \(S_k\) crosses a discontinuity, whereas in smooth regions the difference $\bigl|\omega_k - \lambda_k\bigr|$ vanishes at the designed order of accuracy.
	
	\item \textbf{Final value.}
	The WENO reconstruction at \((x_0,y_0)\) is
	\[
	q^{\mathrm{WENO}}(x_0,y_0)
	=\sum_{k=1}^{r}\omega_k\,q_k(x_0,y_0).
	\]
\end{enumerate}This reconstruction suppresses spurious oscillations near discontinuities while maintaining high‐order accuracy in smooth regions.  For the conserved-variable vector \(\mathbf U\), WENO can be
performed in two principal ways---\emph{characteristic‐wise} or
\emph{component‐wise} \cite{chi1997essentially}.  The former is
inherently EC \cite{johnsen2011treatment,
	qin4621573modified}, whereas in the component-wise variant the nonlinear
weights \(\omega_{k}^{(j)}\) generally differ across components;
consequently, the EC property of direct polynomial reconstruction is
lost for some component‐wise WENO reconstructions.

\medskip
\noindent\textbf{EC modification for WENO schemes via unified nonlinear weights.}
To equip a component-wise WENO reconstruction with the EC property, we adopt a unified set of nonlinear weights
for \emph{all} components of \(\mathbf{U}\):
\begin{enumerate}
	\item Compute \(\omega_k^{(j)}\) for every component \(j\).
	\item Identify
	\(\displaystyle
	j^\star=\min\left(\arg\max_{j}\max_{1\le k\le r}
	\bigl|\omega_k^{(j)}-\lambda_k\bigr|\right)
	\)
	(the index of component whose nonlinear weights deviate most from the linear
	weights).
	\item Set the unified nonlinear weights
	\(\omega_k := \omega_k^{(j^\star)}\) for
	\(k=1,\dots,r\),
	and apply them in the WENO reconstruction for every component of \(\mathbf{U}\).
\end{enumerate}
With these unified weights, the reconstruction vector can be written as
\begin{equation}\label{equ:WENO poly}
	\mathbf{q}^{\mathrm{WENO}}(x_0,y_0)
	=\sum_{k=1}^{r}\omega_k\,\mathbf{q}_k(x_0,y_0),
\end{equation}
where \(\mathbf{q}_k(x,y)\) is the direct polynomial reconstruction of
\(\mathbf{U}\) on stencil \(S_k\). Assume that every cell in the combined stencil $\bigcup_{k=1}^{r} S_k$ belongs to the \emph{same} material---hence shares the identical parameters $\gamma$ and $B$---and that all those cells have the same velocity $\mathbf{v}^*$ and pressure $p^*$. By Lemmas \ref{lem:convex} and \ref{lem:directpoly_EC} it then follows that \(\mathbf{q}^{\mathrm{WENO}}(x_0,y_0)\in\mathcal{G}^{\gamma,B}(\mathbf{v}^*,p^*)\). Consequently, the unified-weight modification equips general WENO reconstructions with the EC property.

\begin{remark}
	If the nonlinear weights are exactly affine-invariant \cite{WANG2022630}, the reconstruction is also EC. However, many widely used WENO reconstructions---including the classical WENO-JS reconstructions \cite{JIANG1996202}---lack this property. 
\end{remark}

\begin{remark}
	For a fixed point $(x_0,y_0)$, characteristic decomposition is a linear operation \cite{chi1997essentially}. Moreover, if $q=2$ in \eqref{equ: recon_mini}, the polynomial coefficients are obtained via a least‐squares fit---another linear operation. Consequently, when $q=2$, characteristic decomposition does not disturb the EC property of the direct polynomial reconstruction proved in Lemma \ref{lem:directpoly_EC}, and therefore does not compromise the EC property of the EC-modified WENO reconstruction introduced in this subsection.
\end{remark}

\subsection{EC-MRWENO reconstruction on asymmetric cut-cell meshes near material interfaces}\label{sec:EC-MRWENO}
As shown in the previous subsection, a WENO reconstruction attains EC as long as all the reconstruction stencils contain data from a \emph{single} material, and the same unified set of nonlinear weights is applied to every conserved component.
Consequently, for the cells near the material interface $\Gamma_h$, the reconstruction stencil must be restricted to cells lying on the \emph{same} side of $\Gamma_h$.

As an example, in Figure \ref{fig:face_rec} the merged cell $\tilde{I}_1$ should build its WENO stencils using \emph{only} the yellow cells; blue cells across $\Gamma_h$ are excluded. Because the merged cells skirting $\Gamma_h$ are unstructured, we employ the third-order MRWENO reconstruction \cite{ZHU201919}---widely used on unstructured grids---for all cells: both the unstructured merged cells adjacent to the interface and the Cartesian cells farther away. Implementation details of MRWENO reconstruction \cite{ZHU201919} are not repeated here, as our focus is on pressure-equilibrium preservation rather than the particulars of a WENO variant. The unified-weight modification \eqref{equ:WENO poly} is applied, yielding the EC-MRWENO reconstruction. As detailed
in Subsection~\ref{subsec:semidisc-num}, any EC reconstruction
could be substituted; here we choose EC‐MRWENO for its superior accuracy in smooth regions.

\begin{remark}
	In the spatial discretization \eqref{equ:semidiscre-totalU}, we employ a
	characteristic‐wise MRWENO reconstruction to reduce spurious oscillations. In the redistribution step
	\eqref{equ:red_p_con}--\eqref{equ:redistribute_1}, we apply a
	component‐wise MRWENO reconstruction---as suggested in \cite{zhang2025remapping}. Consequently, to ensure that
	the redistribution operation preserves pressure equilibrium, the EC modification is essential. Numerical tests
	in Test~\ref{sec:test:smooth_single} of Section~\ref{sec:numerical_tests} confirm the necessity of the EC modification.
\end{remark}

\section{Numerical Experiments}\label{sec:numerical_tests}
This section presents several numerical experiments designed to validate the
high‐order accuracy, pressure‐equilibrium preservation, and GPR property of the GPR cut‐cell scheme. For brevity, this method is hereafter referred to as ``GPR‐CC''. To demonstrate the critical role of evolved geometric moments $\mathcal M_{s,r}^\text{evo}$ in the GPR‐CC scheme, we compare it with two variants built on the same framework:
\begin{itemize}
	\item \textbf{GCL‐CC}, which evolves only the cell volume $\mathcal M_{0,0}^\text{evo}$ via
	\eqref{equ:semidisc_M}---thereby satisfying the GCL---while reconstructing all higher‐order moments
	\(\mathcal{M}_{s,r}^{\mathrm{rec}}\) for \(s+r>0\).
	\item \textbf{Con‐CC}, which uses only the reconstructed moments
	\(\mathcal{M}_{s,r}^{\mathrm{rec}}\), resulting in a purely
	conservative scheme.
\end{itemize}
To demonstrate the superior interface resolution of the GPR-CC scheme over diffusive‐interface methods, we also test the RKDG scheme from \cite{luo2021quasi} (denoted DIM‐DG) on selected problems.  For a fair comparison, the DIM‐DG employs the same LLF flux \eqref{equ:LLF flux} as used in the GPR‐CC scheme. Owing to the GPR property, level set reinitialization need only be performed periodically; by default we perform the reinitialization \cite{zheng2021high} in the GPR-CC scheme every 30 time steps. In all cut-cell methods, the CFL number is set to 0.6.

\subsection{Pure interface problem with a discontinuity of density}\label{sec:test:smooth_single}
To evaluate accuracy, pressure-equilibrium preservation and conservation, we first consider a two-material problem with initial
data
\[
(\rho,u,v,p,\gamma,B)
=
\begin{cases}
	\bigl(2+0.2\sin[\pi(x+y)],1,1,1,1.4,0\bigr),  
	& (x-0.7)^2 + (y-0.7)^2<0.3^2,\\
	\bigl(1+0.2\sin[\pi(x+y)],1,1,1,4,1\bigr), 	& (x-0.7)^2 + (y-0.7)^2>0.3^2.
\end{cases}
\]
Here the density jumps by \(1\) across the circle of radius \(0.3\). The computational domain is \([0,2]\times[0,2]\) with periodic boundary conditions, and the solution is advanced to \(T=0.3\). To avoid accuracy degradation caused by the singularity of the level‐set function \(\phi\) at the center of the circle, we initialize \(\phi\) as $\phi_0(x,y) = \frac{5}{3}\bigl(0.09 - (x-0.7)^2 - (y-0.7)^2\bigr),$
and impose exact boundary conditions for $\phi$. This choice produces a smooth, non‐singular initialization without requiring any reinitialization steps.
\subsubsection{Accuracy}
Since the numerically reconstructed interface $\Gamma_h$ does not coincide exactly
with the exact one, we measure errors in the \emph{analytic-continuation} sense:
for each material, we extend the exact solution to the entire domain and
compute the discrepancy with the numerical solution. Table \ref{tab: t1-1} summarizes the \(L^1\) and \(L^\infty\) density errors
for the Con‐CC, GCL‐CC, GPR‐CC, and GPR‐CC (non‐EC) schemes, where ``non‑EC'' indicates that the MRWENO reconstruction does not employ the EC modification proposed in Section~\ref{sec:EC_WENO_modi}. The results demonstrate that, even \emph{in the presence of a discontinuity}, the GPR-CC and GPR-CC (non-EC) schemes both achieve second‐order convergence in the vicinity of the material interface and third‐order accuracy in regions away from the interface. In contrast, both the Con-CC and GCL-CC schemes suffer severe order degradation. The second-order \(L^\infty\) accuracy of the GPR-CC scheme near the interface is limited by the straight-line interface reconstruction (Eq. \eqref{equ:lsm recon}), a restriction that has been widely reported \cite{cheng2008third}.

\begin{table}[!t]
	\centering
	\captionsetup{font=small}
	\caption{Density errors (in the analytic-continuation sense) for the various schemes in the pure interface problem (Test~\ref{sec:test:smooth_single}).}
	
	\begingroup
	\setlength{\tabcolsep}{2.6666pt} 
	\renewcommand{\arraystretch}{1.2} 
	\centering
	\footnotesize
	\label{tab: t1-1}
	\begin{tabular}{ccccccccc} 
		\bottomrule[1.0pt]
		\multicolumn{2}{c}{$N_{x}\times N_{y}$} & $80\times 80$ & $120\times 120$& $160\times 160$& $200\times 200$&$240\times 240$ & $280\times 280$ & $320\times 320$\\
		\hline
		\multirow{4}{*}{Con-CC}&$L^{1}$ error&6.30E-04&3.51E-04&2.47E-04&2.02E-04&1.67E-04&1.49E-04&1.29E-04\\
		&Order&---&1.45&1.22&0.91&1.03&0.76&1.05\\
		&$L^{\infty}$ error&1.21E-01&1.37E-01&1.07E-01&1.18E-01&1.04E-01&1.10E-01&9.81E-02\\
		&Order&---&-0.30&0.86&-0.44&0.72&-0.40&0.87\\
		\hline
		\multirow{4}{*}{GCL-CC}&$L^{1}$ error&1.97E-04&6.96E-05&3.70E-05&2.45E-05&1.74E-05&1.35E-05&1.08E-05\\
		&Order&---&2.56&2.20&1.84&1.89&1.62&1.71\\
		&$L^{\infty}$ error&3.35E-03&3.11E-03&2.39E-03&2.26E-03&2.11E-03&2.22E-03&2.21E-03\\
		&Order&---&0.19&0.91&0.26&0.36&-0.31&0.03\\
		\hline
		\multirow{4}{*}{GPR-CC}&$L^{1}$ error&1.83E-04&5.45E-05&2.33E-05&1.20E-05&7.03E-06&4.48E-06&3.01E-06\\
		&Order&---&2.99&2.95&2.96&2.95&2.93&2. 98\\
		&$L^{\infty}$ error&1.53E-03&7.24E-04&4.15E-04&2.68E-04&1.86E-04&1.48E-04&1.07E-04\\
		&Order&---&1.85&1.93&1.97&2.00&1.49&2.41\\
		\hline
		\multirow{4}{*}{\shortstack{GPR‑CC\\(non‑EC)}}&$L^{1}$ error&1.85E-04&5.50E-05&2.35E-05&1.22E-05&7.10E-06&4.53E-06&3.04E-06\\
		&Order&---&3.00&2.95&2.96&2.95&2.93&2.97\\
		&$L^{\infty}$ error&1.64E-03&7.69E-04&4.42E-04&2.85E-04&1.98E-04&1.56E-04&1.13E-04\\
		&Order&---&1.87&1.93&1.96&2.00&1.54&2.44\\
		\toprule[1.0pt]
	\end{tabular}
	\endgroup
\end{table}

Figure~\ref{fig:t1-1} compares cross‑sectional density profiles: the Con‑CC scheme exhibits noticeable errors at the interface, whereas the GCL‑CC, GPR‑CC, and GPR‑CC (non‑EC) schemes show much smaller deviations near the material interface. Figure~\ref{fig:t1-2} illustrates the interface location captured by the GPR‑CC scheme on a coarse $40\times40$ grid; even at this low resolution, the GPR‑CC scheme accurately locates the interface.

The results in Table~\ref{tab: t1-1} and Figure~\ref{fig:t1} underscore the essential role of high‑order evolved geometric moments $\mathcal{M}_{s,r}^{\mathrm{evo}}$ in achieving genuinely high‑order accuracy in cut‑cell methods.

\begin{figure}[!htb]
	\centering
	\subfloat[Zoomed-in cross-section of $\rho$]{\label{fig:t1-1}\includegraphics[height=0.2\textheight]{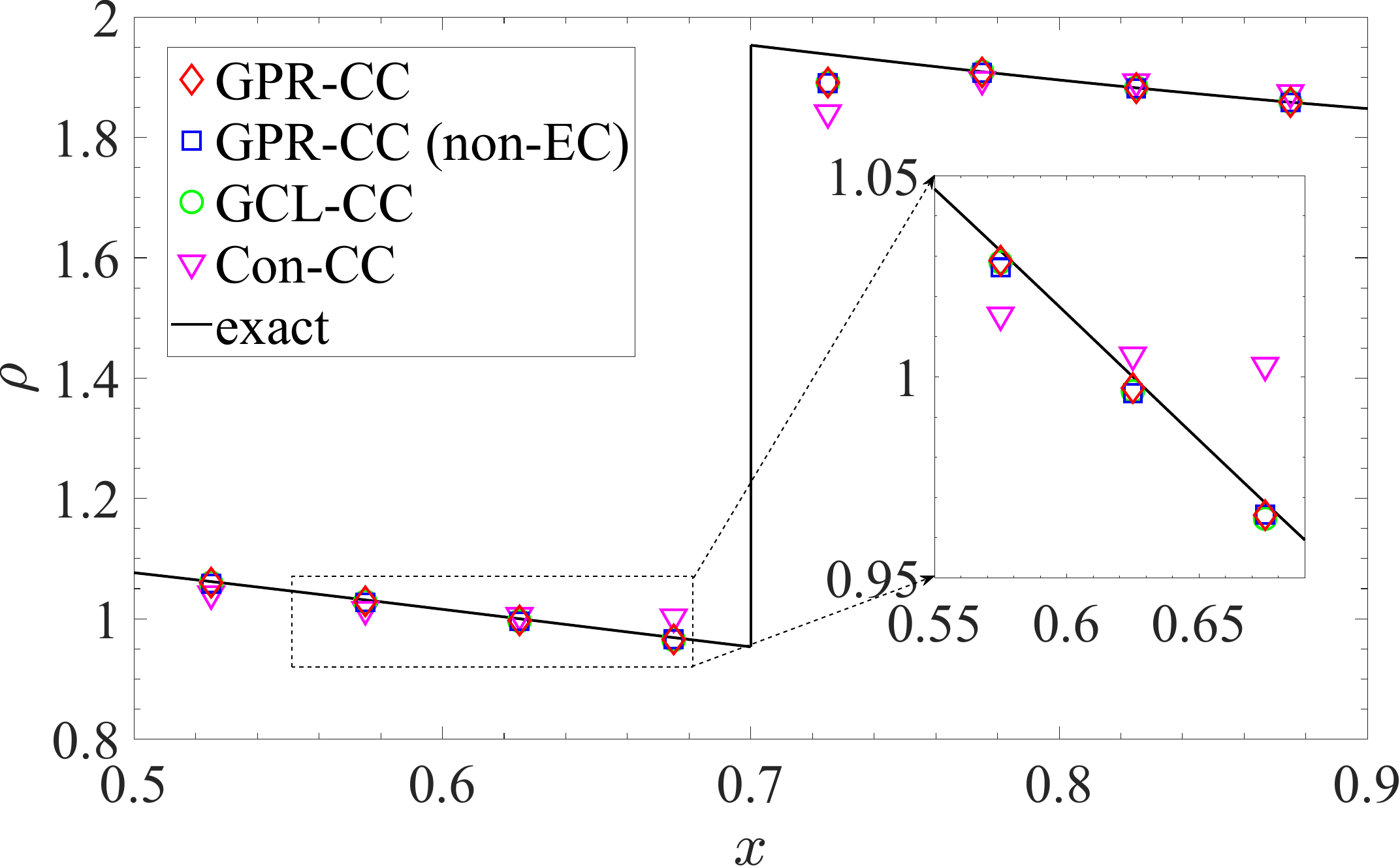}}
	\subfloat[Interface obtained by GPR-CC]{\label{fig:t1-2}\includegraphics[height=0.2\textheight]{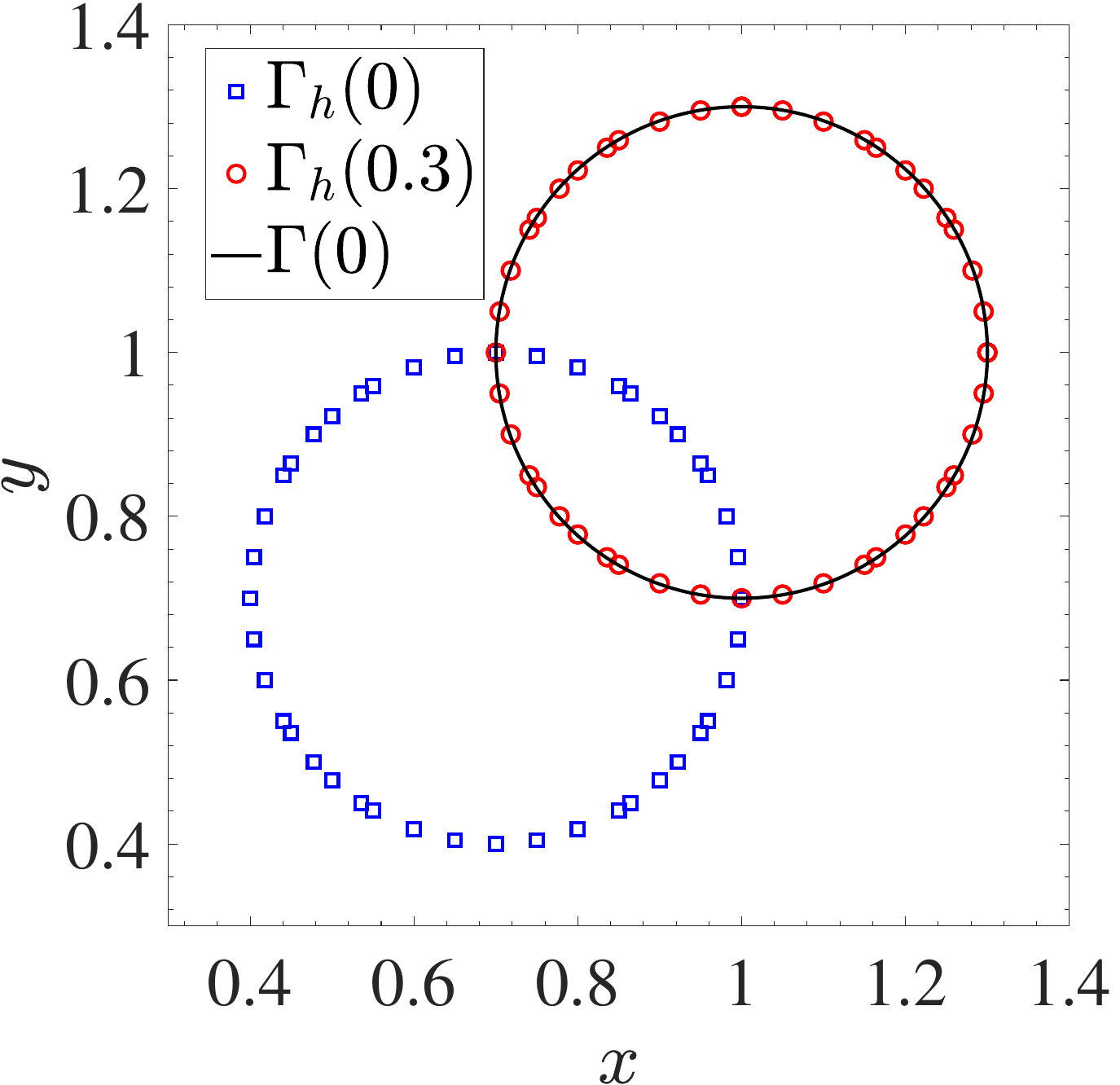}}
	\captionsetup{font=small}
	\caption{Results for Test~\ref{sec:test:smooth_single}, the pure interface problem. Left: zoomed-in cross-section of the density \(\rho\) at \(y=0.5\),
		computed by each scheme on a \(40\times40\) grid. Right: numerical
		interface \(\Gamma_h(t)\) obtained with the GPR-CC scheme versus the exact interface \(\Gamma(t)\).}
	\label{fig:t1}
\end{figure}
\subsubsection{Pressure-equilibrium-preserving}
Figure \ref{fig:t2} reports the pressure errors: pronounced oscillations occur
at the interface for the Con‐CC and GPR‐CC (non‐EC) schemes, whereas the GCL‐CC and GPR‐CC schemes exhibit pressure errors at the level of machine precision. Therefore, both the use of evolved volume \(\mathcal{M}_{0,0}^{\mathrm{evo}}\) and the EC reconstruction are indispensable for preserving pressure equilibrium near the interface.

\begin{figure}[!htb]
	\centering
	\subfloat[Con-CC]{\label{sfig:t2-1}\includegraphics[height=0.12\textheight]{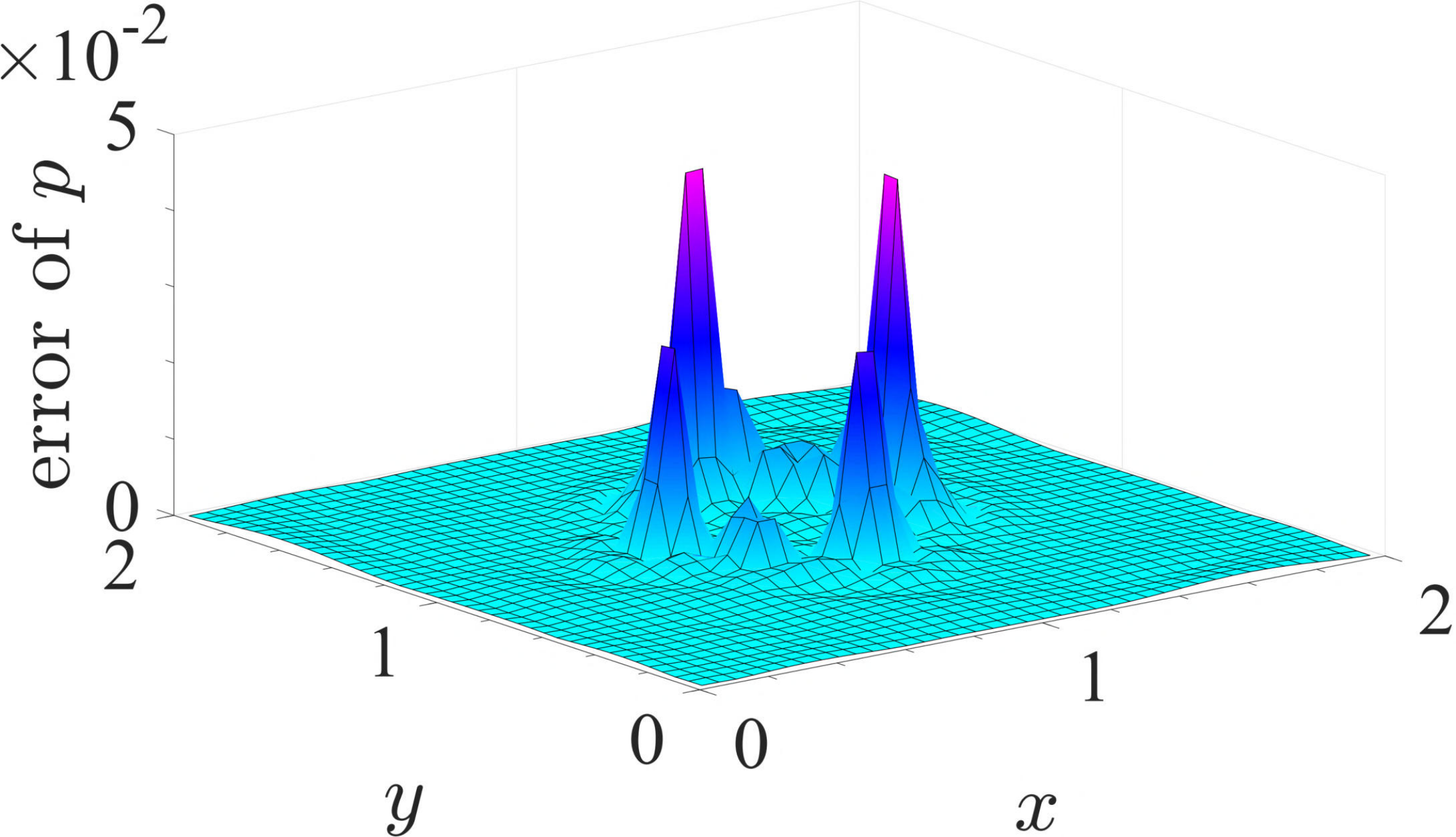}}
	\hspace{0.05\textwidth}
	\subfloat[GCL-CC]{\label{sfig:t2-2}\includegraphics[height=0.12\textheight]{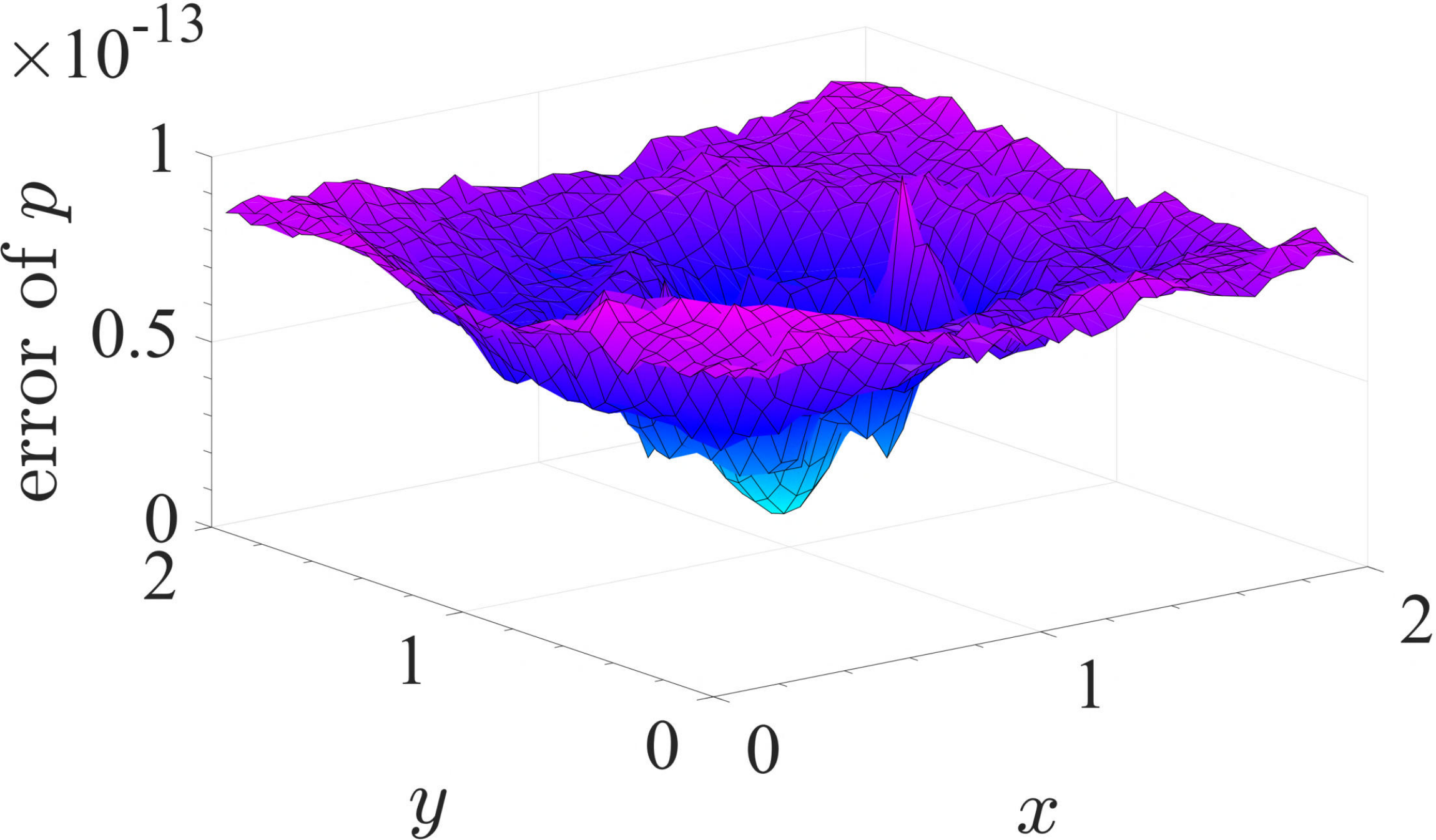}}
	\hfill
	\subfloat[GPR-CC]{\label{sfig:t2-3}\includegraphics[height=0.12\textheight]{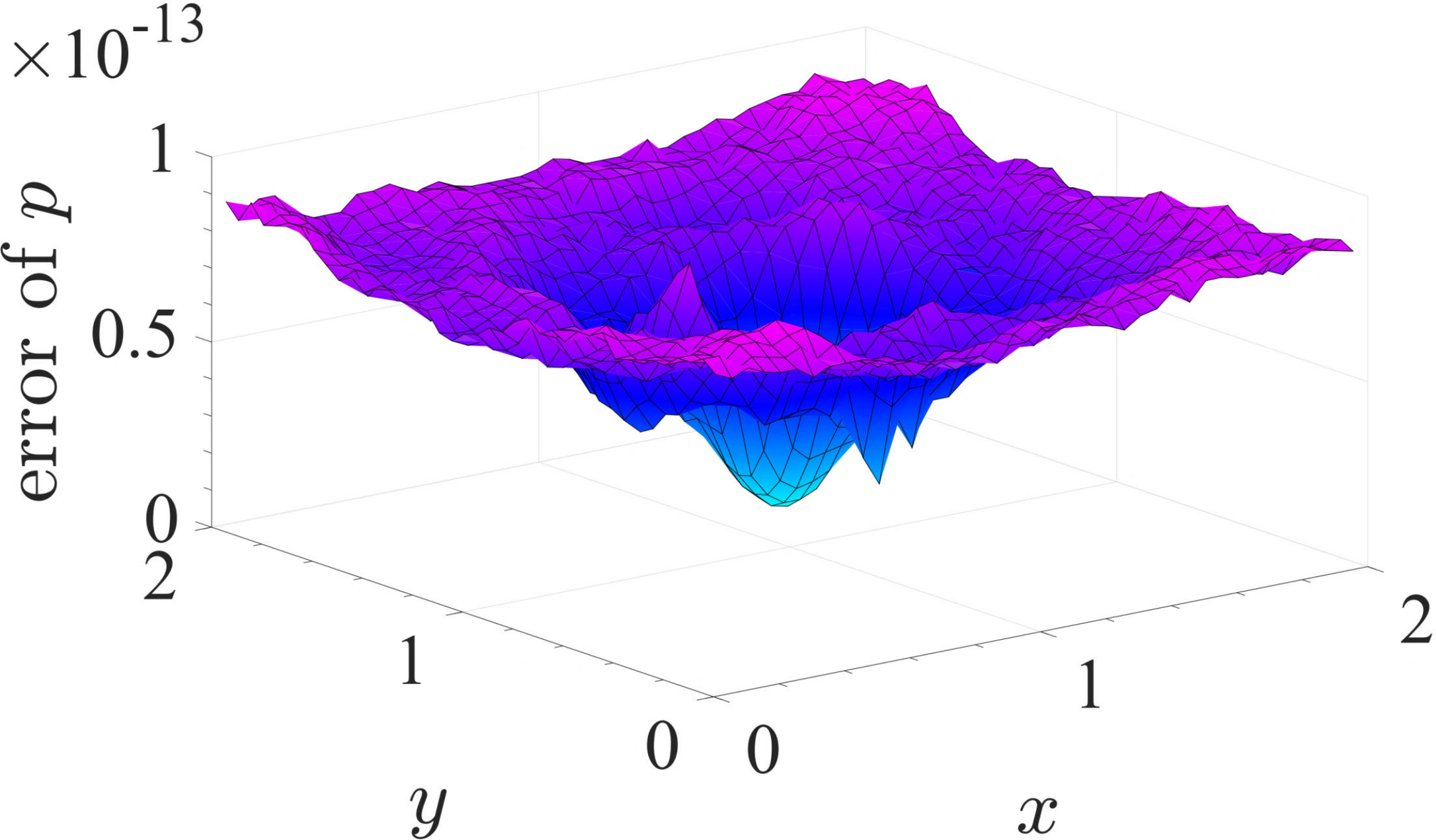}}
	\hspace{0.05\textwidth}
	\subfloat[GPR-CC (non-EC)]{\label{fig:t1-6}\includegraphics[height=0.13\textheight]{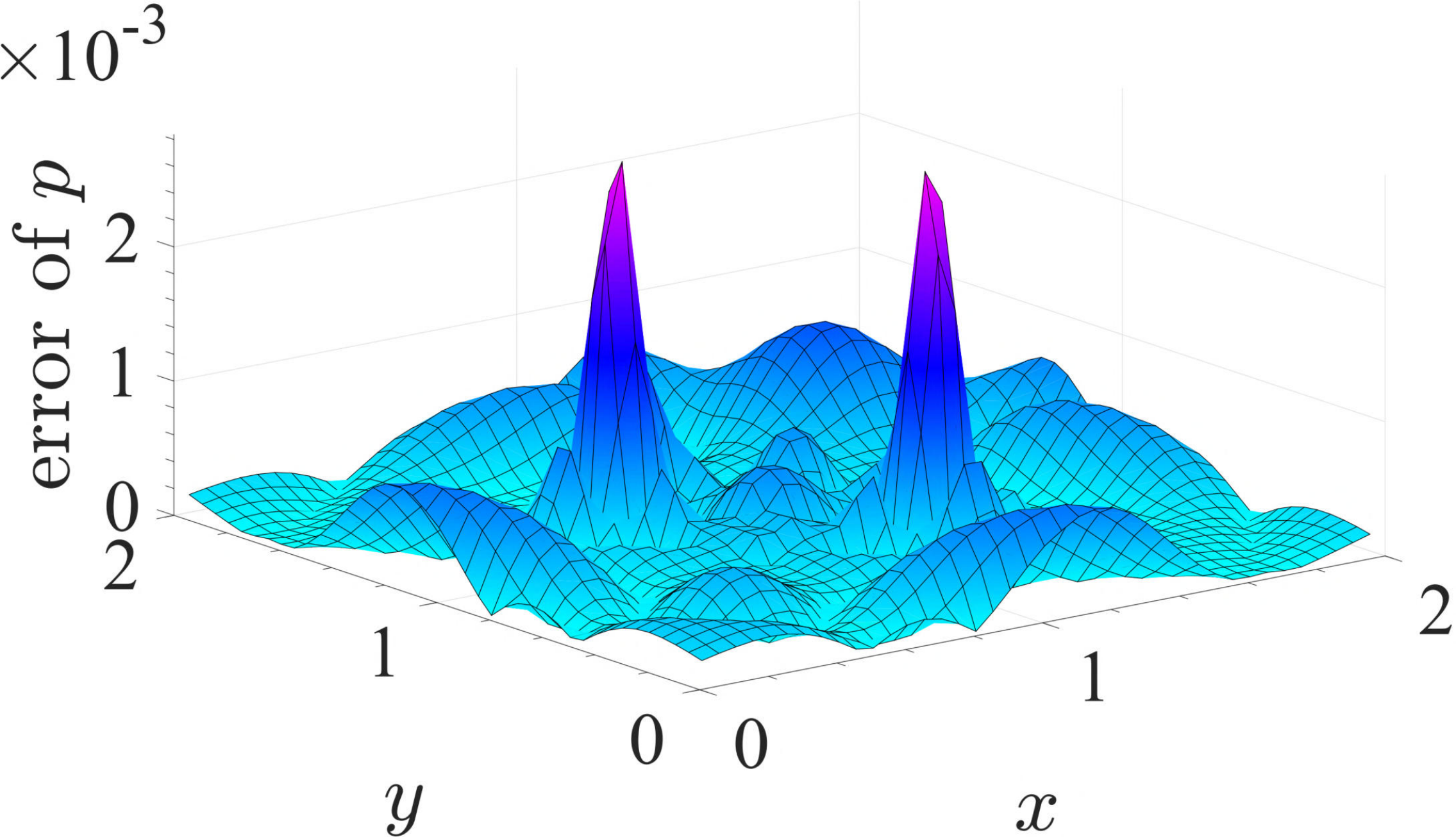}}
	\captionsetup{font=small}
	\caption{Pressure errors for each scheme in the pure interface problem (Test~\ref{sec:test:smooth_single}). \(40\times40\) grid.}
	\label{fig:t2}
\end{figure}
\subsubsection{Conservation error}
Finally, we assess the per-material mass conservation error of the GPR‑CC scheme in this test. Table \ref{tab: t1-2} shows that, on average, the per‑material mass conservation error decreases at fourth‑order as the mesh is refined, reaching negligible levels.

\begin{table}[!t]
	\vspace{5mm}
	\centering
	\captionsetup{font=small}
	\caption{Per-material mass conservation errors of the GPR‐CC scheme for the pure interface problem (Test~\ref{sec:test:smooth_single}).}

	\begingroup
	\setlength{\tabcolsep}{2.6666pt} 
	\renewcommand{\arraystretch}{1.2} 
	\centering
	\footnotesize
	\label{tab: t1-2}
	\begin{tabular}{ccccccccc} 
		\bottomrule[1.0pt]
		\multicolumn{2}{c}{$N_{x}\times N_{y}$} & $80\times 80$ & $120\times 120$& $160\times 160$& $200\times 200$ & $240\times 240$&$280\times 280$&$320\times 320$\\
		\hline
		\multirow{2}{*}{Material 1 ($\phi>0$)}& Error&3.04E-08&6.20E-09&4.31E-09&9.26E-10&3.68E-10&4.23E-10&1.44E-10\\
		&Order&---&3.92&1.27&6.89&4.20&-0.91&8.04\\
		\hline
		\multirow{2}{*}{Material 2 ($\phi<0$)}& Error&-1.51E-07&-2.55E-08&-1.31E-08&-3.00E-09&-1.84E-09&-4.10E-10&-5.48E-10\\
		&Order&---&4.38&2.31&6.62&2.68&9.75&-2.17\\
		\toprule[1.0pt]
	\end{tabular}
	\endgroup
\end{table}

\subsection{Pure interface problem with a strong density discontinuity}
\label{MGBIW}
We simulate a bubble of light gas moving in water, adapted from
Example 6.1 of \cite{lin2017simulation} on the stationary
bubble problem. The initial conditions are
\[
(\rho,u,v,p,\gamma,B)
=
\begin{cases}
	(1.2,\,1,\,1,\,1,\,1.4,\,0), \quad & \phi_0(x,y)=\sqrt{x^2+y^2}-1<0,\\
	(1000,\,1,\,1,\,1,\,4.4,\,6000),\quad & \phi_0(x,y)=\sqrt{x^2+y^2}-1>0.
\end{cases}
\]
The computational domain is \([-2,2]\times[-2,2]\) with periodic boundary
conditions, and the solution is advanced to \(T=0.1\).

\begin{figure}[!htb]
	\centering
	\subfloat{\label{sfig:t2-1-1}\includegraphics[height=0.12\textheight]{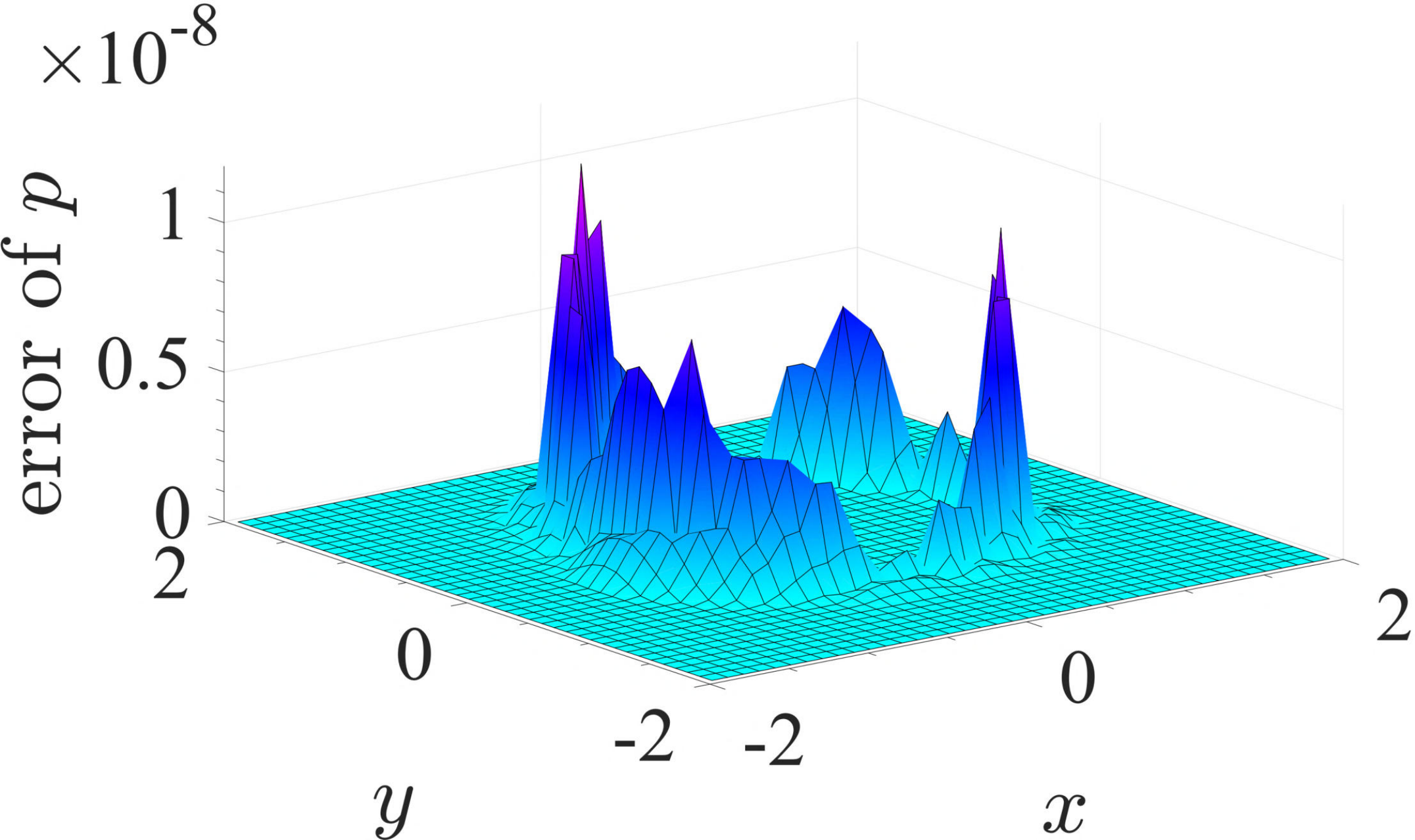}}
	\subfloat{\label{sfig:t2-1-2}\includegraphics[height=0.12\textheight]{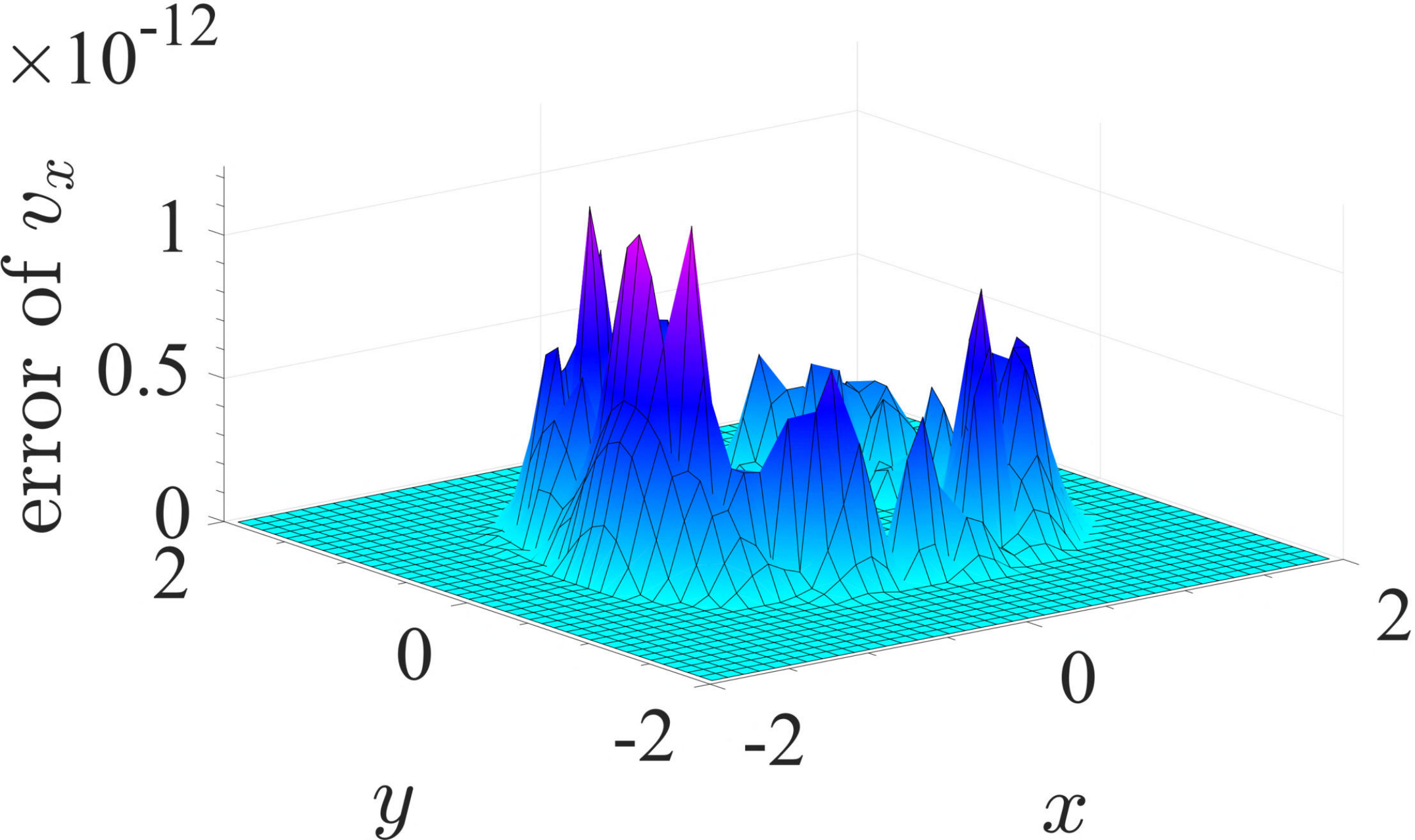}}
	\subfloat{\label{sfig:t2-1-3}\includegraphics[height=0.12\textheight]{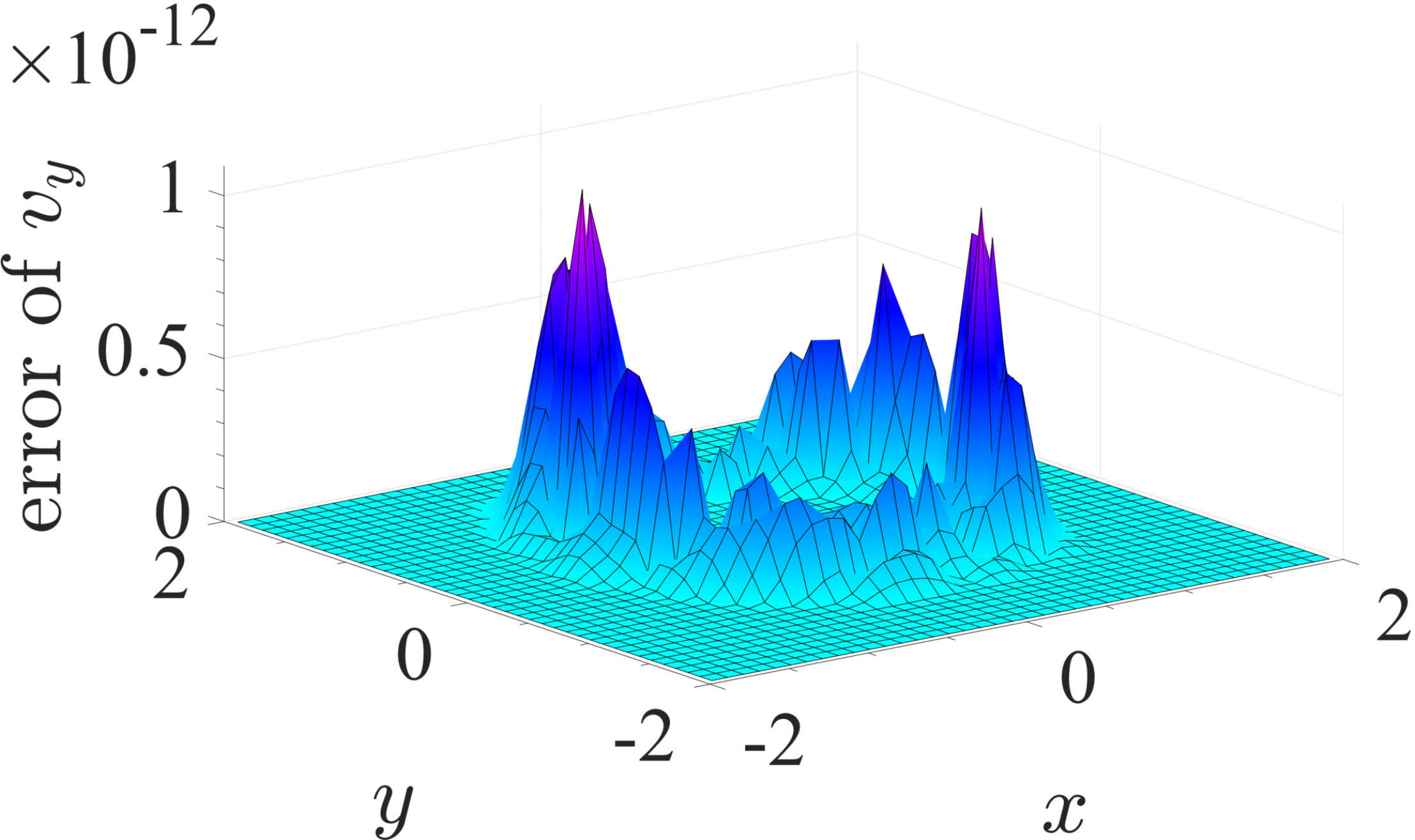}}
	\captionsetup{font=small}
	\caption{Pure‐interface test with a strong density discontinuity (Problem~\ref{MGBIW}). Errors obtained with the GPR-CC scheme. \(N_x\times N_y=40\times40\).}
	\label{fig: t2-1}
\end{figure}

Due to the strong density discontinuity, preserving pressure equilibrium in this cut‑cell test is particularly challenging. Figure~\ref{fig: t2-1} plots the errors of the GPR‐CC scheme. The GPR‐CC scheme preserves pressure equilibrium accurately, in agreement with Theorem~\ref{thm:GPR is GPR}.

\subsection{Geometric‐perturbation robustness tests: Riemann problems}

To further evaluate the GPR property of the GPR-CC scheme, we
consider two quasi‐1D Riemann problems.

\subsubsection{Weak dependence on reinitialization frequency}
\label{sec:AHST}

We first examine the air-helium shock‐tube problem \cite{qiu2007runge} with initial data
\[
(\rho,u,v,p,\gamma,B)=
\begin{cases}
	(1,\,0,\,0,\,1\times10^5,\,1.4,\,0), \quad & x<0.5,\\
	(0.125,\,0,\,0,\,1\times10^4,\,1.2,\,0), \quad & x>0.5.
\end{cases}
\]
The domain \([0,1.2]\times[-0.015,0.015]\) is discretized with \(200\times5\) cells,
and the final time is \(T=7\times10^{-4}\). Outflow boundary conditions are applied.
To assess the impact of reinitialization frequency, we compare the GPR-CC scheme against the Con-CC scheme when the level‐set is reinitialized every time
step (RN=1) versus every 200 steps (RN=200). Figure \ref{fig: t4-1} shows cross‐sectional
profiles at \(y=0\) obtained with the GPR-CC and Con-CC schemes. Because reinitialization serves to reduce interface‐reconstruction error, the GPR-CC scheme---owing to its GPR property—remains stable and accurately captures the interface structure even when reinitialization is infrequent. In contrast, Con-CC is highly sensitive to geometric errors, and its stability deteriorates rapidly under delayed reinitialization.

\begin{figure}[!htb]
	\centering
	\subfloat[Density $\rho$]{\label{sfig: t4-1}\includegraphics[height=0.162\textheight]{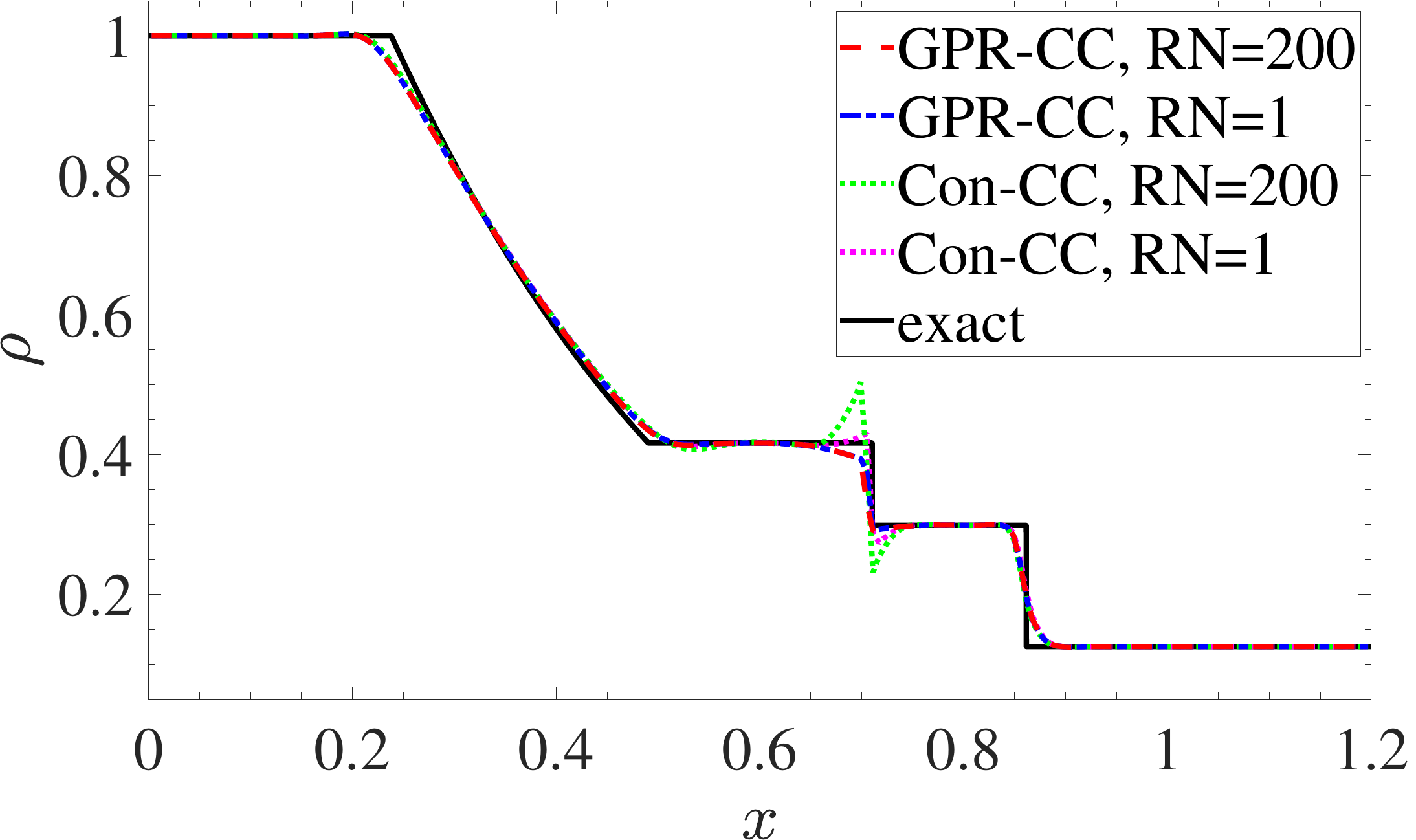}}
	\hspace{0.05\textwidth}
	\subfloat[Density $\rho$ (zoomed‐in)]{\label{sfig: t4-2}\includegraphics[height=0.165\textheight]{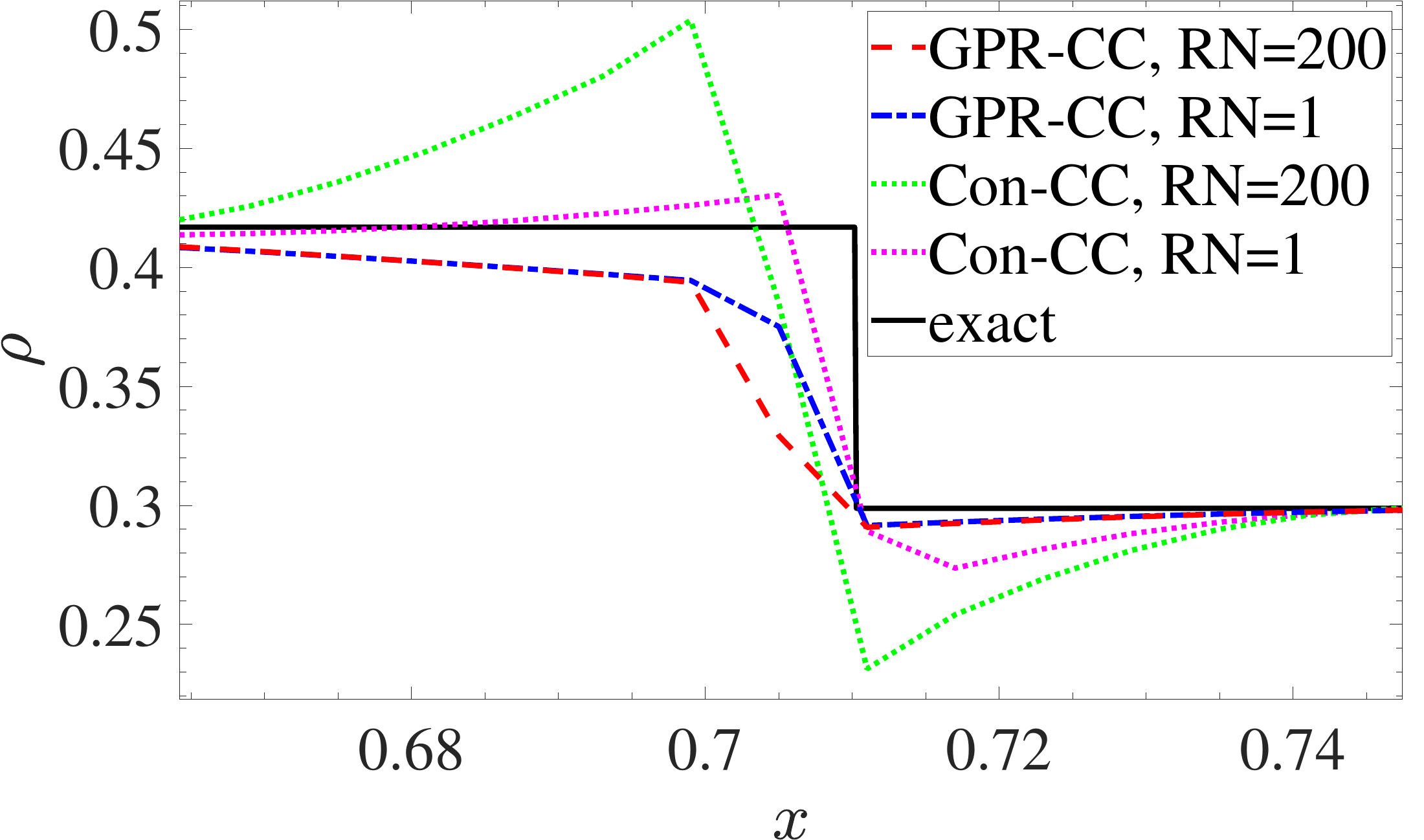}}
	\hfill
	\subfloat[Pressure $p$]{\label{sfig: t4-3}\includegraphics[height=0.17\textheight]{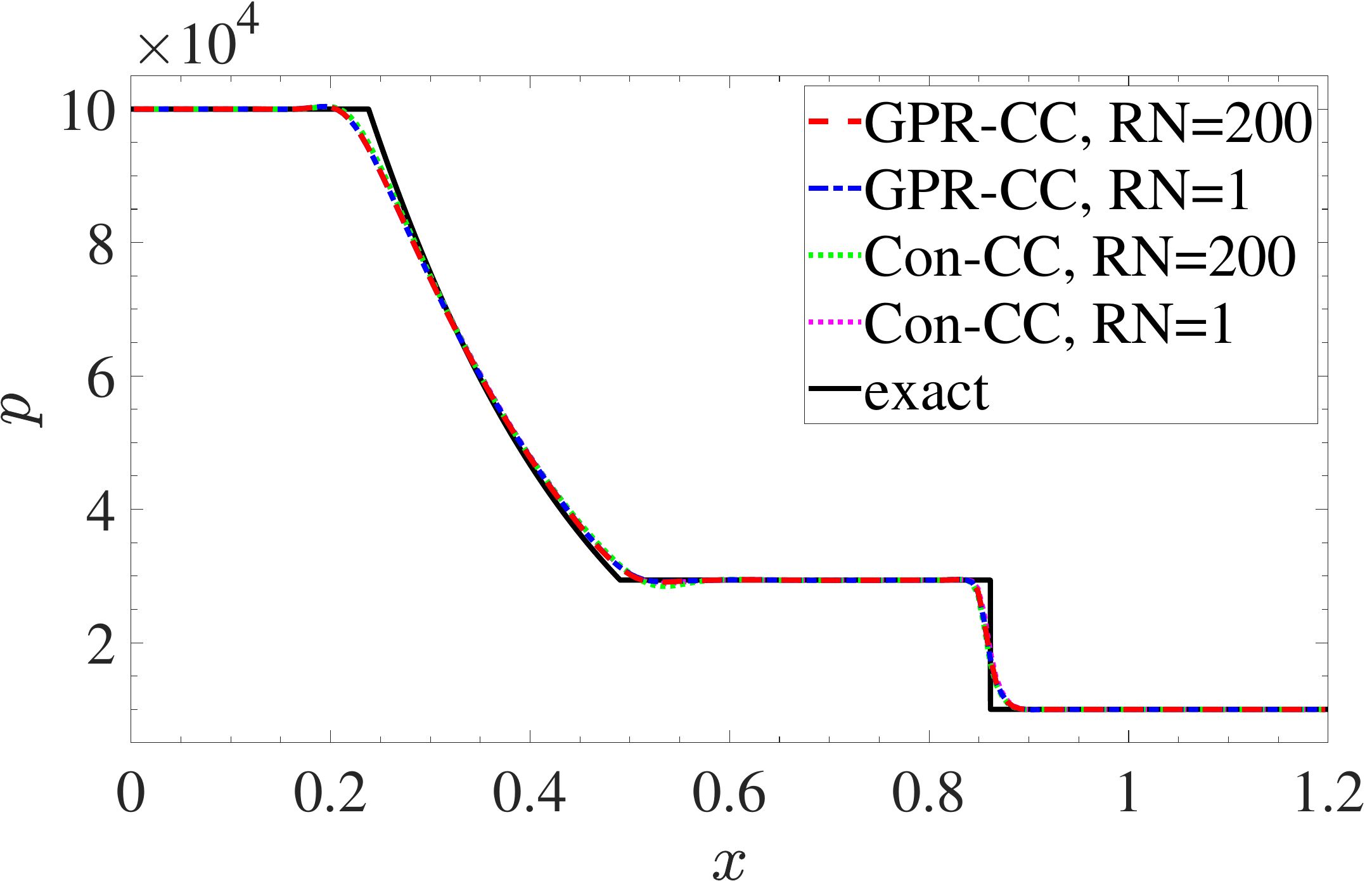}}
	\hspace{0.05\textwidth}
	\subfloat[Pressure $p$ (zoomed‐in)]{\label{sfig: t4-4}\includegraphics[height=0.17\textheight]{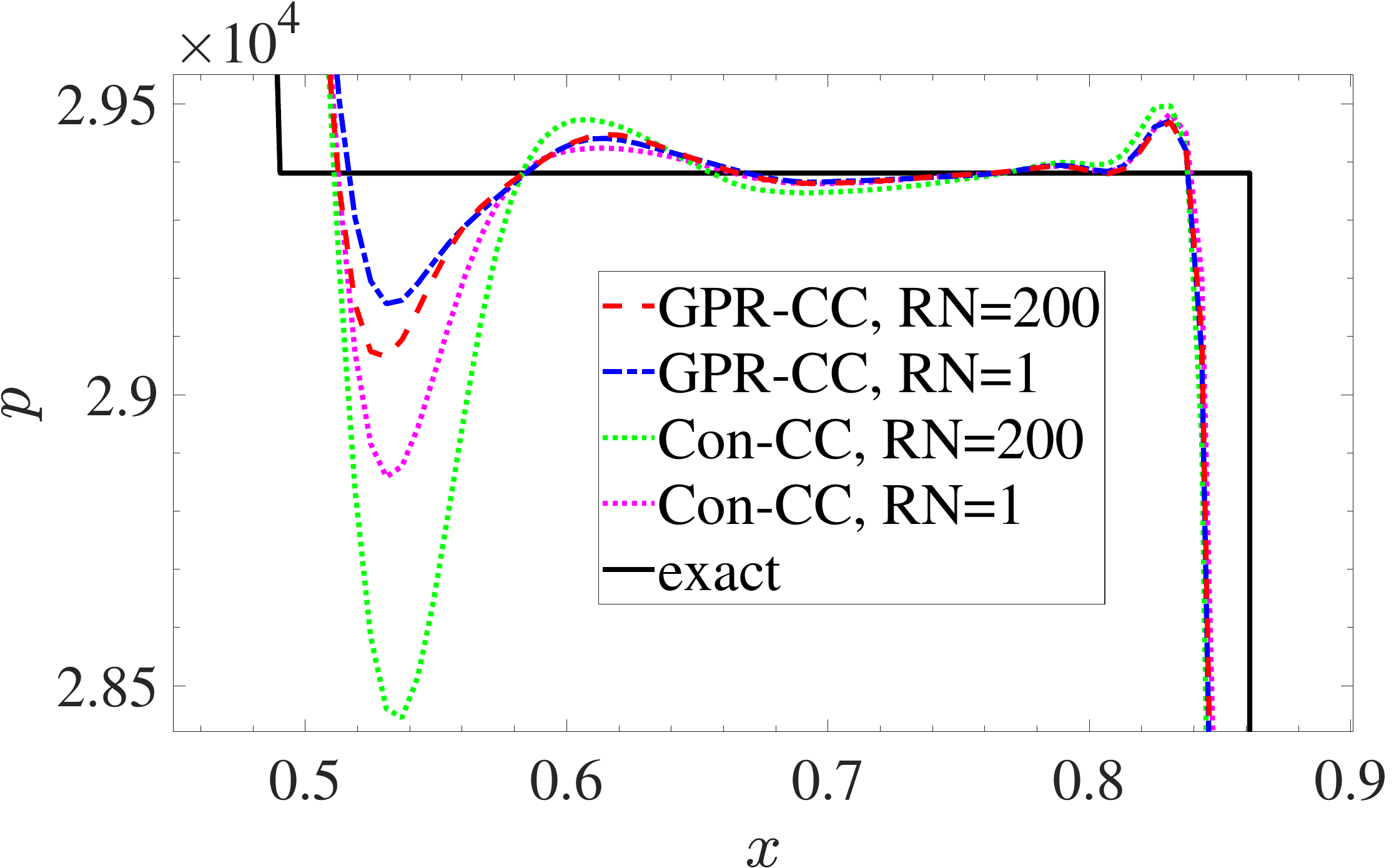}}
	\captionsetup{font=small}
	\caption{Air-helium shock‐tube problem (Problem~\ref{sec:AHST}). Cross‐sectional profiles of density and pressure at \(y=0\) for the GPR-CC and Con-CC schemes with reinitialization frequencies RN=1 and RN=200 on a \(200\times5\) grid.}
	\label{fig: t4-1}
\end{figure}

\subsubsection{Insensitivity to geometric perturbations}
\label{sec:GWRP}

Next we test the gas-water Riemann problem \cite{zheng2021high} with initial data
\[
(\rho,u,v,p,\gamma,B)=
\begin{cases}
	(1.241,\,0,\,2.753,\,1.4,\,0), \quad & x<0,\\
	(0.991,\,0,\,3.059\times10^{-4},\,5.5,\,1.505), \quad & x>0.
\end{cases}
\]
The computational domain is \([-5,5]\times[-0.125,0.125]\) with outflow boundaries, and the solution is advanced to \(T = 1\). Both the GPR-CC and Con-CC schemes are reinitialized
every time step. To probe the effect of geometric perturbations, at each Runge-Kutta stage we add
a tiny random perturbation of magnitude \(10^{-3}\,\Delta x\) to the level‐set function. Figure \ref{fig: t5-1}
shows the resulting profiles along \(y=0\). Whereas the Con-CC scheme exhibits pronounced spurious
oscillations under the perturbations, the GPR-CC scheme remains unaffected, confirming its GPR property.

\begin{figure}[!htb]
	\centering
	\subfloat[Density $\rho$]{\label{sfig: t5-1}\includegraphics[height=0.165\textheight]{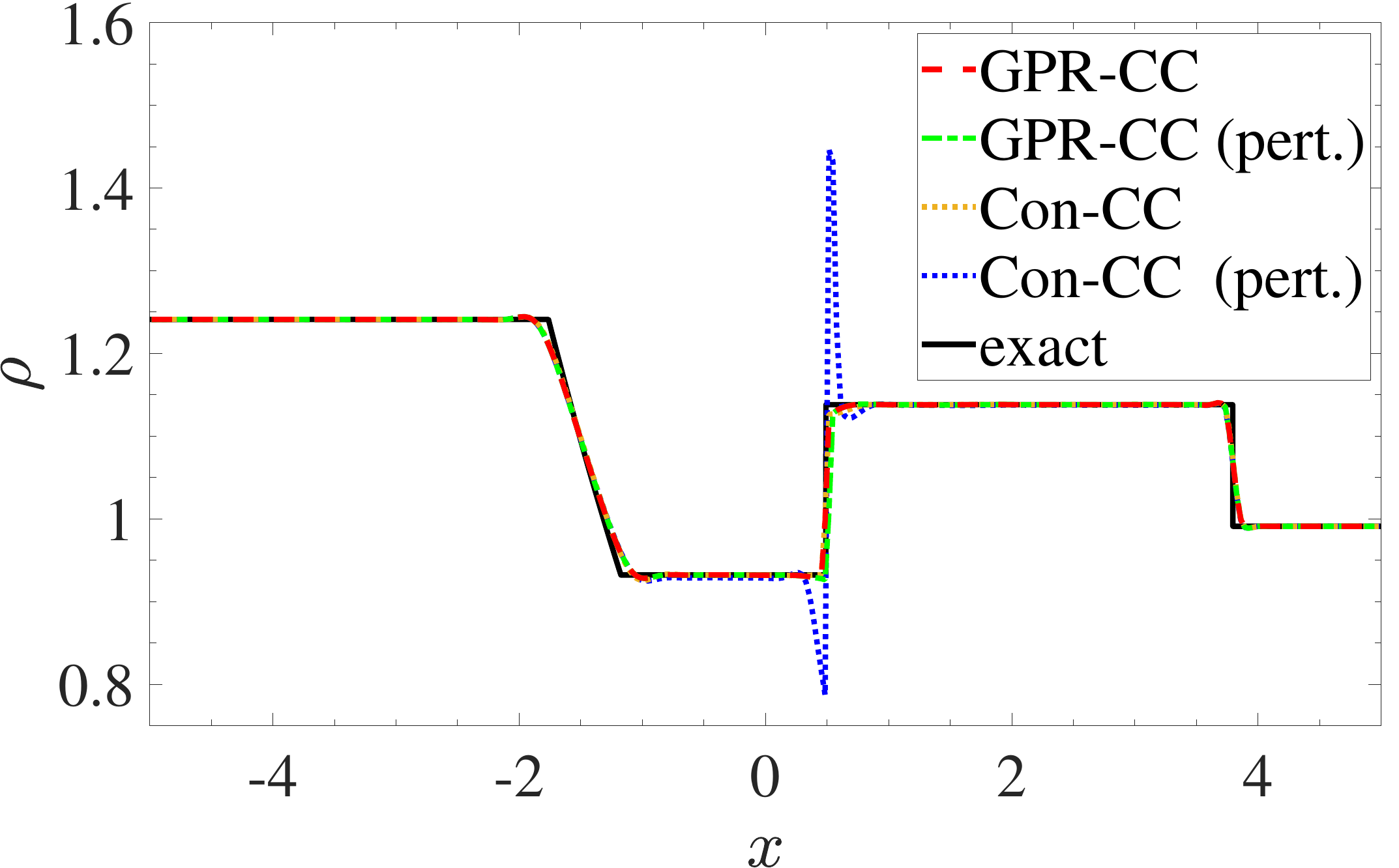}}
	\hspace{0.05\textwidth}
	\subfloat[Density $\rho$ (zoomed‐in)]{\label{sfig: t5-2}\includegraphics[height=0.160\textheight]{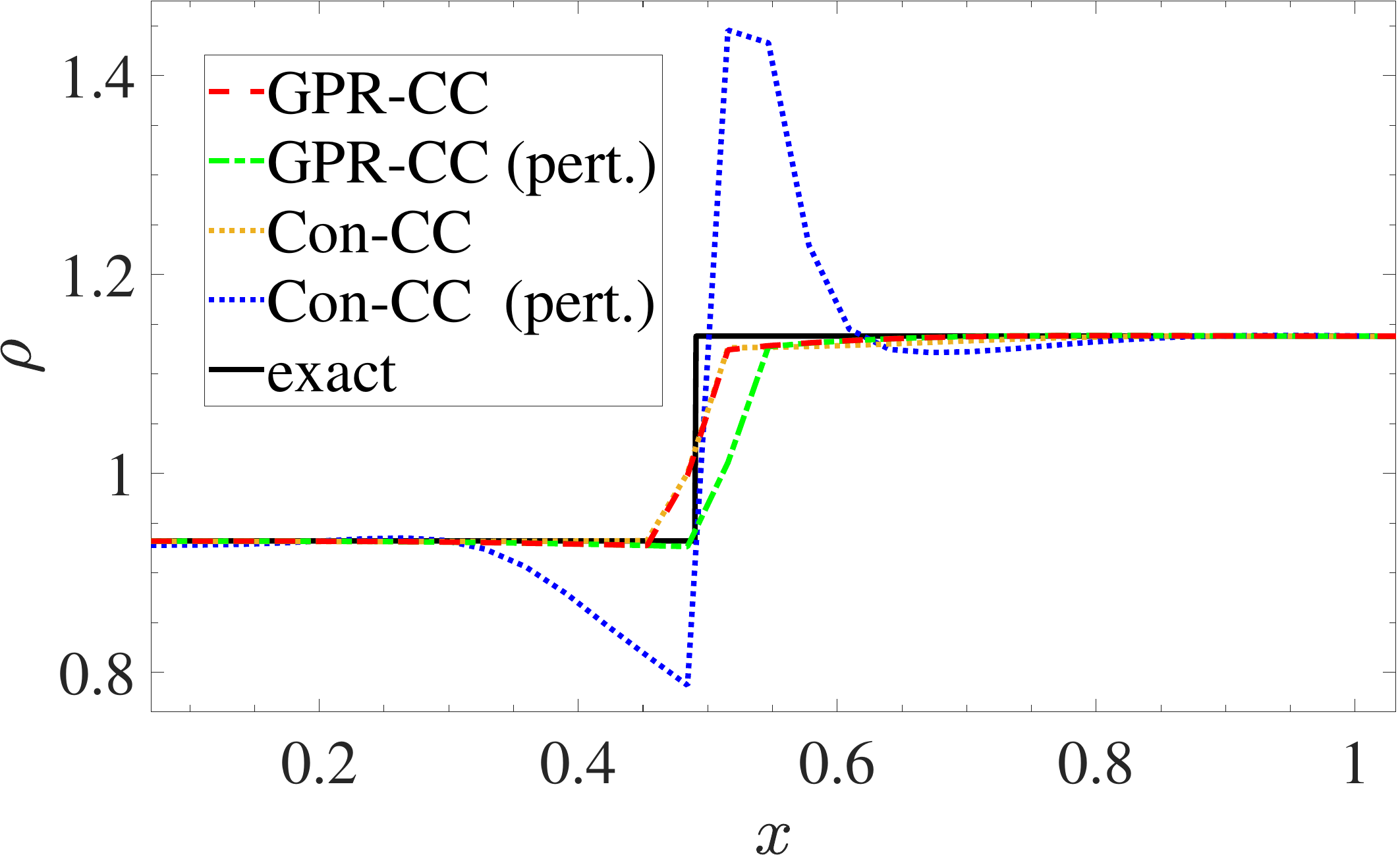}}
	
	\subfloat[Pressure $p$]{\label{sfig: t5-3}\includegraphics[height=0.165\textheight]{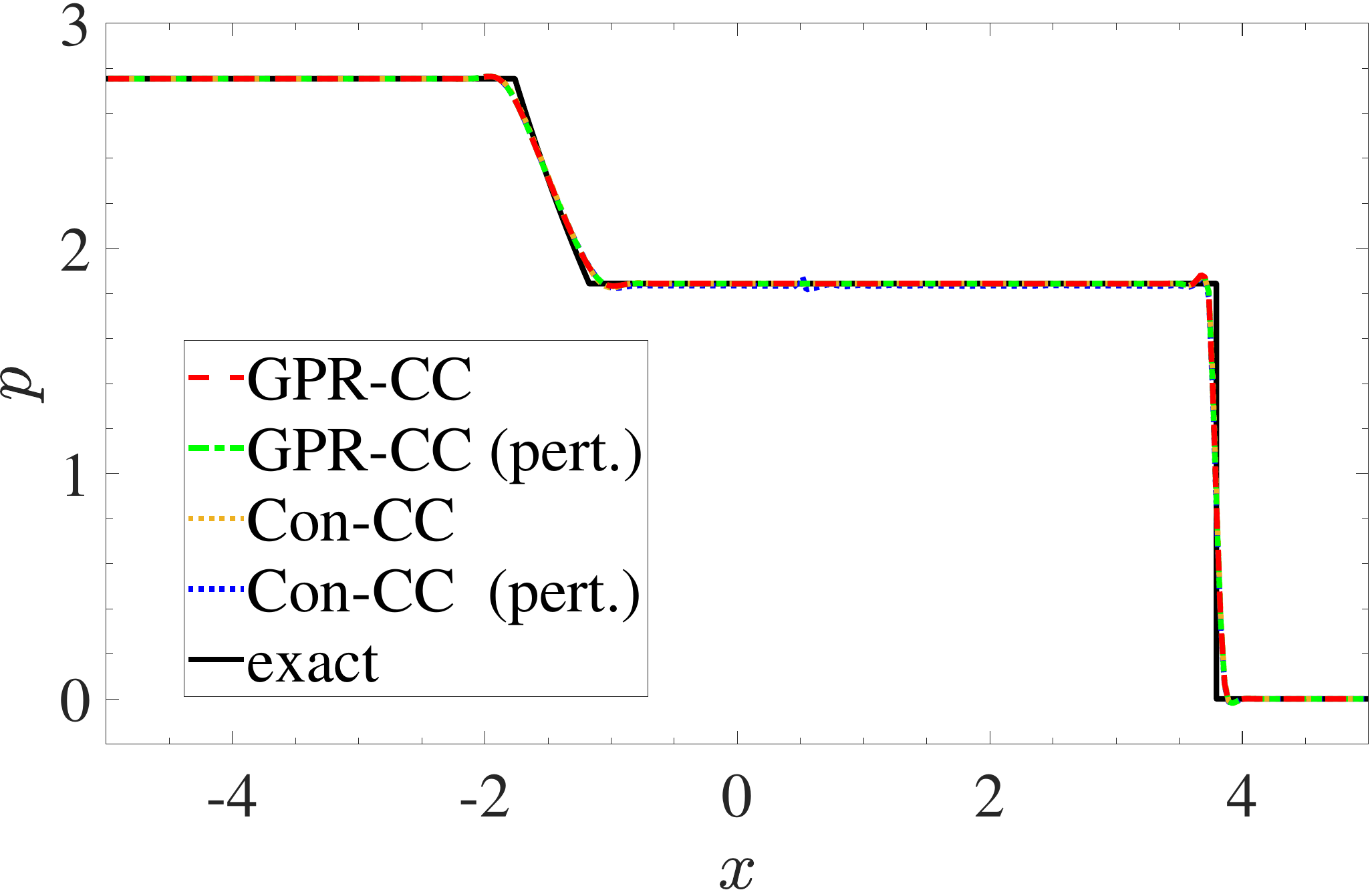}}
	\hspace{0.05\textwidth}
	\subfloat[Pressure $p$ (zoomed‐in)]{\label{sfig: t5-4}\includegraphics[height=0.162\textheight]{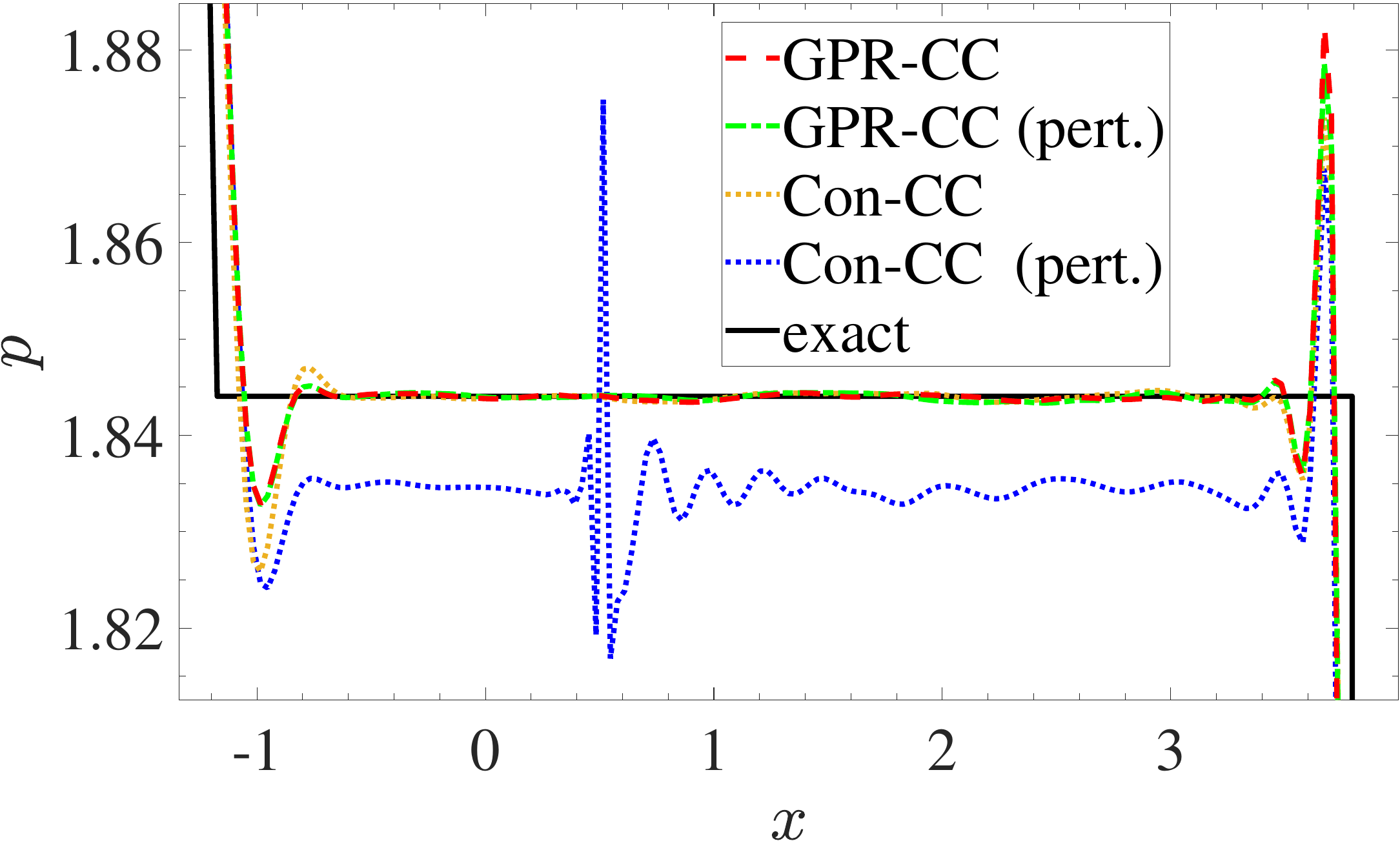}}
	\captionsetup{font=small}
	\caption{Gas-water Riemann problem (Problem~\ref{sec:GWRP}). Cross‐sectional profiles of density and pressure at \(y=0\) for the GPR-CC and Con-CC schemes and with their level-set-perturbed variants (GPR-CC (pert.) and Con-CC (pert.)), where at each Runge--Kutta stage the level‐set function is randomly perturbed by an amplitude of \(10^{-3}\,\Delta x\). $N_x\times N_y=300\times8$.}
	\label{fig: t5-1}
\end{figure}

\subsection{Gas-gas shock bubble problem}
\label{test6}
We consider the two‐dimensional Richtmyer--Meshkov instability benchmark in which a planar Mach\,1.22 air shock interacts with a unit‐radius cylindrical helium bubble. The computational domain is $\Omega = [-3,4]\times[-3,3],$
initially filled with quiescent air ($\rho=1$, $p=1$, $\gamma=1.4$, $B=0$).  A stationary helium bubble ($\rho=0.138$, $p=1$, $\gamma=5/3$, $B=0$) is located by $\phi_0(x,y)=\sqrt{x^2+y^2}-1<0$. To generate the Mach\,1.22 shock, we prescribe post‐shock air conditions for $x<-1.2$ with
$(\rho,v_x,v_y,p,\gamma,B) = (1.3764,\;0.394,\;0,\;1.5698,\;1.4,\;0),$
while retaining the pre‐shock state for $x>-1.2$. Non‐reflecting boundary conditions are imposed at $y=\pm3$, with inflow condition at $x=-3$ and outflow condition at $x=4$. 

Figure \ref{fig: SGGI-3} presents the interface morphology at $t=4$. The GPR-CC scheme maintains a line‐thin, sharply resolved interface, whereas DG-DIM exhibits noticeable interface smearing. These results demonstrate that the proposed GPR-CC framework delivers substantially higher resolution in the immediate vicinity of the material interface.


\begin{figure}[!htb]
	\centering
	\subfloat[GPR-CC]{\label{sfig: SGGI-2}\includegraphics[width=0.35\textwidth]{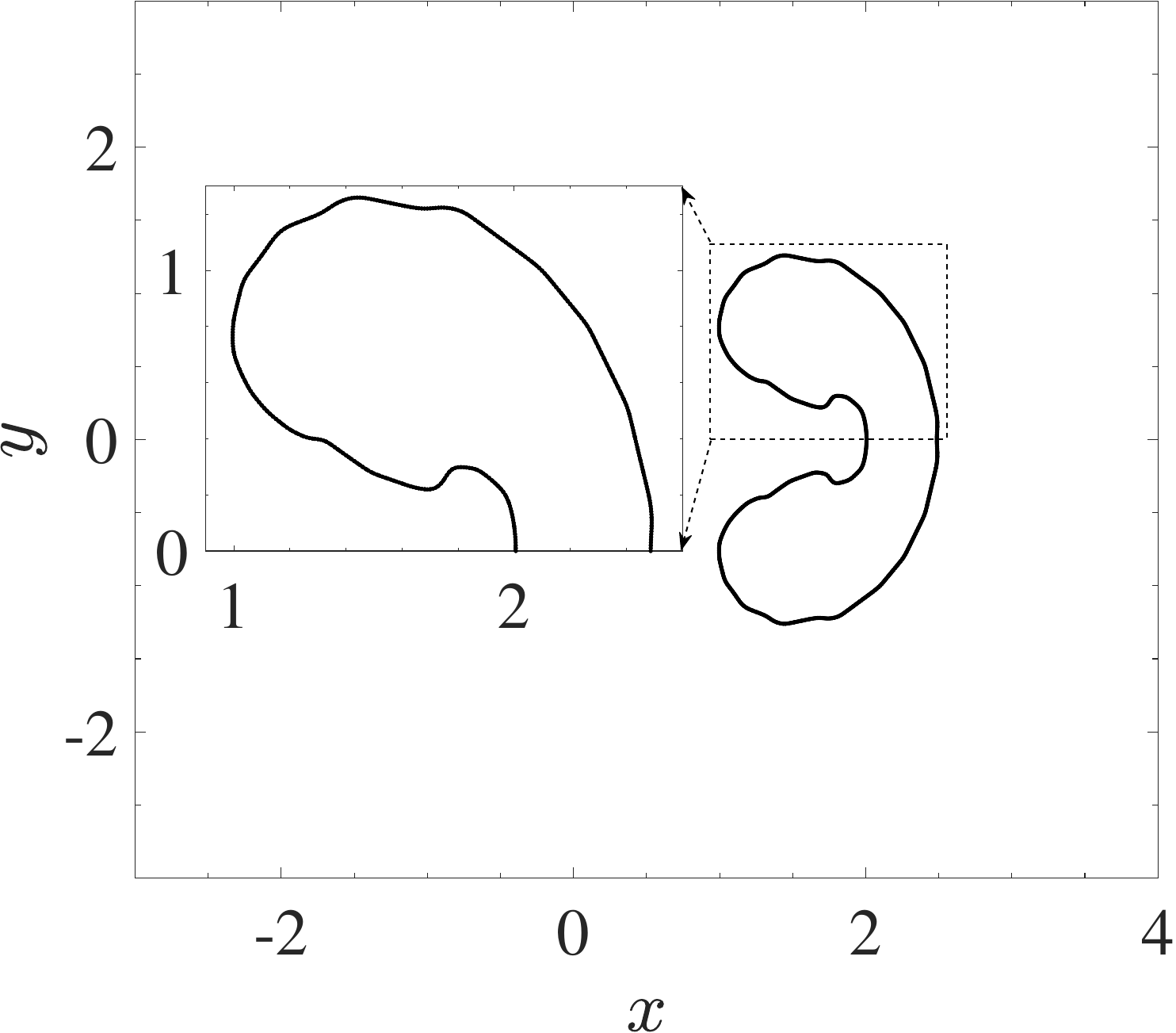}}
	\hspace{0.05\textwidth}
	\subfloat[DG-DIM]{\label{sfig: SGGI-8}\includegraphics[width=0.35\textwidth]{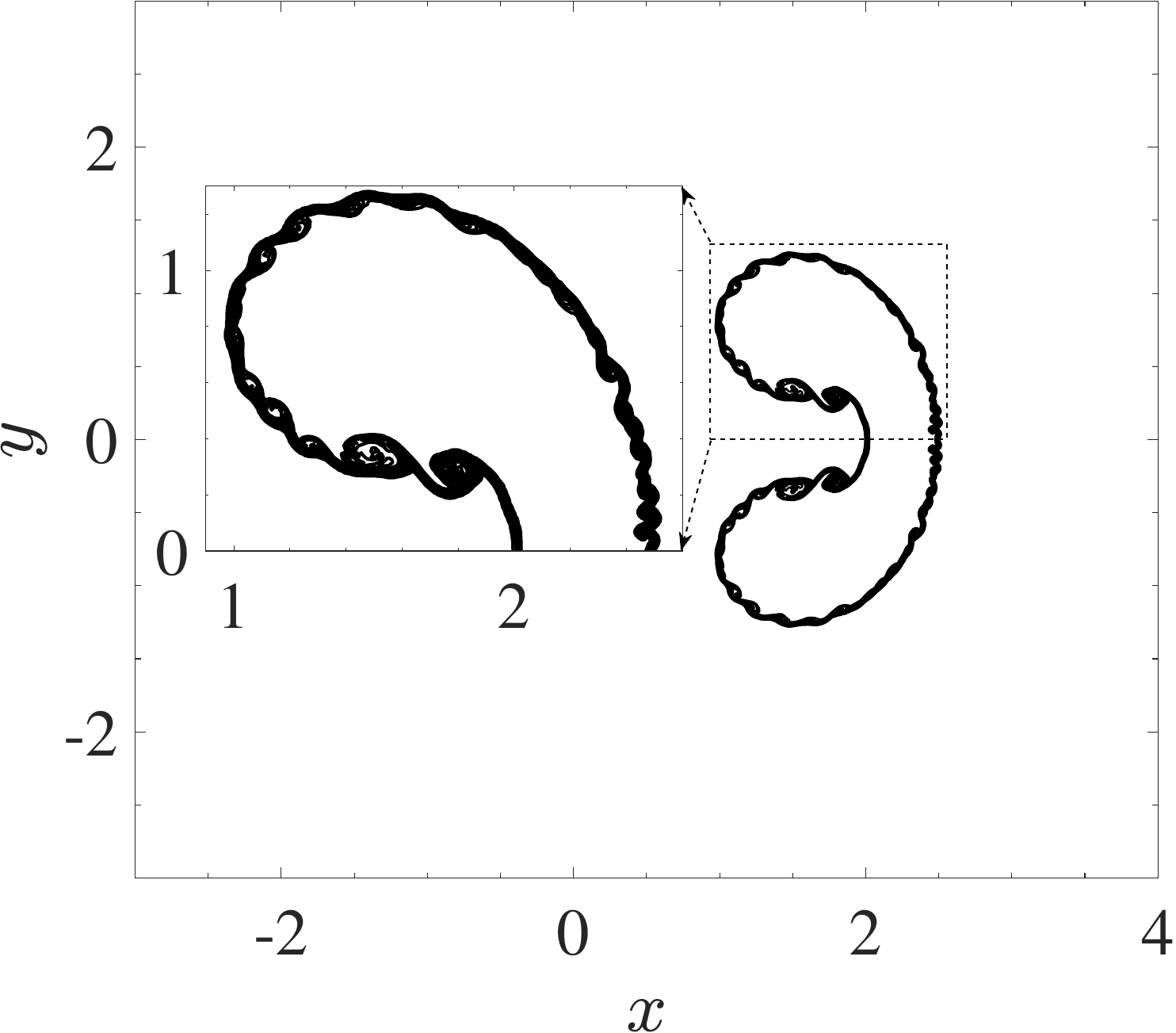}}
	\captionsetup{font=small}
	\caption{Gas-gas shock bubble problem (Problem \ref{test7}). The interface positions at \(T = 4\) computed by the GPR-CC and DG-DIM schemes with \(N_x \times N_y = 700\times600\). In GPR-CC the interface is explicitly captured via a level-set method, whereas in DG-DIM it is delineated by twenty volume‐fraction contours ranging from 0.001 to 0.999.}
	\label{fig: SGGI-3}
\end{figure}

\subsection{Water-gas shock bubble problem}\label{test7}
We next consider a two‐dimensional water-gas shock bubble problem, in which a planar Mach\,1.65 water shock impacts on a unit‐radius cylindrical air bubble. The computational domain is $\Omega = [-4,3]\times[-3,3],$ initial level set function $\phi_0(x,y) = \sqrt{x^2 + y^2} - 1,$ with \(\phi<0\) denoting the air (($\gamma,B$)$=(1.4,0)$) and \(\phi>0\) the water (($\gamma,B$)$=(7.15,3309)$).
The non‐dimensional initial states are
\[
(\rho,u,v,p) =
\begin{cases}
	(1000,0,0,1), \quad & \text{pre‐shocked water (}x<-1.2\text{)},\\
	(1176.3576,1.1692,0,9120), \quad & \text{post‐shocked water (}x>-1.2\text{)},\\
	(1,0,0,1), \quad & \text{air bubble}.
\end{cases}
\]
Non-reflective open boundary conditions are applied at \(y=\pm3\), an inflow boundary at \(x=-4\) prescribes the post‐shock water state, and a non‐reflecting outflow boundary is enforced at \(x=3\). Density contours are recorded at \(t=0.471,0.51\) to monitor interface deformation, jet penetration, and vortex roll‐up.

Figure \ref{fig: SWGI-4} presents the evolution of density contours obtained using the GPR-CC and DG‐DIM schemes, while Figure \ref{fig: SWGI-1} compares the bubble interfaces computed by the GPR‐CC and DG-DIM schemes. Throughout the simulation, the GPR‐CC scheme maintains a line‐thin, sharply resolved interface, whereas DG‐DIM exhibits noticeable smearing---demonstrating the superior material‐interface resolution afforded by the GPR‐CC framework.


\begin{figure}[!htb]
	\centering
	\subfloat[$t = 0.471$, GPR-CC]{\label{sfig: SWGI-6}\includegraphics[width=0.40\textwidth]{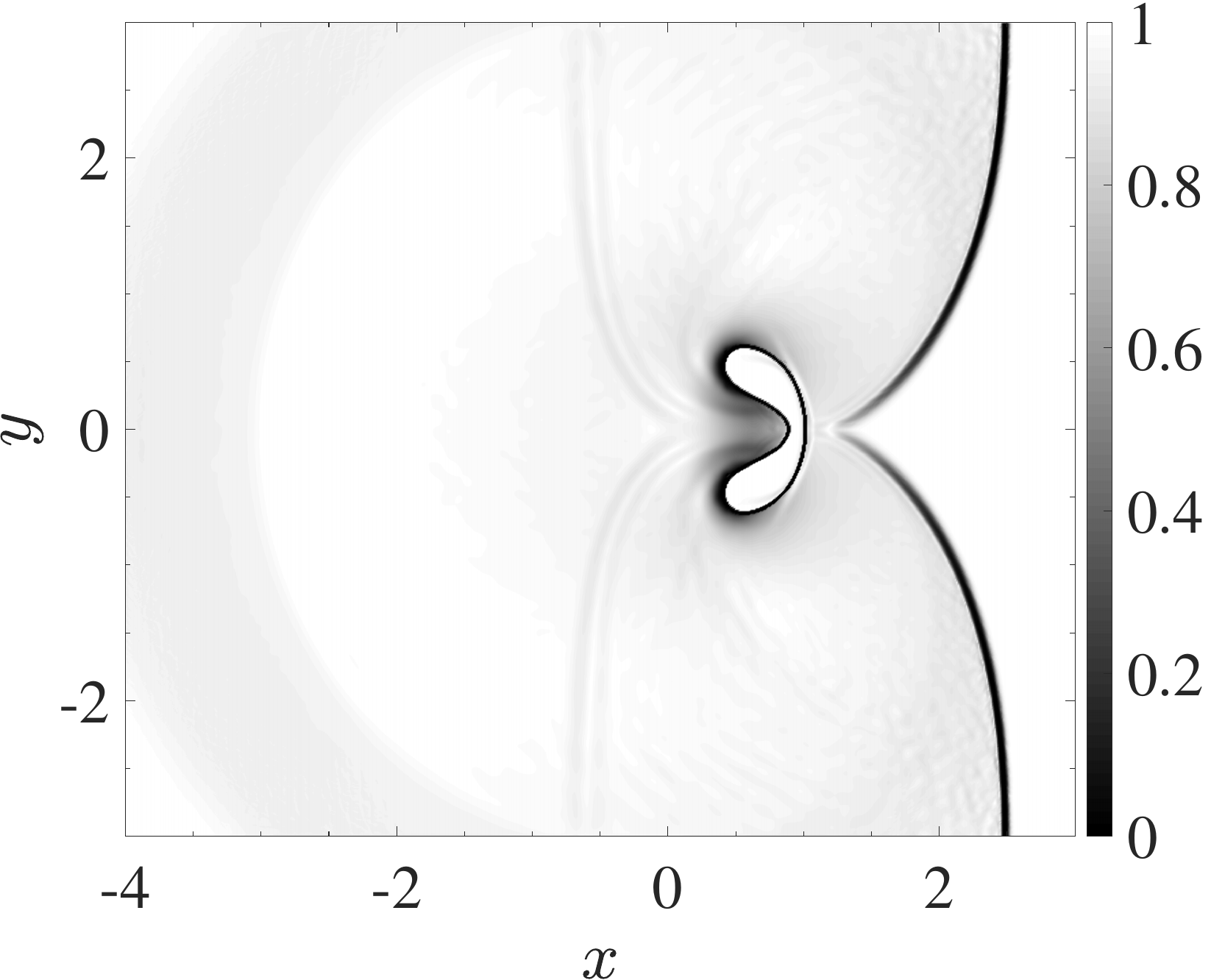}}
	\hspace{0.05\textwidth}
	\subfloat[$t = 0.51$, GPR-CC]{\label{sfig: SWGI-7}\includegraphics[width=0.40\textwidth]{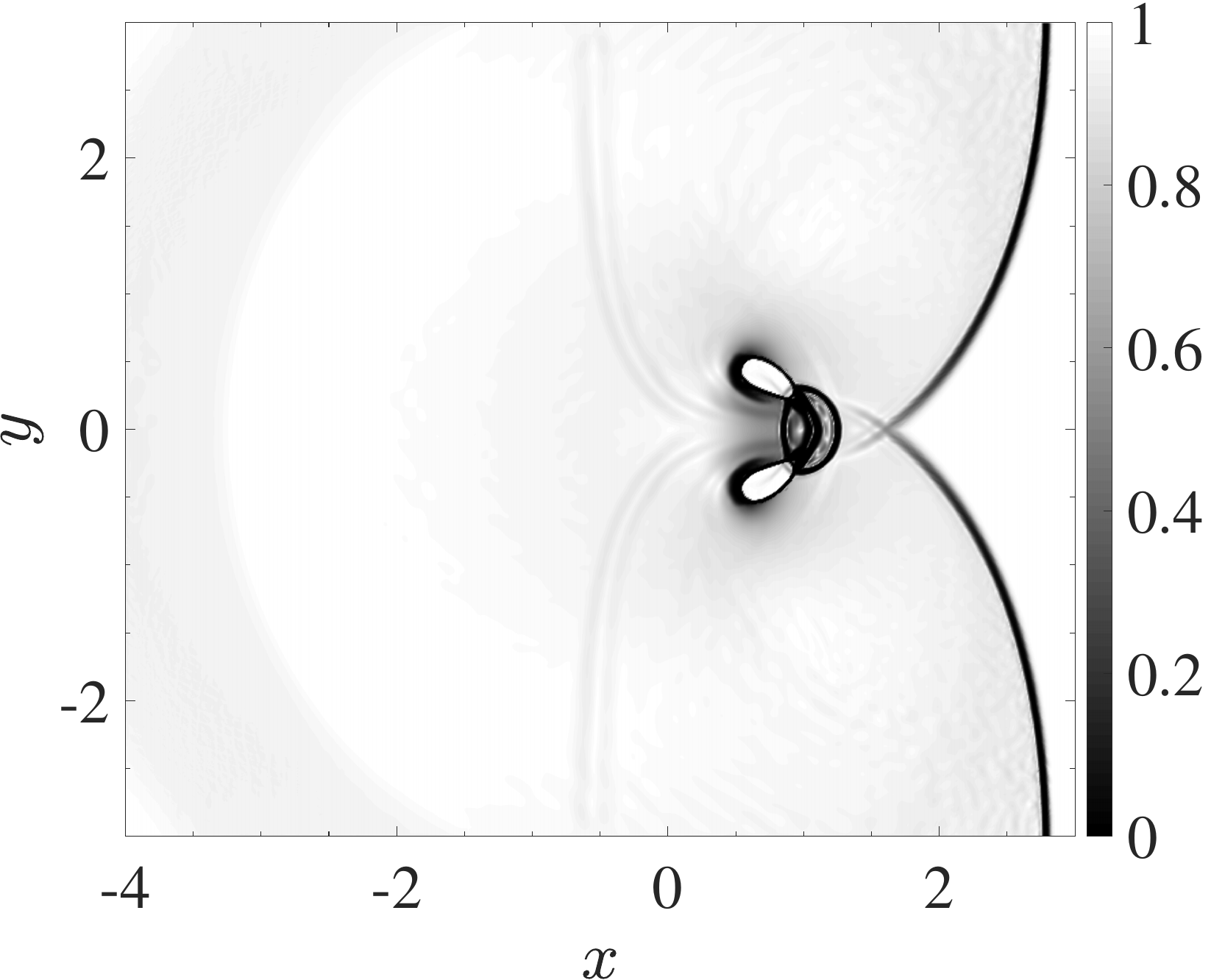}}
	
	\subfloat[$t = 0.471$, DG-DIM]{\label{sfig: SWGI-13}\includegraphics[width=0.40\textwidth]{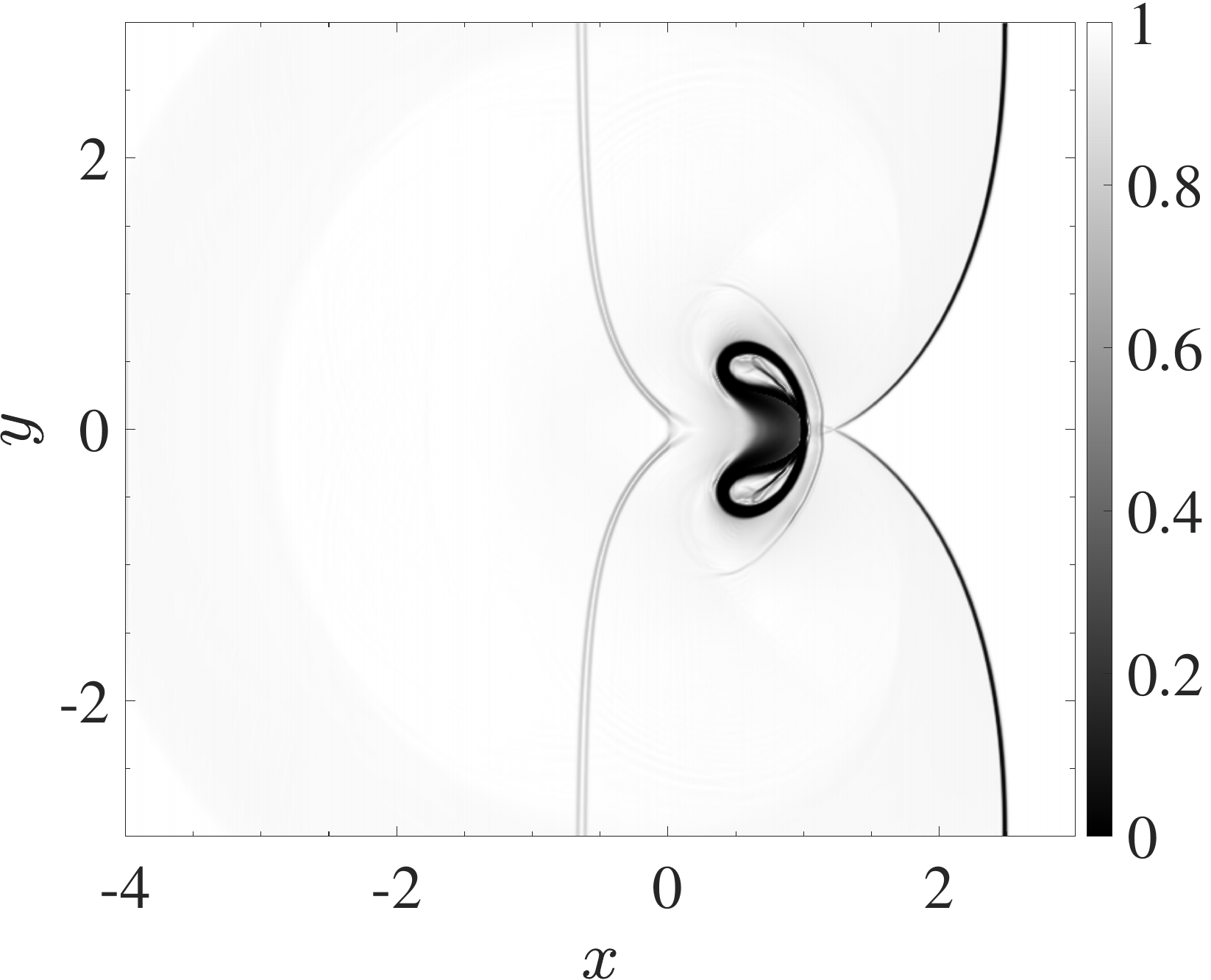}}
	\hspace{0.05\textwidth}
	\subfloat[$t = 0.51$, DG-DIM]{\label{sfig: SWGI-14}\includegraphics[width=0.40\textwidth]{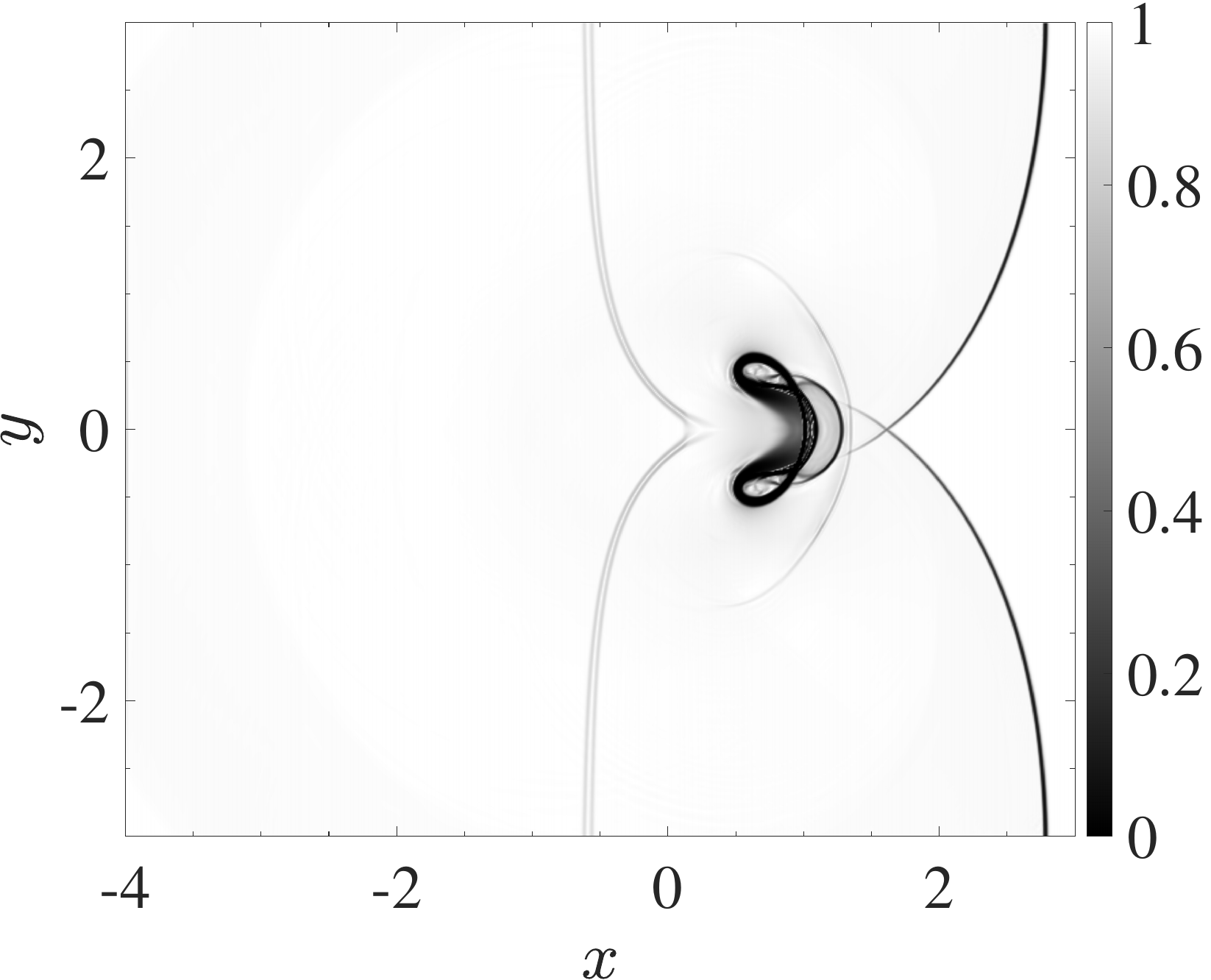}}
	\captionsetup{font=small}
	\caption{Water-gas shock bubble problem (Problem~\ref{test7}). Evolution of density contours obtained by the GPR-CC and DG-DIM schemes on a \(700\times600\) grid.}
	\label{fig: SWGI-4}
\end{figure}

\begin{figure}[!htb]
	\centering
	\subfloat[DG-DIM]{\label{sfig: SWGI-9}\includegraphics[width=0.35\textwidth]{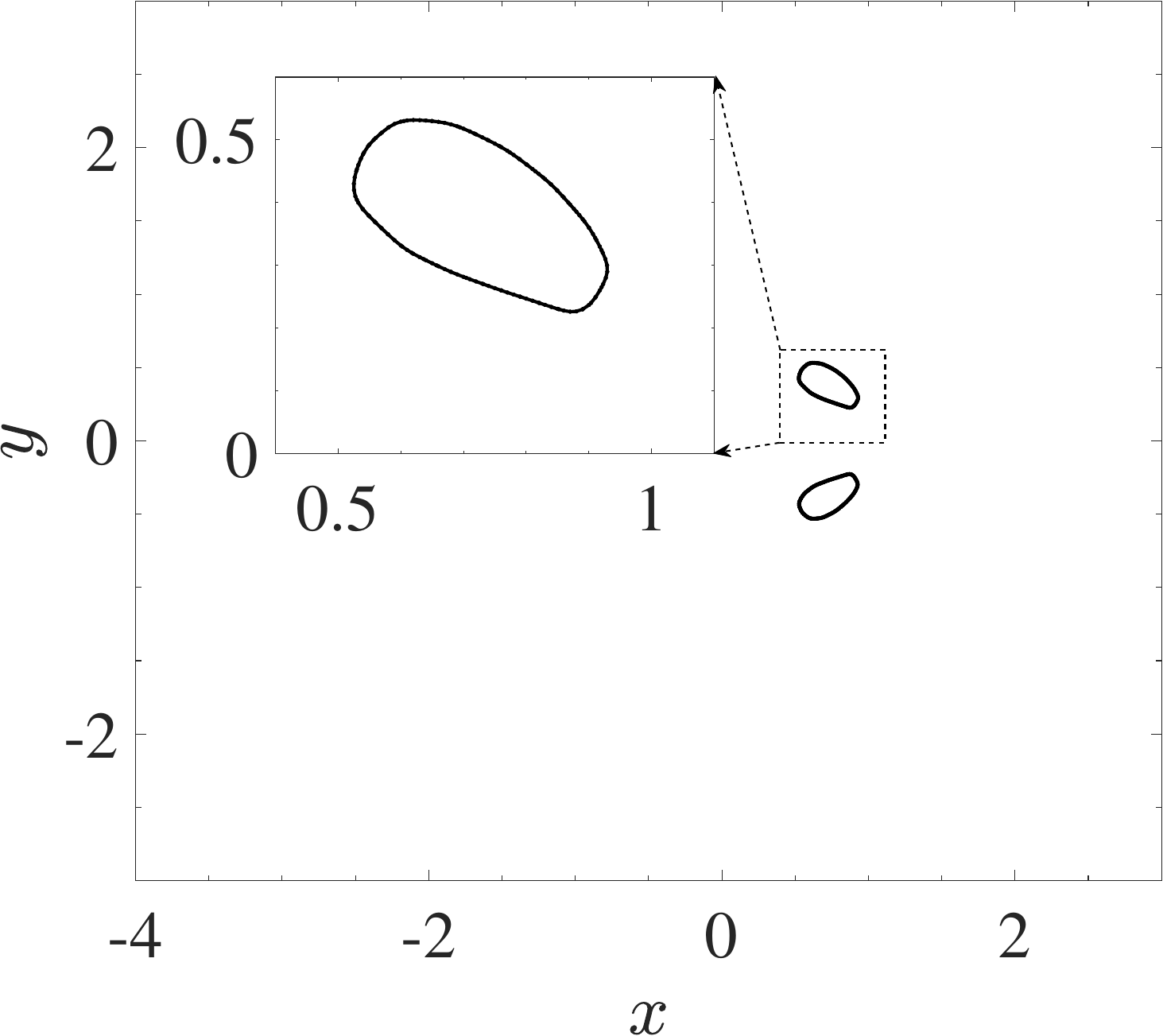}}
	\hspace{0.05\textwidth}
	\subfloat[GPR-CC]{\label{sfig: SWGI-2}\includegraphics[width=0.35\textwidth]{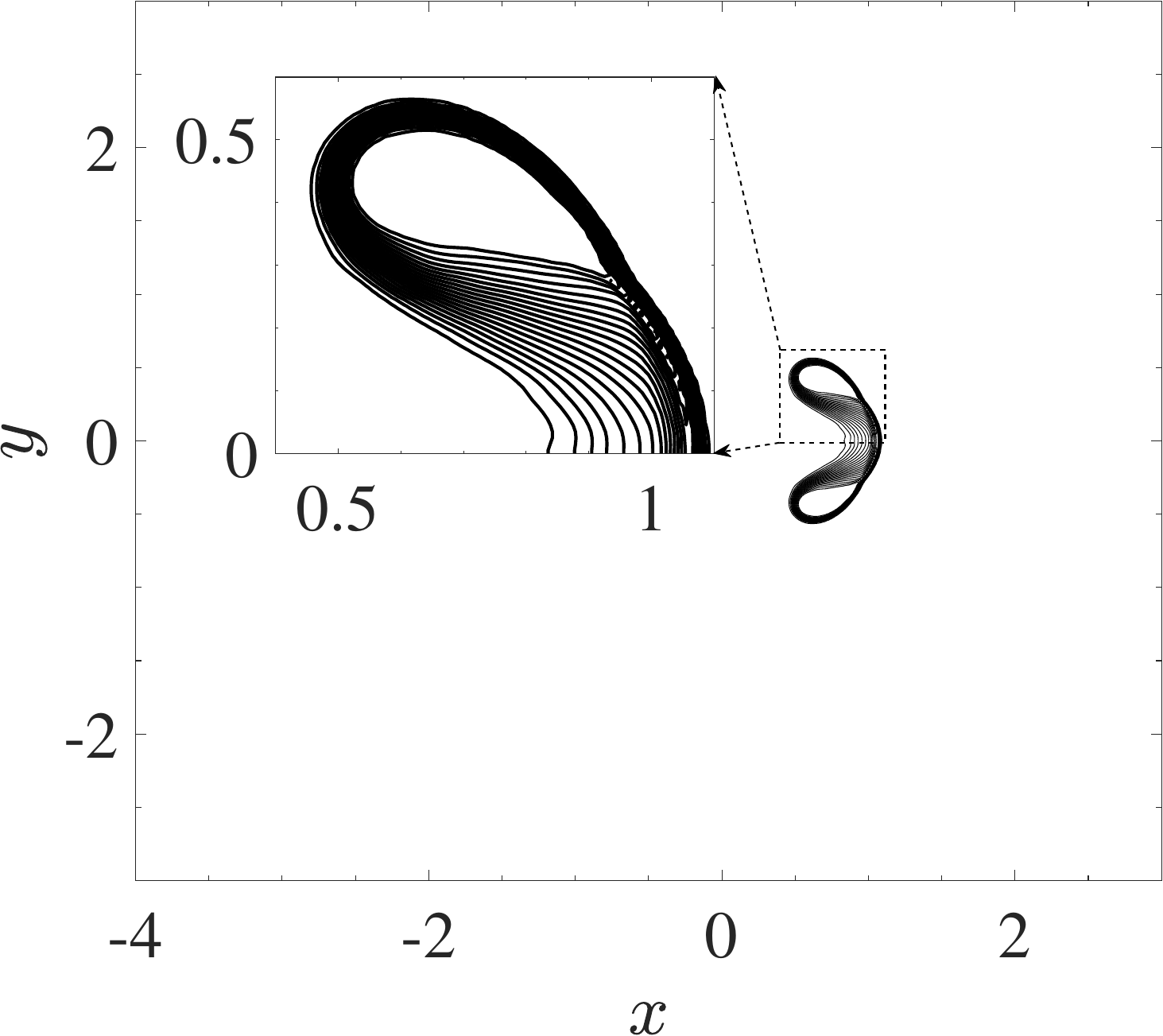}}
	\captionsetup{font=small}
	\caption{Water-gas shock bubble problem (Problem \ref{test7}). The interface positions at \(T = 0.51\) computed by the GPR-CC and DG-DIM schemes with \(N_x \times N_y = 700\times600\). In GPR-CC the interface is explicitly captured via a level-set method, whereas in DG-DIM it is delineated by twenty volume‐fraction contours ranging from 0.001 to 0.999.}
	\label{fig: SWGI-1}
\end{figure}

\subsection{Underwater explosion}\label{test8}

In the final problem, we consider an underwater explosion problem, in which a highly compressed cylindrical air bubble is placed in quiescent water beneath a free surface. This configuration, previously studied in \cite{DENG2018945,cheng2020quasi}, presents a challenge for interface-capturing methods by undergoing a topological change.

The computational domain is $\Omega = [-2,2]\times[-1.5,1.5],$
and the initial interfaces are defined by the level-set function $\phi(x,y)
=\min\bigl(\sqrt{x^2+(y+0.3)^2}-0.12,\;-y\bigr),$
with \(\phi<0\) denoting the gas regions (submerged bubble and above-water air) and \(\phi>0\) the water. The non-dimensional initial states are
\[
(\rho,u,v,p,\gamma,B)=
\begin{cases}
	(1,\;0,\;0,\;1.01325,\;4.4,\;6000), \quad &\text{water } (\phi>0),\\
	(0.001225,\,0,\,0,\,1.01325,\,1.4,\,0), \quad &\text{atmospheric air }(y>0,~\phi<0),\\
	(1.25,\,0,\,0,\,10000,\,1.4,\,0), \quad &\text{underwater air bubble }(\text{else}).
\end{cases}
\]
A reflecting boundary condition is applied at \(y=-1.5\), while non-reflecting conditions are imposed on the remaining boundaries. Density contours are recorded at times \(t=0.0083,0.0128,0.019\) to capture the bubble's violent expansion.

Figure~\ref{fig: UWE-1} presents the evolution of the density fields computed with our GPR-CC scheme. The GPR-CC method accurately and stably captures the interface dynamics throughout the computation, demonstrating the robustness of the GPR-CC framework under extreme underwater explosion conditions.

\begin{figure}[!htb]
	\centering
	\subfloat[$t = 0.0083$]{\label{sfig: UWE-4}\includegraphics[width=0.34\textwidth]{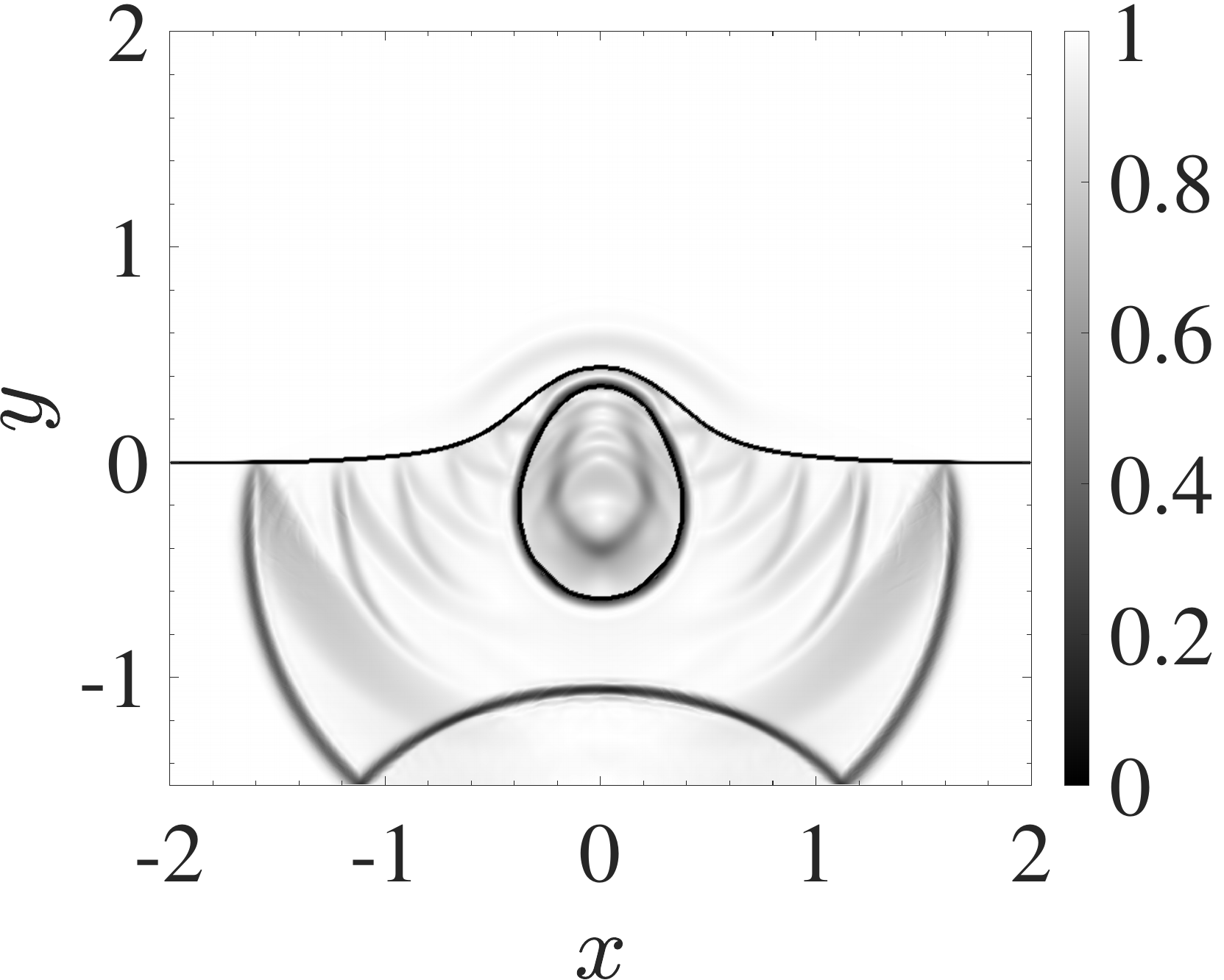}}
	\subfloat[$t = 0.0128$]{\label{sfig: UWE-5}\includegraphics[width=0.34\textwidth]{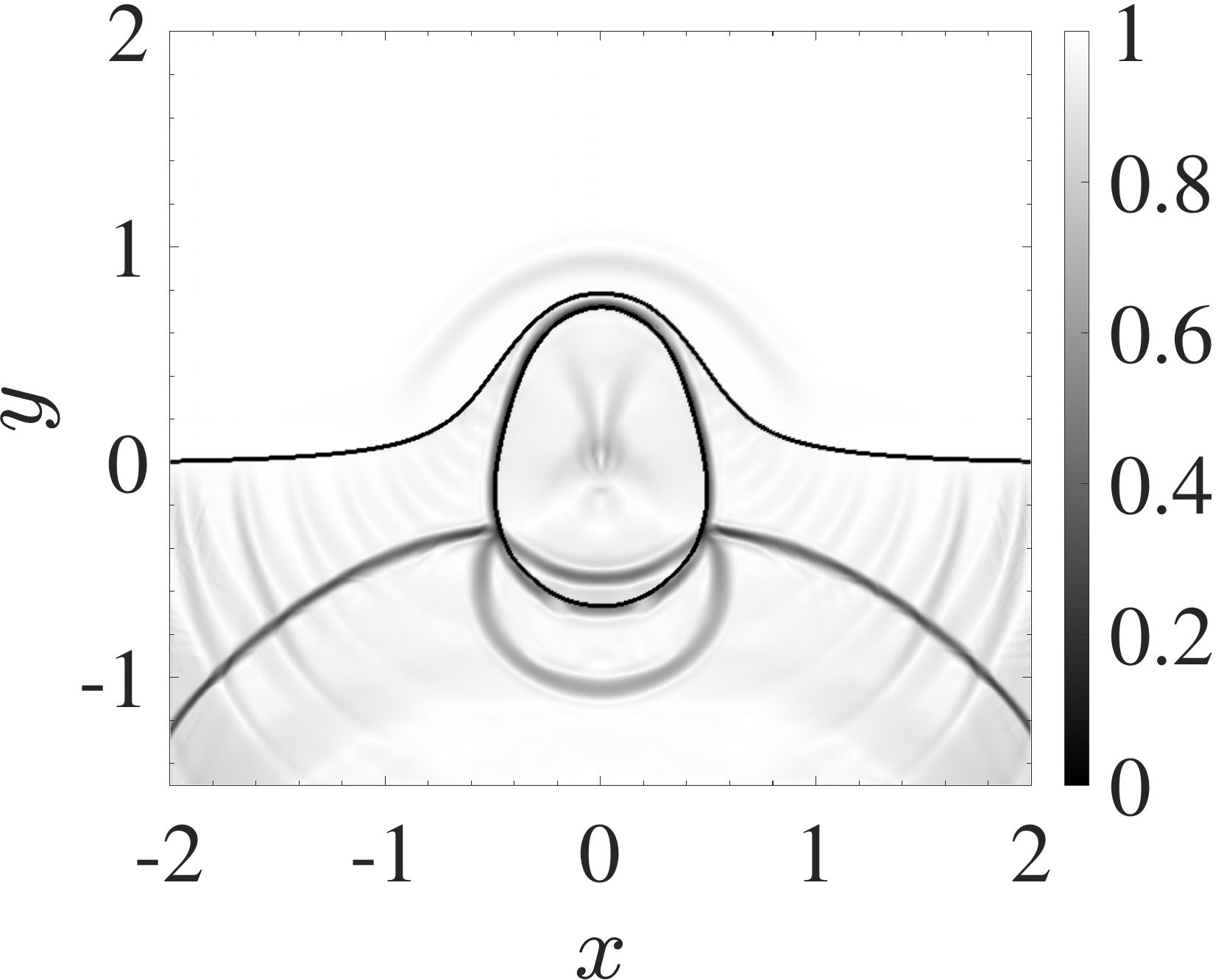}}
	\subfloat[$t = 0.019$]{\label{sfig: UWE-6}\includegraphics[width=0.34\textwidth]{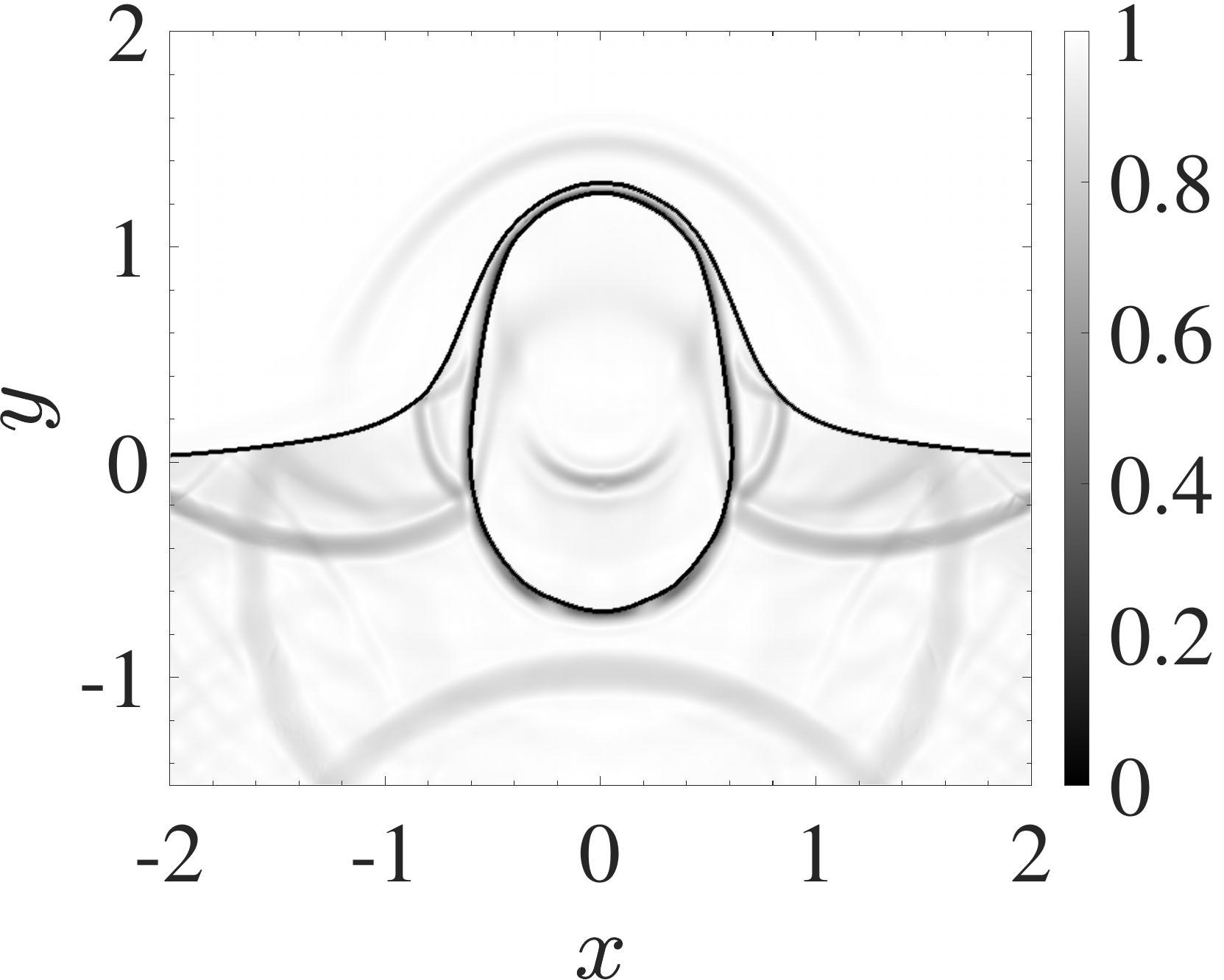}}
	\captionsetup{font=small}
	\caption{Schlieren images of density for the underwater explosion problem (Problem~\ref{test8}), computed with the GPR-CC scheme on a \(600\times450\) grid.}
	\label{fig: UWE-1}
\end{figure}


\section{Conclusions}\label{sec:conclusion}
This paper has developed a \emph{geometric-perturbation-robust} (GPR) moving-interface cut-cell framework that overcomes three challenges in compressible two-material flow simulation: exact preservation of pressure equilibrium, robustness to interface-location errors, and genuinely high-order accuracy. The central innovation is the \emph{evolved geometric-moment} formulation: by advancing every geometric moment through an auxiliary transport equation discretized consistently with the flow solver, we extend the classical geometric conservation law from cell volumes to \emph{all} higher-order geometric moments. We rigorously prove that, for the pure advection equation $u_t = 0,$ if the same linear spatial discretization operator is applied to both the evolved geometric moments and $u$, polynomials of \textbf{arbitrary degree} are transported exactly. This strict synchronization between geometry and conserved variables eliminates the accuracy degradation and instability that arise from geometric inconsistencies on deforming meshes. To prevent reconstruction from destroying pressure equilibrium, we introduced the concept of an \emph{equilibrium-compatible} (EC) reconstruction and showed how a minimal, fully local modification equips standard WENO procedures with this property. We detail a third-order EC multi-resolution WENO (EC-MRWENO) variant and employ it both for cell-interface flux calculations and for the redistribution procedure following each Runge--Kutta step. The tight coupling of EC-MRWENO with evolved geometric moments yields, to our knowledge, the first cut-cell solver that is simultaneously
\begin{itemize}
	\item \textbf{GPR}: it preserves machine-precision pressure equilibrium even under small interface perturbations;
	\item \textbf{Genuinely high order}: it achieves second order accuracy at the interface (even across a pure contact discontinuity), while preserving third-order accuracy elsewhere;
	\item \textbf{Topology-change capability:} it remains stable under severe interface deformation and breakup.
\end{itemize}
Extensive two-dimensional test cases validate these properties and demonstrate significant accuracy gains over conservative schemes. Because both the auxiliary geometric-moment equations and the EC modification are entirely local, the proposed techniques are non-intrusive: they can be grafted onto a broad class of cut-cell and general moving-mesh solvers, and are not limited to multi-material simulations.  
While the present implementation attains second-order accuracy at straight interfaces, the framework is fully compatible with curved interface and higher-order EC reconstructions; our preliminary results for genuinely higher-order GPR cut-cell schemes are promising and will be reported separately. The ideas of GPR, evolved geometric moments, and EC reconstruction will prove valuable well beyond the context considered here, enabling the design of stable, consistent, high-order algorithms on dynamically evolving meshes across a wide range of scientific-computing applications.

\section*{Acknowledgments}
We thank Lecturer Feng Zheng and Ph.D.\ candidates Qiqin Cheng and Linhao Zhu for their valuable advice on this work. Chaoyi Cai also gratefully acknowledges Professor Chang Shu for his invaluable assistance during Cai's visit to the National University of Singapore

\appendix
\begin{changemargin}{-0.2cm}{-0.2cm}
	\renewcommand\baselinestretch{0.86}
	\bibliographystyle{siamplain}
	\bibliography{ref}
\end{changemargin}
\newpage 
\end{document}